\documentclass[a4paper,twoside,10pt]{amsart}

        \usepackage{color}

\usepackage{amsthm,amsmath,amssymb,amscd}
\usepackage{graphicx}
\usepackage{epic}

\newtheorem{theorem}{Theorem}[section]
\newtheorem{proposition}[theorem]{Proposition}
\newtheorem{lemma}[theorem]{Lemma}
\newtheorem{corollary}[theorem]{Corollary}

\newtheorem{claim}[theorem]{Claim}

\theoremstyle{definition}
\newtheorem{definition}[theorem]{Definition}
\newtheorem{notation}[theorem]{Notation}

\newtheorem{example}[theorem]{Example}
\theoremstyle{remark}
\newtheorem{remark}[theorem]{Remark}

\def\Par{{\par  \noindent}}
\def\A{{\mathbf A}}

\def\P{{\mathbf P}}

\def\Q{{\mathbf Q}}

\def\R{{\mathbf R}}

\def\C{{\mathbf C}}
\def\L{{\mathcal L}}
\def\N{{\mathbf N}}
\def\Newton{{\mathcal N}}

\def\Z{{\mathbf Z}}

\def\G{{\Gamma}}

\def\a{\alpha}
\def\A{\mathbf A}

\def\b{\beta}
\def\C{\mathbf C}
\def\cprime{$'$}
\def\codim{\mathrm{codim}}

\def\d{\delta}

\def\g{\gamma}

\def\limproj{\mathop{\oalign{lim\cr\hidewidth$\longleftarrow$\hidewidth\cr}}}

\def\J{{\mathcal{J}}}
\def\l{\lambda}

\def\N{\mathbf N}
\def\orb{\mathrm {orb}}
\def\int{\mathrm{int}}
\def\ord{\mathrm {ord}}

\def\p{\pi}

\def\P{{\mathbf P}}

\def\Q{\mathbf Q}
\def\R{\mathbf R}
\def\r{\rho}

\def\s{\sigma}
\def\t{\tau}

\def\w{\omega}

\def\y{\wedge}
\def\Z{\mathbf Z}
\def\z{\zeta}

\def\limproj{\mathop{\oalign{lim\cr\hidewidth$\longleftarrow$\hidewidth\cr}}}
\def\elem(#1,#2){  \{ \frac{#1}{\overline{\ #2\ }}\}  }

\title{TORIC GEOMETRY AND THE SEMPLE-NASH MODIFICATION}
\author{Pedro D. Gonz\'alez P\'erez}

\address{Instituto de Ciencias Matem\'aticas-CSIC-UAM-UC3M-UCM. Dpto. de Algebra. Facultad de Ciencias Matem\'aticas.
Universidad Complutense de Madrid.
Plaza de las Ciencias 3. 28040. Madrid. Spain.}

\email{pgonzalez@mat.ucm.es}
\thanks{Gonz\'alez P\'erez is supported by Severo Ochoa project SEV-2011-0087 and
MTM2010-21740-C02-01 grants. Teissier is partially supported by grant MTM2007-64704}

\author{Bernard Teissier}
\address{
Institut Math\'{e}matique de Jussieu, UMR 7586 du CNRS, 175 Rue du Chevaleret, 75013 Paris, France}

\email{teissier@math.jussieu.fr}



\begin{document}

\maketitle

\begin{center} 
\emph{To Heisuke Hironaka on the occasion of his 80th birthday}
\end{center}

\begin{abstract}
This paper proposes some material towards a theory of general
toric varieties without the assumption of normality. Their
combinatorial description involves a fan to which is attached a
set of semigroups subjected to gluing-up conditions. In particular
it contains a combinatorial construction of the blowing up of a
sheaf of monomial ideals on a toric variety. In the second part
this is used to show that iterating the Semple-Nash modification
or its characteristic-free avatar provides a local uniformization of
 any monomial valuation of maximal rank dominating
a point of a toric variety.
\end{abstract}

\section* {Introduction}
In the first part of this paper we study  abstract  toric
varieties without the assumption of normality. Since Sumihiro's
Theorem on the existence of a covering of a toric variety by
invariant affine varieties fails without the assumption of
normality, we have to set the existence of such a covering as part
of the definition of a toric variety. Then an abstract toric
variety has a combinatorial description: it corresponds to certain
semigroups in the convex duals of the cones of a fan, which
satisfy a natural gluing-up condition. This generalizes the
definition of \cite{GKZ} which concerns toric varieties
equivariantly embedded in projective space. In spirit it is also a
continuation of our previous work \cite{GP-T} on embedded
normalization and embedded toric resolution of singularities of
affine toric varieties. We can then define blowing-ups of sheaves
of monomial ideals as toric varieties, and describe the
corresponding operations on semigroups. We also provide the
combinatorial description of torus-invariant Cartier divisors on a
toric variety and the general versions of the classical criteria
for ampleness and very-ampleness.\par In the second part of the
paper we use the description of blowing-ups given in the first
part to show that one can
obtain a local uniformization of a monomial valuation of maximal rank dominating
a toric variety, by a finite number of iterations of  the
blowing-up of the logarithmic jacobian ideal introduced in
\cite{GS-ENS}. If the field is of characteristic zero, this
blowing-up is isomorphic to the Semple-Nash modification.
Recall that this is a canonical modification of a reduced equidimensional space which replaces each point by
the set of limit positions of tangent spaces at nearby non-singular points. See the second part for details.

\par  Our result is related in the special case of
toric varieties to an interesting question apparently first asked by
Semple in \cite{S}, that is,  if
the iteration of the Semple-Nash modification
eventually resolves the singularities of an algebraic variety
defined over a field of characteristic
zero. One can present the relation of this conjectural
resolution process with \textit{classical} resolution of singularities as
follows: two types of proper birational correspondences naturally
associate a singular variety to a non-singular one: proper
birational projections of an embedded non-singular algebraic
variety to a smaller dimensional ambient space, and the taking of
the \textit{envelope} of a family of linear subspaces (of an
affine or projective space) whose parameter space is a non
singular algebraic variety. For example, a family of lines in the real plane parametrized by a circle, such as the Simson-Wallace lines of a triangle, parametrized by its circumscribed circle, have an envelope with singularities, generically with an odd number of cusps (see \cite{Thom}). In the Simson-Wallace lines case, it is a quartic with three cusps (see \cite{St}). The natural map from the circle to the envelope is a homeomorphism and a local isomorphism outside of the cusps. Another example is given by the discriminant hypersurface of a versal deformation of an isolated complex hypersurface singularity: it is the envelope of a family of complex hyperplanes parametrized by the critical locus, which is non-singular and is the normalization of the discriminant (see \cite{T-hunting}) and therefore a resolution of singularities. \par Hironaka's resolution shows that in
characteristic zero all singularities may be created by the first
process, and Semple-Nash resolution in general would show that at least all
singularities of projective varieties in characteristic zero may be created by iterating
a natural generalization of the second process, if we allow singular spaces as parameter
spaces.
Moreover  it would produce a canonical process for local uniformization of
valuations in characteristic zero, as our results here do  for maximal rank monomial valuations on
toric varieties  in a characteristic-free manner.

\par\medskip\noindent
{\bf Acknowledgements.} We are grateful to Monique Lejeune-Jalabert for
introducing us to the reference \cite{S}, to Ezra Miller for bringing to our attention the work of H.M. Thompson,
to Michael Thaddeus for detecting errors in previous versions of this work.\par
We also mention that while this paper was in preparation, Mr. Daniel Duarte provided in \cite{Du} a proof of a result similar to our main result, \emph{a priori} somewhat stronger and giving an effective bound on the number of steps required for desingularization, in the two-dimensional case. A preprint of Dima Grigoriev and Pierre Milman (see \cite{G-M}) also provided an approach to the Semple-Nash desingularization problem for toric varieties with explicit results in dimension 2. In particular it contains a version of our Lemma 12.1. 
We thank the referee for his careful work and helpful suggestions. \par\noindent
We also thank Patrice Philippon for his help, David Cox for his remarks
and the  \textit{Institut
Math\'ematique de Jussieu} and  the \textit{Dpto.~\'Algebra,
Universidad Complutense de Madrid} for their hospitality.
\par\medskip\noindent
\centerline{\textbf{Part I: Toric
varieties}}\par\medskip

\Par\noindent
The purpose of this part is to develop the combinatorial theory of toric varieties
without the assumption of normality. 
We refer to \cite{Da},  \cite{Fulton}, \cite{TE}, \cite{Oda-M}, and \cite{Oda} 
for background on normal toric varieties, and to the books of
Oda-Miyake (\cite{Oda-M}), Gel{\cprime}fand, Kapranov, and  Zelevinsky (\cite{GKZ}) 
and Sturmfels (\cite{Sturmfels}) for certain classes of non necessarily normal toric varieties. 
We also point to previous work by H.M. Thompson towards the development of a general theory of 
toric varieties, see \cite{Thompson1}, \cite{Thompson2} and, from the perspective of Log Schemes,
\cite{Thompson3}.
This part is also connected with the more general theory of monoid schemes recently developped by
Corti\~nas \textit{et al.} in \cite{CHWW}.
We recommend \cite{Ewald} as a particularly accessible introduction to (normal)
toric varieties, and also the recent book  \cite{CLS} of D. Cox, J. Little and H. Schenk on the subject. \large\par\medskip

\begin{section}{Semigroups and semigroup algebras}\label{semi}
The theory of affine toric varieties over a field $k$ is the geometric version of the
theory of semigroup algebras over $k$. For part of the theory, one can omit the
assumption that the semigroup is finitely generated, and replace the field $k$
by a commutative ring.
\begin{definition} A (commutative) semigroup $\Gamma$ is a set equipped with
an operation  $+\colon\Gamma\times \Gamma\to \Gamma$ such that
$\epsilon_1+\epsilon_2=\epsilon_2+\epsilon_1$, which satisfies the associativity property and is cancellative ($\epsilon_1+\epsilon_2=\epsilon_1+\epsilon_3$ implies $\epsilon_2=\epsilon_3$). We shall assume
that $\Gamma$ contains a zero element $0$ such that
$\epsilon+0=\epsilon$.
We denote by $\Z \Gamma$ the group
generated by $\Gamma$ (defined in a similar way as the field of
fractions of an integral domain). To say that $\Gamma$ is cancellative means that the natural map of semigroups $\Gamma\to\Z\Gamma$ is injective. We say that $\Gamma$ is torsion free if the abelian group $\Z\Gamma$ is, which means that the only solution in $\N$ of an equation $m\gamma= m\gamma'$ with $\gamma,\gamma'\in\Gamma, \gamma\neq\gamma',$ is $m=0$. Since $\Gamma$ is cancellative, it implies that the only solution in $\Gamma$ of an equation $m\gamma =n\gamma$ with $m,n\in\N,\ m\neq n,$ is $\gamma=0$.\par\noindent 
 A system of generators of a semigroup is a
subset $(\gamma_i)$ of $\Gamma$ such that each element of $\Gamma$
is a (finite) linear combination of the $\gamma_i$ with non
negative integral coefficients. The elements of $\Z \Gamma$ are
finite linear combinations of the $\g_i$ with integral
coefficients. If the semigroup $\Gamma$ is cancellative, torsion free and finitely generated, the group $\Z\Gamma$ is a lattice so that $\Gamma$ is isomorphic to a finitely generated subsemigroup of a lattice $\Z^d$.
\end{definition}

\noindent \textbf{Examples of semigroups}: \Par  \noindent
$\bullet$ Given finitely many coprime integers the set of all
combinations of these integers with non negative integral
coefficients is a subsemigroup $\Gamma$ of the semigroup $\mathbf
N$ of integers, and ${\mathbf N}\setminus \Gamma$ is finite. In
fact any semigroup of integers is finitely generated.\Par
\noindent $\bullet$ Let $(s_i)_{i\geq 1}$ be a sequence of
integers such that $s_i\geq 2$ for $i\geq 2$. Define a sequence of
rational numbers $\gamma_i$ inductively by:
$$\gamma_1=\frac{1}{s_1},\ \ \ \gamma_{i+1}=s_i\gamma_i+\frac{1}{s_1\ldots s_{i+1}}.$$
The set of integral linear combinations of the $\gamma_i$ is a
subsemigroup of ${\mathbf Q}_{\geq 0}$, which is not finitely
generated. In fact the $\gamma_i$ form a minimal set of
generators.\Par \noindent $\bullet$ Let $d$ be an integer and let
$\check\sigma$
 (the reason for the dual notation will appear below) be a convex cone of dimension $d$ in $\check\R^d$.
Denote by $M$ the integral lattice of $\check\R^d$. Then the
intersection $\check\sigma\cap M$ is a subsemigroup of the group
$M$, which generates $M$ as a group.\Par  \noindent By a Theorem
of Gordan, if the convex cone $\check\sigma$ is rational in the
sense that it is the intersection of finitely many half spaces
determined by hyperplanes with integral coefficients, then the
semigroup $\check\sigma\cap M$ is finitely generated.

\begin{definition} If $\Delta$ is a
subsemigroup of $\Lambda$ the saturation of $\Delta$ in $\Lambda$
is the semigroup  $\Theta$ consisting of those elements of
$\Lambda$ which have a multiple in $\Delta$. The semigroup
$\Delta$ is saturated in $\Lambda$ if $\Delta = \Theta$.
\end{definition}
\begin{lemma}\label{reduc}
Let $\check\t$ be a rational convex cone in $\check\R^d$ for
the lattice $M$. The semigroup $\check\t\cap M$ is saturated
in $M$ and the saturation of a subsemigroup $\Gamma$ of $M$ is
$\check\s\cap M$ where $\check{\s} = \R_{\geq 0} \Gamma$ is
the closed convex cone generated by $\Gamma$.
\end{lemma}
\textit{Proof.} The first statement is clear. If $\R_{\geq 0}
\Gamma= \check\sigma$, any element of
 $\check\sigma\cap M$ is a combination with rational coefficients of elements of $\Gamma$.
 Chasing denominators shows that an integral multiple of this element is in $\Gamma$. The converse is clear.
$\Box$

\begin{definition}
Let $\Gamma$ be a finitely generated commutative semigroup and
$A$ a commutative ring. The semigroup algebra $A[t^\Gamma]$ of $\Gamma$ with coefficients in $A$
is the ring consisting of finite sums $\sum_\gamma a_\gamma t^\gamma$ with $a_\gamma\in A$,
endowed with the multiplication law
$$(\sum_\gamma a_\gamma t^\gamma)(\sum_\delta b_\delta t^\delta)\
=\  \sum_\zeta(\sum_{\gamma+\delta=\zeta} a_\gamma b_\delta)t^\zeta .$$
\end{definition}

\begin{proposition} \label{normal} If $\Gamma$ is a finitely generated subsemigroup of the lattice
$M\subset \R^d$ such that $\Z \Gamma  = M$  and $\check\s=\R_{\geq
0}\Gamma$ is the rational convex cone generated by $\Gamma$,  the
integral closure of $k[t^\Gamma]$ in its field of fractions is
$k[t^{\check\s\cap M}]$.\end{proposition} This follows directly
from Lemma \ref{reduc}.

\begin{remark}
Quite generally, if $k$ is a field the Krull dimension of $k
[t^\Gamma]$ is equal to the \textit{rational rank} of the
semigroup $\Gamma$, which is the integer $\hbox{\rm dim}_{\Q}
\Gamma \otimes_\Z\Q$ (see \cite{T-valuations}, Proposition 3.1).
\end{remark}

\begin{remark}       \label{special-point}
If $\Gamma$ is a semigroup the ideal of $A[t^\Gamma]$
generated by the $(t^\gamma)_{\gamma\in \Gamma\setminus\{0\}}$ is
non trivial if and only if the cone  $ \R_{\geq 0} \Gamma$ is
strictly convex. If $k$ is a field, it is then a maximal ideal. We
shall mostly be interested in the local study of the spectrum of
semigroup algebras in the vicinity of the origin of coordinates,
which corresponds precisely to that ideal.
\end{remark}

The semigroup algebra has the following universal property:
any semigroup map from $\Gamma$ to the multiplicative semigroup of an $A$-algebra $B$ extends uniquely to an
homomorphism $A[t^\Gamma]\to B$ of $A$-algebras.

 An additive map of semigroups $\phi\colon
\Gamma\to \Gamma'$ induces a graded map of $A$-algebras
$A[\phi]\colon A[t^\Gamma]\to A[t^{\Gamma'}]$ which is injective
(resp. surjective) if $\phi$ is. If the semigroup $\Gamma$ is
torsion-free, the semigroup algebra $A[t^\Gamma]$ injects into
$A[t^{\Z^d}]=A[t_1^{\pm 1},\ldots ,t_d^{\pm 1}]$ and therefore is
an integral domain if $A$ is.
\begin{proposition} \label{product}
Let $\Gamma, \Gamma'$ be two semigroups.
The map of $A$-algebras
$$A[t^{\Gamma\times\Gamma'}]\rightarrow A[u^\Gamma]\otimes_A A[v^{\Gamma'}];\
\ t^{(\gamma,\gamma')}\mapsto u^\gamma\otimes_A v^{\gamma'}$$ is
an isomorphism.
\end{proposition}
\textit{Proof.} This follows immediately from the universal
property. \hfill $\ {\Box}$

\section{Algebraic tori} \label{tori}

Let $k$ be a field. The  multiplicative group  $k^*$ of non-zero
elements of $k$ is equipped with the structure of algebraic group
over $k$, usually denoted by $\mathbf{G}_m := \mbox{Spec} k
[t^{\pm 1}]$. A $d$-dimensional \textit{algebraic torus}  over $k$
is an algebraic group isomorphic to a $(k^*)^d$.

If $M$ is a rank $d$ lattice then $T^M := \mathrm{Spec } k [t^M]$
is an algebraic torus over $k$. If we fix a basis $m_1, \dots,
m_d$ of the lattice $M$ we get a group isomorphism
\[ \Z^d \rightarrow
M, \quad a= (a_1, \dots, a_d) \mapsto \sum_{i=1, \dots d}  a_i m_i \]
 and
isomorphism of $k$-algebras $k [t_1^{\pm 1}, \dots, t_d^{\pm 1} ]
\rightarrow k [t^M]$ which induces an isomorphism $T^M (k)
\rightarrow (k^*)^d$.

\begin{remark}
More generally one can consider the scheme $\mathrm{Spec }  A
[t^M]$, which is an algebraic torus over $\mathrm{Spec } A$ for
any commutative ring $A$.
\end{remark}

A \textit{character} of the torus $T(k)$ is a group
homomorphism $T(k) \rightarrow k^*$. The set of characters
$\mbox{\rm Hom }_{\mathrm{alg. groups}}({T}^M, k^*)$ of $T^M(k)$
is a multiplicative group isomorphic to the lattice $M$ by the
homomorphism given by $m \mapsto t^m$ for $m \in M$.  We identify
the monomials $t^m$ of the semigroup algebra $k[t^M]$ with the
characters of the torus.

By the universal property of the semigroup algebras applied to $k
[t^M]$ we have a representation of $k$-rational points of $T^M$ as
group homomorphisms:
\[
T^M (k) = \mathrm{Hom}_{\mathrm{groups}} (M , k^*) = N \otimes_\Z
k^*,
\]
where $N := \mbox{\rm Hom }(M, \Z)$  is the dual lattice of $M$.
We denote by  $\langle \, , \, \rangle \colon N \times M
\rightarrow \Z$ the duality pairing between the lattices $N$ and
$M$.

A  one parameter subgroup of  $T^M (k)$ is  group  homomorphism
$k^* \rightarrow  T^M (k) $. Any vector $\nu \in N$ gives rise to
a one parameter subgroup  $\l_{\nu} $ which maps $z \in k^*$ to
the closed point of    $T^M (k)$ given by the homomorphism of
semigroups $M \rightarrow k^*$, $m \mapsto z^{\langle \nu, m
\rangle}$.  The set of one parameter subgroups $\mbox{\rm Hom
}_{\mathrm{alg. groups}}(k^* , {T}^M)$  forms a multiplicative
group, which is isomorphic to $N$ by the homomorphism given by
$\nu \mapsto \l_\nu$.

 \section{Affine toric varieties}\label{affine}

In this section  we consider a finitely generated subsemigroup
$\Gamma$ of a free abelian group $M$ of rank $d$. We assume in
addition that the group $\Z \Gamma$ generated by $\Gamma$ is equal
to $M$. We denote by $N$ the dual lattice of $M$. We introduce some useful
 notations.

\begin{notation}
We denote by $M_\R$ the $d$-dimensional real vector space $M
\otimes_\Z \R$. The semigroup $\Gamma$, viewed in $M_\R$, spans
the cone $\R_{\geq 0} \Gamma \subset M_\R$  which we denote also
by $\check{\sigma}$. The dual cone
of $\check{\s}$ is the cone $\s:= \{ \nu \in N_\R \mid \langle
\nu, \g \rangle \geq 0,\, \forall \g \in \check{\s} \}$.  We use
the notation $\t \leq \s$ to indicate that $\t$ is a \textit{face}
of $\s$.
 Any face of $\check{\s}$ is of the form
$\check{\s} \cap \t^\bot$ for a unique face $\t$ of $\s$, where
$\t^\bot$ is the linear subspace $ \{ \g \in M_\R \mid \langle
\nu, \g \rangle = 0, \, \forall \nu \in \t \}$.
\end{notation}

Let $\gamma_1,\ldots ,\gamma_{r}$ be generators of $\Gamma$.  Then the semigroup $\Gamma$ is the image of
$\N^{r}\subset \Z^r$ by the surjective linear map
$b\colon \Z^{r}\to M$ determined by $b(e_i)=\gamma_i$ where the $e_i,\ 1\leq i\leq {r}$
form the canonical basis of $\N^{r}$. The    kernel $\L$ of $b$ is isomorphic to $\Z^{{r}-d}$.

Let
us consider the map of semigroup algebras associated to the map
$b\vert \N^{r}\colon\N^{r}\to \Z^d$, whose image is $\Gamma$. It is a
map of $A$-algebras $A[U_1,\ldots ,U_{r}]\to A[t_1^{\pm 1},\ldots
,t_d^{\pm 1}]$. Its image is the subalgebra $A[t^\Gamma]$ of
$A[t_1^{\pm 1},\ldots ,t_d^{\pm 1}]$.

  An element $m\in \Z^{r}$
can be written uniquely $m=m_+-m_-$ where $m_+$ and $m_-$ have non
negative entries and disjoint support.

By construction, the
kernel of the surjection $A[U_1,\ldots ,U_{r}]\to A[t^\Gamma]$ is
the ideal generated by the binomials $(U^{m_+}-U^{m_-})$ where
$b(m_+)=b(m_-)$. It is the \textit{toric ideal} associated to the
map $ b$. Note that it is not in general generated by the
binomials associated to a basis of $\L$.
Since the algebra
$A[t^\Gamma]$ is an integral domain if $A$ is, the toric ideal is
a prime ideal in that case.

Conversely, assuming now that $A$ is an algebraically closed field
$k$, an ideal generated by binomials in $k[U_1,\ldots ,U_{r}]$ is
called a \textit{binomial ideal}. Those ideals are studied in
\cite{E-S}, where it is shown that a prime binomial ideal
$I\subset k[U_1,\ldots ,U_{r}] $ gives rise to a semigroup algebra
$k[U_1,\ldots ,U_{r}] /I\simeq k[t^\Gamma]$, where
$\Gamma=\N^{r}/_\sim$, and $\sim$ is an equivalence relation
associated to the binomial relations.   The \textit{affine toric
variety} ${T}^\Gamma : =\hbox{\rm Spec } k[t^\Gamma]$ is the
subvariety of the affine space $\A^{r} (k)$ defined by the
binomial equations generating the toric ideal.  By the universal
property of the semigroup algebra, there is a bijection
$$\{\hbox{\rm Closed points of } \hbox{\rm Spec }k[t^\Gamma]\}\leftrightarrow
\{ \hbox{\rm semigroup homomorphisms} \ \Gamma\to k\},$$  where
$k$ is considered as a semigroup with respect to multiplication (in particular $0\in\Gamma$ goes to $1\in k$).

In particular, the torus  ${T}^M (k) = \hbox{\rm
Hom}_{\mathrm{groups}} (M,k^*)$ is embedded in ${T}^{\Gamma}$, as
the principal open set where $t^{\g_1 } \cdots t^{\g_r } \ne 0$.

From the description of closed points of ${T}^\Gamma$ in terms of
homomorphisms of semigroups we have an action of the torus ${T}^M
(k) $ on ${T}^\Gamma (k)$.
Another way to describe this action, which shows that it is
algebraic, is to say that thanks to the universal property of
semigroup algebras it corresponds to the composed map of
$k$-algebras
$$k[t^\Gamma]\rightarrow k[t^\Gamma]\otimes_k k[t^\Gamma]\rightarrow k[t^M]\otimes_k k[t^\Gamma]$$
where the first map is determined by $t^\gamma\mapsto
t^\gamma\otimes_k t^\gamma$ and the second by the inclusion
$\Gamma\subset M$. The corresponding map ${T}^M \times {T}^\Gamma\to
{T}^\Gamma$ is the action.

 Let us now seek the invariant
subsets of $T^{\Gamma}$ under the torus action.

\begin{definition}
Given a semigroup $\Gamma$, a subsemigroup $F\subset \Gamma$ is a
\textit{face} of $\Gamma$ if whenever $x,y\in \Gamma$ satisfy
$x+y\in F$, then $x$ and $y$ are in $F$.
\end{definition}
Let us remark that this condition is equivalent to the fact that
the vector space of finite sums $\sum_{\delta\in \Gamma\setminus
F}a_\delta t^\delta$ is in fact a prime ideal $I_F$ of
$k[t^\Gamma]$. It also implies that $\Gamma\setminus F$ is a
subsemigroup of $\Gamma$ (which in general is not finitely generated) and that the Minkowski sum $\Gamma +
(\Gamma \setminus F) $ is contained in $\Gamma \setminus F$.

\begin{lemma} \label{face}
The faces of the semigroup $\Gamma$ are of the form $\Gamma \cap
\t^\bot$, for $\t \leq \s$.
\end{lemma}
\textit{Proof.} Let $F$ be a face of the semigroup $\Gamma$. Then
there is a face $\check{\s} \cap \t^\bot$ of $\check{\s}$  which
contains $F$ and is of minimal dimension. Then $F$ is also a face
of the semigroup $\Gamma \cap \t^\bot$ and there is an element
$\g_0 \in F$ which belongs to the relative interior of the cone
$\check{\s} \cap \t^\bot$. Under these conditions is enough to
prove that if $\t= 0$ then $F = \Gamma$.

Notice that if $\g \in \Gamma $ and if $(\g + \Gamma) \cap
\Z_{\geq 0} \g_0 \ne \emptyset$ then $\g \in F$ since $F$ is a
face and $\g_0 \in F$. By Theorem 1.9 \cite{K-K} there is $\d_0
\in \Gamma \cap \int (\check{\s} ) $ such that $\d_0 + \check{\s}
\cap M \subset \Gamma$.  We deduce that the intersection $(\g +
\d_0 + \check{\s} \cap M ) \cap \Z_{\geq 0} \g_0$ is non-empty,
for any $\g \in \Gamma$,
 since $\g_0 \in \int (\check{\s}) \cap \Gamma$. \hfill ${\Box}$

\begin{notation}
If $\t \leq \s$ the set $\Gamma \cap \t^\bot$ is a subsemigroup of
finite type of $\Gamma$. If $\t \leq \s$ the lattice $M (\t,
\Gamma)$  spanned by
 $\Gamma \cap \t^\bot$
is a sublattice of finite index of $M (\t) := M \cap
 \t^\bot$.
\end{notation}

\begin{remark} \label{orb-clos}
The torus of the affine toric variety ${T}^{\Gamma \cap \t^\bot}$
is $ T^{M(\t, \Gamma)}$. If $A$ is a commutative ring, the
homomorphism of $A$-algebras $ A[\Gamma] \rightarrow A[\Gamma \cap
\t^\bot] \cong A [\Gamma] / I_{\Gamma \cap \t^\bot}$,   is
surjective and defines a closed embedding \[ i_\t: {T}^{\Gamma
\cap \t^\bot} \hookrightarrow {T}^{\Gamma} \]  over $\mathrm{Spec
} A$.  If $k =A$ the image by the embedding $i_\t$ of a closed
point $u \in T^{\Gamma \cap \t^\bot} (k)$ (or $u \in  T^{M( \t,
\Gamma)} (k) $) is the semigroup homomorphism $i_\t (u) : \Gamma
\rightarrow k$ given by
\[
\g \mapsto \left\{
\begin{array}{cr}
u(\g) & \mathrm{if} \g \in \t^\bot,
\\
 0  & \mathrm{otherwise}.
\end{array}
\right.
\]
\end{remark}

\begin{proposition} \label{orbit-aff}
The map
\[
\t \mapsto \orb (\t, \Gamma):= i_\t( {T}^{M(\t, \Gamma)})  \quad
(\mbox{ resp. } \t \mapsto i_\tau ({T}^{\Gamma \cap \t^\bot})\,) \] defines
a bijection (resp. inclusion-reversing bijection) between
 the faces of $\s$ and the
orbits (resp. the closures of the orbits) of the torus action on
${T}^\Gamma$.
\end{proposition}
\textit{Proof.}  Let $u: \Gamma \rightarrow k$ be a semigroup
homomorphism. Then $u^{-1} (k^*)$ is a face of $\Gamma$, hence of
the form $\Gamma \cap \t^\bot$ for some face $\t $ of $\s$. Any
such $u$ extends in a unique manner to a group homomorphism $M(\t,
\Gamma) \rightarrow k^*$ defining an element of the torus
${T}^{M(\t, \Gamma)}$ of the affine toric variety ${T}^{\Gamma
\cap \t^\bot}$. Conversely, given   a group homomorphism  $u
\colon M(\t, \Gamma) \rightarrow k^*$ we define a semigroup
homomorphism $ i_\t (u) : \Gamma \rightarrow k$  as indicated
above.

It follows that the orbit of the point defined by $u$ by the
action of $T^{M}$ coincides with the image by $i_\t$ of the orbit
${T}^{M(\t, \Gamma)}$ of the point $u_{| \Gamma \cap \t^\bot}
\colon \Gamma \cap \t^\bot \rightarrow k^*$ on the toric variety
${T}^{\Gamma \cap \t^\bot}$.  The rest of the assertion follows
from Remark \ref{orb-clos}. \hfill $\ {\Box}$

The partition induced by the orbits of the torus action on
${T}^{\Gamma}$  is of the form:
\begin{equation} \label{orbit}
{T}^{\G} = \bigsqcup_{\t \leq \s} \orb (\t, \G).
\end{equation}

\begin{proposition} \label{affine-inv}
If $X$ is an affine toric variety with torus ${T}^M$ then $X$ is
${T}^M$-equivariantly isomorphic to ${T}^{\Gamma}$, where $\Gamma
\subset M$ a semigroup  of finite type such that $\Z \Gamma =M$.
\end{proposition}
{\textit Proof.} This is well-known (see  Proposition
2.4, Chapter 5 of \cite{GKZ}). \hfill $\ {\Box}$

We characterize the affine $T^M$-invariant open subsets of
$T^{\Gamma}$.
\begin{definition}
For any face $\t$ of $\s$ the set
\begin{equation} \label{Gamma_tau}
 \Gamma_\t := \Gamma + M ( \t,\Gamma)
\end{equation} is a semigroup of finite type
generating the lattice $M$.
\end{definition}
Notice that the cone $\R_{\geq 0} \Gamma_\t $ is equal to
$\check{\t}$ and  if
 $\t \leq \s$ the set $\int ( \check{\s} \cap
\t^\bot ) \cap \Gamma$ is non empty ($\int$ denotes relative
interior).

\begin{lemma} $\,$ \label{conditions-aff}
\begin{enumerate}
\item[{\rm i.}] The minimal face of the semigroup $\Gamma$ is a
sublattice of $M$  equal to $\Gamma \cap \s^\bot$.

\item[{\rm ii.}] For any $m \in \Gamma$ in the relative interior
of $( \check{\s} \cap \t^\bot )$  we have that
\[
\Gamma_\t = \Gamma + \Z_{\geq 0} (-m).
\]
\item[{\rm iii.}] If $\t \leq \theta \leq \s$ we have that $M (\t,
\Gamma_\theta) = M(\t , \Gamma_\t)$ and $\Gamma_\t = \Gamma_\theta
+  M(\t, \Gamma_\theta)$.
\end{enumerate}
\end{lemma}
\textit{Proof.} i.~By Lemma \ref{face} the correspondence $\t
\mapsto \Gamma \cap \t^\bot$ is a bijection between the  faces of
the cone $\s$  and the faces of the semigroup $\Gamma$. By duality
the minimal face of $\Gamma$ is equal to $\Gamma \cap \s^\bot$. It
is enough to prove that if $\Gamma$ is a semigroup such that
$\Z\Gamma = M$ and $\R_{\geq 0}\Gamma = M_\R$ then
$\Gamma = M$. Since $M$ is the saturation of $\Gamma$ the
assertion reduces to the case of rank one semigroups, for which it
is elementary by Bezout identity.

ii.~If $m  \in \int ( \check{\s} \cap \t^\bot ) \cap \Gamma$ then
the semigroup $\Gamma + \Z_{\geq 0} (-m) \subset M$ spans the cone
$\check{\t} = \check{\s} + \t^\bot \subset M_\R$. By i.~the
minimal face of this semigroup is the lattice $(\Gamma + \Z_{\geq
0} (-m)) \cap \t^\bot$ which coincides by definition with the
lattice $ M(\t, \Gamma)$.

iii.~The lattices $M(\t, \Gamma_\theta)$ and $M(\t, \Gamma)$ are
both generated by $\Gamma \cap \t^\bot$ hence are equal. We have
that $\Gamma_\t = \Gamma_\theta +  M(\t, \Gamma_\theta)$ since
$\theta^\bot \subset \t^\bot$. \hfill $\ {\Box}$

\begin{lemma} \label{open-embedding}
If  $\t \leq \s$ the inclusion of semigroups $\Gamma \subset
\Gamma_\t$ determines a ${T}^M$-equivariant  embedding
${T}^{\Gamma_{\t}} \subset {T}^{\Gamma}$ as an affine open set.
Conversely, if $X \subset T^{\Gamma}$ is  a ${T}^M$-equivariant
embedding of an affine open set then there is a unique  $\t \leq
\s$ such that $X$ is ${T}^M$-equivariantly isomorphic to
$T^{\Gamma_{\t}}$.
\end{lemma}
\textit{Proof.} By Lemma \ref{conditions-aff} we have that
$\Gamma_\t = \Gamma + \Z_{\geq 0} (-m)$. More generally  if $\g
\in \Gamma$ and  $f = t^\g$,  the localization ${T}^{\Gamma}_f =
\mbox{\rm Spec \, } k [ \Gamma ]_f $ is equal to ${T}^{\Gamma +
(-\g)\Z_{\geq 0}}$ and it is embedded in ${T}^{\Gamma}$ as a
principal open set.

Conversely, an affine ${T}^M$-invariant open subset  of
${T}^{\Gamma}$ is an affine toric variety for the torus ${T}^M$
hence it is of the form ${T}^{\Lambda}$, for $\Lambda \subset M$ a
subsemigroup of finite type, such that $\Z \Lambda  = M$
(see Proposition \ref{affine-inv}). We denote the cone $ \R_{\geq
0} \Lambda$ by $\check{\theta}$. Since the embedding
${T}^{\Lambda} \subset {T}^{\Gamma}$ is ${T}^M$ equivariant it is
defined by the inclusion of algebras $k [t^{\Gamma}] \rightarrow k
[ t^{\Lambda} ]$ corresponding to the inclusion of semigroups
$\Gamma \subset \Lambda$. We deduce that $\check{\s} \subset
\check{\theta}$ and hence that $\theta \subset \s$ by duality. We
prove that if $\t$ is the smallest face of $\s$ which contains
$\theta$ then $\Lambda = \Gamma_{\t}$. It is enough to prove that
if $\int{( \theta )} \cap \int (\s)  \ne \emptyset$ then $\Lambda =
\Gamma$.

Notice that the lattice $F= \s^\bot \cap M$ is the minimal face of
$\Gamma$ and the prime ideal $I_{F}$ of $k[ t^{\Gamma} ]$ defines
the orbit $\orb (\s, \Gamma)$, which is embedded as a closed
subset of ${T}^{\Gamma}$. Let us consider a vector $\nu$ such that
$\nu \in \int(\theta) \cap \int (\s)$. Then we get that  $ \s^\bot
\cap M = \check{\s}  \cap \nu^\bot  \cap M $ is contained in  $
\check{\theta}\cap \nu^\bot \cap M   = \theta^{\bot} \cap M $
hence $\Gamma \setminus (\s^\bot \cap M)$ is contained in $
\Lambda \setminus (\theta^\bot \cap M)$ and therefore $1 \notin
I_F k [t^{\Lambda}]$. Since ${T}^{\Lambda} \subset {T}^{\Gamma}$
is an open immersion  $\orb (\s, \Gamma)$ is contained in
${T}^{\Lambda}$. By  (\ref{orbit}) and Proposition \ref{orbit-aff}
the closure of any orbit contained  $  {T}^{\Gamma} $ contains
$\orb (\s, \Gamma)$ thus  ${T}^{\Gamma} \subset {T}^{\Lambda}$.
\hfill $\ {\Box}$

\begin{remark} \label{open-nor}
The immersion of ${T}^M$-invariant affine open subsets is
compatible with normalization. By Lemma \ref{open-embedding} any
$T^{M}$-invariant affine open set of ${T}^{\Gamma}$ is of the form
$T^{\Gamma}_f$ for $f = t^\g$, $\g \in \Gamma$. Then the following
diagram commutes:
\[
\begin{array}{ccc}
{T}^{\check{\sigma} \cap M } &  \hookrightarrow &  {T}^{\Gamma}
\\
\uparrow &  & \uparrow
\\
{T}^{\check{\sigma} \cap M }_f  & \hookrightarrow &
{T}^{\Gamma}_f,
\end{array}
\]
since $\Gamma + (-\g )\Z_{\geq 0}$ is saturated in $\check{\sigma}
\cap M + (-\g )\Z_{\geq 0}$. The vertical arrows are embeddings as
principal open sets while the horizontal arrows are normalization
maps (see Proposition \ref{normal}).
\end{remark}

\end{section}

\begin{section}{Toric varieties}

Given a finite dimensional lattice $N$, recall that a \textit{fan} is a finite set $\Sigma$  of strictly
convex polyhedral cones of the real vector space $N_\R$ which are rational for the  lattice $N$, such that
if $\s \in \Sigma$ any face $\t$ of $\s$ belongs to $\Sigma$ and
if $\s, \s' \in \Sigma$ the cone $\t= \s \cap \s'$ is in $\Sigma$.
If $j \geq 0$ is an integer the subset of $\Sigma (j)$ of
$j$-dimensional cones of $\Sigma$ is called the
\textit{$j$-skeleton} of the fan. The \textit{support} of the fan
$\Sigma$ is the set $|\Sigma | = \cup_{\s \in \Sigma} \s \subset
N_\R$.

We give first a \textit{combinatorial definition} of toric
varieties.

\begin{definition} \label{def-general-toric}
A toric variety is given by the datum of a triple $(N, \Sigma,
\Gamma)$ consisting of lattice $N$,  a fan $\Sigma$ in $N_{\mathbf
R}$ and a family of finitely generated subsemigroups $\Gamma = \{
\Gamma_\s \subset \check\s \cap M\}_{\s \in \Sigma}$ of a lattice $M =
\mbox{Hom} (N, \Z)$ such that:
\begin{enumerate}
\item[{\rm i.}]  $\Z \Gamma_\s = M$ and $\R_{\geq 0}\Gamma_\s=\check\s$, for $\s  \in \Sigma$.
\item[{\rm ii.}] $\Gamma_\t = \Gamma_\s + M (\t, \Gamma_\s)$, for
a each $\s \in \Sigma$ and  any face $\t$ of $\s$.
\end{enumerate}
 The  corresponding toric variety ${T}_{\Sigma}^{\Gamma}$
 is the union of the affine varieties ${T}^{\Gamma_\s}$
for $\s \in \Sigma$ where for any pair $\s, \s'$ in $\Sigma$ we
glue up ${T}^{\Gamma_\s}$ and ${T}^{\Gamma_{\s'}}$ along their common
open affine variety ${T}^{\Gamma_{\s \cap \s'}}$.
\end{definition}
\begin{remark}
The lattice $N$ in the triple $(N, \Sigma, \Gamma)$ is determined
by $\Gamma$. We recall it by convenience. We omit the reference to
the lattice $N$ in the notation $T^\Gamma_\Sigma$.
\end{remark}

\begin{remark}
This definition is consistent with the case of affine toric
varieties. Let  ${T}^\Gamma$ be an affine toric variety in the
sense of Section \ref{affine}.  If $\s' := \{ \t \mid \t \leq \s
\} $ and  $\Gamma' := \{ \Gamma_\t \mid \t \leq \s \}$, where
$\Gamma_\t$ is the semigroup defined by (\ref{Gamma_tau}) for $\t
\leq \s$ then the conditions i.~and ii.~are satisfied by Lemma
\ref{conditions-aff}. Then ${T}^\Gamma $ is $T^M$-equivariantly
isomorphic to ${T}^{\Gamma'}_{\s'}$.
\end{remark}

\begin{remark}
 A triple $(N, \Sigma, \Gamma)$
determines similarly a toric scheme over $\mathrm{Spec } A$, for
any commutative ring $A$.
\end{remark}

\begin{lemma}\label{sum}
Let $(\Sigma, \Gamma)$ as in Definition \ref{def-general-toric}
define a toric variety ${T}_{\Sigma}^{\Gamma}$. Then we have:
\begin{enumerate}
\item[{\rm i.}]
 If
$\s, \theta \in \Sigma$ and if $\t = \s \cap \theta$ then
$\Gamma_\t = \Gamma_\s + \Gamma_\theta$. \item[{\rm ii.}] The
variety  ${T}_{\Sigma}^{\Gamma}$ is separated.
\end{enumerate}
\end{lemma}
{\textit Proof.} The intersection $\t=\s \cap \theta$ is a face of
both $\s$ and $\theta$. By Lemma \ref{conditions-aff} we have that
$M(\t, \Gamma_\t) = M(\t, \Gamma_\s) = M (\t, \Gamma_\theta)$. By
axiom i.~in the Definition \ref{def-general-toric} we get
$\Gamma_\theta, \Gamma_\s \subset \Gamma_\t$ and $\Gamma_\theta +
\Gamma_\s \subset \Gamma_\t$. Conversely, by the separation lemma
for polyhedral cones, for any  $u \in \int (\check{\s} \cap
(-\check{\theta}))$ we have that $\t = \s \cap u^\bot = \theta
\cap u^\bot$.  Notice that we can assume that $ u \in \Gamma_\s
\cap (-\Gamma_\theta) \cap \int (\check{\s} \cap
(-\check{\theta})) \ne \emptyset$. Then by Lemma
\ref{conditions-aff} we obtain $\Gamma_\t = \Gamma_\s + \Z_{\geq
0}  (-u)$. Hence $\Gamma_\t$ is contained in $\Gamma_\s +
\Gamma_\theta$ since $-u \in \Gamma_\theta$.

 The homomorphism $k
[t^{\Gamma_\theta}] \otimes_k k [t^{\Gamma_\s}] \rightarrow k
[t^{\Gamma_\t}]$ which sends $t^{\gamma} \otimes t^{\gamma'}
\mapsto t^{\gamma+ \gamma'}$ is surjective since $\Gamma_\s +
\Gamma_\theta = \Gamma_\t$. In geometric terms this implies that
the diagonal map ${T}^{\Gamma_\t} \rightarrow {T}^{\Gamma_\theta}
\times {T}^{\Gamma_\s}$ is a closed embedding for any $\theta, \s
\in \Sigma$ with $\t = \theta \cap \t$, hence the variety
${T}^{\Gamma}_\Sigma$ is separated (see Chapter 2 of \cite{Ha}).
\hfill $\ {\Box}$

\begin{remark}\label{norma} The morphisms
corresponding to the inclusions $k[t^{\Gamma_\sigma}]\to
k[t^{\Gamma_\sigma+\Gamma_{\sigma'}}]$ are open embeddings
compatible with the normalization maps. The normalization of the
toric variety ${T}_{\Sigma}^{\Gamma}$ is the toric variety
${T}_{\Sigma}$ corresponding to the fan $\Sigma$ and the
normalization map is obtained by gluing-up normalizations
${T}_{\s}:= T^{\check{\s} \cap M} $ of the charts
${T}^{\Gamma_\s}$, for $\Gamma_\s \in \Gamma$ and $\s \in \Sigma$.
\end{remark}

\begin{lemma} Let $\l_v$ be a one-parameter subgroup of the
torus $T^M$ for some $v \in N$. Then $\lim_{z \rightarrow 0} \l_v
(z)$ exists in the toric variety ${T}^{\Gamma}_\Sigma$ if and only
if $v$ belongs to $\vert\Sigma\vert \cap N $.
\end{lemma}
{\em Proof.} The statement is well-known in the normal case (see
Proposition 1.6 \cite{Oda}).  The normalization map $n \colon
{T}_\Sigma \rightarrow {T}^\Gamma_\Sigma$ is an isomorphism over
the torus $T^{M}$. If $\l_v \colon k^* \rightarrow T^{M} \subset
{T}^\Gamma_\Sigma$ is a one-parameter subgroup defined by $v \in
N$ it lifts to the normalization, i.e., there is a morphism
$\bar{\l}_v \colon k^* \rightarrow T^{M} \subset {T}_\Sigma$ in
such a way that $n
 \circ \bar{\l}_v = \l_v$.
Since the normalization is a proper morphism we get by the
valuative criterion of properness that $\lim_{z \rightarrow 0}
\l_v (z)$ exists in the toric variety ${T}^{\Gamma}_\Sigma$ if and
only if $\lim_{z \rightarrow 0} \bar{\l}_v (z)$ exists in
${T}_\Sigma$. \hfill $\ {\Box}$

\begin{lemma} \label{orbit-gen}
Let $T^\Gamma_\Sigma$ be a toric variety. Then the map
\[
\t \mapsto \orb (\t, \Gamma_{\t}):= i_\t( {T}^{M(\t, \Gamma_\t)})
\] defines a bijection
between the faces of $\Sigma$ and the orbits of the torus action
on ${T}^\Gamma_\Sigma$.
\end{lemma}
\textit{Proof.} This is consequence of the definitions and Lemma
\ref{orbit-aff}. \hfill $\ {\Box}$

In order to illustrate the combinatorial definition of a toric
variety we describe the orbit closures as toric varieties.

\begin{notation}   \label{ind}
If $\t \in \Sigma$ we denote by $N_\t$ the sublattice of $N$
spanned by $\t \cap N$ and by $N(\t)$ the quotient $N /
N_\t$. The lattice $N (\t)$ is the dual lattice of $M(\t) = M \cap
\t^\bot$. Since $M(\t, \Gamma_\t)$ is a sublattice of finite index
$i (\t, \Gamma_\t)$ of $M(\t)$ then the dual lattice $N(\t,
\Gamma_\t)$ of $M(\t, \Gamma_\t)$ contains $N(\t)$ as a sublattice
of  finite index equal to  $i (\t, \Gamma_\t)$.
\end{notation}

If $\s \in \Sigma$ and $\t \leq \s$ the image $\s(\t)$ of $\s$ in
$N(\t)_\R = N_\R / (N_\t)_\R$ is a polyhedral cone, rational for
the lattice $N(\t, \Gamma_\t)$. The set $\Sigma(\t) := \{ \s(\t)
\mid \s \in \Sigma, \t \leq \s \}$ is a fan in $N(\t)_\R$. If
$\s(\t) \in \Sigma (\t)$ we set $\Gamma_{\s(\t)} := \Gamma_\s \cap
\t^\bot$.  The set  $\s(\t) \subset N(\t)_\R$ is
 the dual cone   of the cone spanned by $\Gamma_\s
\cap \t^\bot$ in $M(\t)_\R$. Let us denote by $\Gamma(\t)$ the set
$\{ \Gamma_{\s(\t)} \mid \s(\t) \in \Sigma(\t) \}$.

\begin{lemma} \label{orbit-closure}
Let $T^\Gamma_\Sigma$ be a toric variety. If $\t \in \Sigma$ the
triple  $( N(\t, \Gamma_\t), \Sigma(\t), \Gamma(\t) )$ defines a
toric variety $T_{\Sigma (\t)}^{\Gamma(\t)}$. We have a closed
embedding $i_\t : T_{\Sigma (\t)}^{\Gamma(\t)} \to
T^{\Gamma}_{\Sigma}$. The map
\[
\t \mapsto i_{\t}({T}_{\Sigma(\t)}^{\Gamma(\t)})
\] defines a bijection
between the faces of $\Sigma$ and orbit closures of the action of
${T}^M$ on ${T}^\Gamma_\Sigma$.
\end{lemma}
\textit{Proof.} If $\t$ is not a face of $\s$, for $\s \in \Sigma$
then $\orb (\t, \Gamma_{\t})$ is does not intersect the affine
invariant open set ${T}^{\Gamma_\s}$; if $\t \leq \s$, for $\s \in
\Sigma$  the closure of the orbit $\orb (\t, \Gamma_{\t})$ in the
affine open set ${T}^{\Gamma_\s}$ is equal to ${T}^{\Gamma_\s \cap
\t^\bot}$ (see  Lemma \ref{orbit-aff}).

If $\t \leq \theta \leq \s$ then $\theta(\t) \leq \s(\t)$ and
$\theta^\bot \subset \t^\bot$ hence $M(\theta, \Gamma_\s) = M (
\theta(\t), \Gamma_\s \cap \t^\bot)$ is the sublattice spanned by
$\Gamma_\s \cap \theta^\bot$.

If $\t \leq \s, \s' $ and if $\theta = \s \cap \s'$ then we deduce
from condition ii.~in Definition \ref{def-general-toric} that:
\[
\Gamma_\theta \cap \t^\bot = \Gamma_\s \cap \t^\bot + M(\theta
(\t) , \Gamma_{\s (\t)} ) = \Gamma_{\s'} \cap \t^\bot +
M(\theta(\t), \Gamma_{\s'(\t)}).
\]
We obtain that  the triple $( N(\t, \Gamma_\t), \Sigma(\t),
\Gamma(\t) )$ satisfies the axioms in Definition
\ref{def-general-toric} with respect to the torus $T^{M(\t,
\Gamma_\t)}$.

We have also described an embedding ${T}_{\Sigma(\t)}^{\Gamma(\t)}
\hookrightarrow {T}^\Gamma_\Sigma$ in such a way that the
intersection of this variety with any affine chart containing
$\orb(\t, \Gamma)$ is the closure of the orbit   $\orb(\t,
\Gamma)$ in the chart. The conclusion follows from Lemma
\ref{orbit-gen}.
 \hfill $\ {\Box}$
\begin{remark} The non-singular locus of the toric variety
${T}^\Gamma_\Sigma$ is the union of the orbits $\orb(\t, \Gamma)$
corresponding to \textit{regular} cones $\t \in \Sigma$ such their
index $i (\t, \Gamma_\t)$ is equal to $1$.

\end{remark}

\end{section}

\begin{section} {Blowing ups}\label{blow}

The theory of normal toric varieties deals with normalized
equivariant blowing ups, i.e., blowing ups of equivariant ideals
followed by normalization. In this section we build blowing ups of
equivariant ideals in toric varieties.

Let $\sigma$ be a strictly convex rational cone in $N_{\mathbf R}$
and $\Gamma$ a subsemigroup of finite type of the lattice $M$ such
that $\Z \Gamma = M$ and the saturation of $\Gamma$ in $M$ is
equal to $\check\sigma\cap M$. For simplicity we assume that the
cone $\s$ is of dimension $d$ hence $\check \s$ is strictly
convex.

Let us consider a graded ideal ${\mathcal I}$  in $A[t^\Gamma]$,
which is necessarily generated by monomials $t^{m_1},\ldots
,t^{m_k}$. We build the corresponding \textit{Newton polyhedron}
$\Newton_\s ({\mathcal I})$, by definition the convex hull in
$M_{\mathbf R}$ of the $m_i+\check\sigma$, which is also the
convex hull of the set  $|{\mathcal I}|$ of exponents of monomials
belonging to the ideal ${\mathcal I}$ of $A[t^{\Gamma}]$. It is
quite convenient to denote with the same letter $\mathcal{I}$ the
set $\{ m_1, \dots m_k \}$.

The set  $\mathcal{I}$ determines the \textit{order function}:
\begin{equation} \label{ord}
\ord_{\mathcal{I}} \colon \s \rightarrow \R, \quad  \nu \mapsto
\min_{m \in \mathcal{I}} \langle \nu, m \rangle.
\end{equation}
The order function $\ord_{\mathcal{I}}$ coincides with the support
function \emph{\`a la} Minkowski of the polyhedron ${\mathcal N}=\Newton_\s (\mathcal{I})$ defined as the function $H\colon N_{\mathbf R}\to {\mathbf R}$ given by $H(\nu)={\rm min}_{m\in {\mathcal N}}\langle\nu,m\rangle$. It is a
gauge ($\ord_{\mathcal{I}}(\lambda u)=\lambda
\ord_{\mathcal{I}}(u)\ \ \hbox{\rm for }\lambda>0$) which is
piecewise linear. The maximal cones of linearity of the function
$\ord_{\mathcal{I}}$ form the $d$-skeleton  of the fan
$\Sigma(\mathcal{I})$ subdividing $\s$. Each such cone $\s_i $ in
the $d$-skeleton of  $\Sigma(\mathcal{I}) $ is the convex dual of
the convex rational cone generated by the vectors $(m -m_i)_{m \in
\Newton_\s (\mathcal{I})}$, where $m_i$ is a vertex of $\Newton_\s
(\mathcal{I})$. The correspondence $m_i \mapsto \s_i$
 is a bijection  between the set of vertices $\{m_1, \dots, m_s
\} \subset \mathcal{I}$ of the polyhedron $\Newton_\s
(\mathcal{I})$ and the $d$-skeleton of $\Sigma(\mathcal{I})$, such
that
\[
 m_i \mapsto \s_i  \mbox{ if and only if  } \ord_{\mathcal{I}} (\nu) =
\langle \nu, m_i \rangle \mbox { for all } \nu \in \s_i.
\]

 Note that $\Gamma\subset \check\sigma \cap M\subset
\check \sigma_{i} \cap M$. In each of the cones $\check\sigma_{i}$
we consider  the semigroup
\begin{equation} \label{Gamma}
\Gamma_i= \Gamma+\langle m_1-m_i,\ldots,
m_{i-1}-m_i,m_{i+1}-m_i,\ldots, m_k-m_i\rangle \subset
\check\sigma_i \cap M. \end{equation} By Lemma \ref{reduc}, the
saturation in $M$ of this semigroup is equal to $\check\sigma_i
\cap M$. We denote by $\Gamma(\mathcal{I})$ the set consisting of
the semigroups $\Gamma_i$,  together with $\Gamma_{i, \tau}$
(defined by equation (\ref{Gamma_tau})) for $\tau \leq \s_i$,
$i=1, \dots, s$.

\begin{proposition} \label{blowing}
The triple $(N, \Sigma  (\mathcal{I}), \Gamma(\mathcal{I}))$
defines a toric scheme $B$ over $\hbox{\rm Spec }A$. The
inclusions $\Gamma \subset \Gamma_i$, $i=1, \dots,s $, determine a
map of schemes
\[ \p \colon B  \rightarrow \hbox{\rm Spec }A[t^{\Gamma}]
\]
over $ \hbox{\rm Spec }A$, which is the blowing up of the ideal
$\mathcal{I}$.
\end{proposition}
\textit{Proof.} We prove  first that  the triple $(N, \Sigma
(\mathcal{I}), \Gamma(\mathcal{I}))$ satisfies the compatibility
conditions stated in Definition \ref{def-general-toric}. By  Lemma
\ref{conditions-aff} it is enough to check them for the affine
open sets corresponding to two vertices, say $m_1$ and $m_2$, of
$\Newton_\s( \mathcal{I})$. Then, if $\t = \s_1 \cap \s_2$ the
condition we have to prove is that $\Gamma_{1, \t} = \Gamma_{2,
\t}$.

Notice that the vector $m := m_2 - m_1 \in \Gamma_1$ belongs to
the interior of $\check\s_1 \cap \t^\perp$. By Lemma
\ref{conditions-aff}  and the definitions we get $\Gamma_{1, \t} =
\Gamma_1 + \Z_{\geq 0} (- m)$ and similarly $\Gamma_{2, \t} =
\Gamma_2 + \Z_{\geq 0} m $.  Then the assertion follows since
$\Gamma_{1, \t}$, which  is equal to
\[
  \Gamma +   \Z (m_2 - m_1) + \sum_{j=2, \dots, k}
\Z_{\geq 0} (m_j - m_1) = \Gamma +   \Z (m_1 - m_2) + \sum_{j=2,
\dots, k}  \Z_{\geq 0} (m_j - m_2),
\]
is the same semigroup as $\Gamma_{2, \t}$.

It follows that the scheme $B$ is covered by the affine sets
$\hbox{\rm Spec }A[t^{\Gamma_i}]$ for $i =1, \dots, s$.
 Since each
$\Gamma_i$ contains $\Gamma$, there is a natural map
$\pi \colon   \hbox{\rm Spec }A[t^{\Gamma_i}]    \to \hbox{\rm Spec }A[t^\Gamma]$. The sheaf of
ideals on ${B}$ determined by the compositions with $\pi$ of the
generators of ${\mathcal I}$ is generated by $t^{m_i}\circ \pi$ in
the chart $\hbox{\rm Spec}A[t^{\Gamma_i}]$.

It is not difficult to prove that any  semigroup  $\Gamma_{i}$ defined
by (\ref{Gamma}), for $i
> s$, that is when $m_i$ is not a vertex of $\Newton_\s
(\mathcal{I})$, is of the form $\Gamma_{j, \t}$ for some $1 \leq j
\leq s$ and $\t \leq \s_j$. This means that the corresponding
affine chart $\hbox{\rm Spec }A[t^{\Gamma_i}]$ of the blowing up
of $\mathcal{I}$ is in fact an affine open subset of $\hbox{\rm
Spec }A[t^{\Gamma_j}]$, where $m_j$ is a vertex of $\Newton_\s
(\mathcal{I})$. \hfill $\ {\Box}$

\begin{remark}
With the above notations let us consider the \textit{Rees algebra} of $\mathcal{I}$, defined as 
$R[\mathcal{I}] = \bigoplus_{l \geq 0} \mathcal{I}^l s^l $. Since  each power $\mathcal{I}^l$ is a monomial ideal,  the
term $\mathcal{I}^l s^l$ is of the form  $\mathcal{I}^l s^l = \oplus_{\gamma \in |\mathcal{I}^l |} t^\gamma s^l $.
Consider also the semigroup $\Gamma_{\mathcal{I}}$ of $M \times \Z$ generated by 
$(\Gamma \times \{0 \} )  \cup  (|\mathcal{I} | \times \{ 1 \})$. 
By using the map of semigroups  $\Gamma \to \Gamma_{\mathcal{I}}$, 
defined by $\gamma \mapsto (\gamma, 0)$, the semigroup algebra 
$A [t^{\Gamma_{\mathcal{I}}}]$ has the structure of  $A[t^\Gamma]$-algebra. 
There is a unique isomorphism of semigroups algebras $R[\mathcal{I}] \to A[t^{\Gamma_{\mathcal{I}}}]$ over $A$, such that 
$t^{\gamma} s^l \mapsto t^{(\gamma, l)}$.  This is also an isomorphism of $A[t^\Gamma]$- graded algebras, when the grading of 
a monomial $t^{(\gamma, l)}$ is defined to be equal to  $l$.
 The canonical map $\textrm{Proj} (R[\mathcal{I}] ) \to \textrm{Spec} A[t^\Gamma]$ is 
 the blowing up of the ideal $\mathcal{I}$. We have given in Proposition \ref{blowing}
 the combinatorial description of this map as a toric morphism as defined in Section \ref{tm}. See  also Section 11.3 \cite{CLS}.
\end{remark}

\begin{corollary} The blowing-up of an equivariant sheaf of ideals on a toric variety
$T_\Sigma^\Gamma$ is a toric variety. Its description above each
equivariant open affine chart of $T_\Sigma^\Gamma$ is given by
Proposition \ref{blowing}.
\end{corollary}

\end{section}

\section{Toric morphisms} \label{tm}

Recall that a morphism $\phi \colon T^{M'} \to T^{M}$ of algebraic
tori gives rise to  two group homomorphisms
\[ \phi^*: M \to M' \mbox{ and } \phi_*:
N' \to N
\]
between the corresponding lattices of characters and  between the
corresponding lattices of one-parameter subgroups. The
homomorphisms $\phi^*$ and $\phi_*$ are mutually dual and
determine the morphism $\phi \colon T^{M'} \to T^{M}$ of algebraic
tori. Note that $\phi$ is defined algebraically by
\[
k[t^M] \to k [t^{M'}], \quad t^m \mapsto t^{\phi^* (m)}, \quad  m
\in M.
\]

Now suppose that we have two toric varieties $T^{\Gamma}_{\Sigma}$
and $T^{\Gamma'}_{\Sigma'}$ with respective tori $T^{M}$ and
$T^{M'}$ defined by the  combinatorial data given by the triples
 $(N, \Sigma,
\Gamma)$ and $(N', \Sigma', \Gamma')$ (see Definition
\ref{def-general-toric}).

\begin{definition} \label{map-fan}
The homomorphism $\phi_*$ is a \textit{map of fans with attached semigroups} $(N, \Sigma, \Gamma) \to (N', \Sigma', \Gamma ')$ if for
any $\s' \in \Sigma'$ there exists $\s \in \Sigma$ such that
$\phi^* ( \Gamma_\s) \subset \Gamma_{\s'} '$.
\end{definition}

Note then that  $\phi_*$  is a \textit{map of fans}, that is, for
any $\s' \in \Sigma'$ there is a cone $\s \in \Sigma$ such that
image by $\s'$ by the $\R$-linear extension of $\phi_*$ is
contained in $\s$. See Section 1.5 \cite{Oda}.

\begin{proposition} \label{morphism}
Let $\phi \colon T^{M'} \to T^M$ be a morphism of algebraic tori.
If $\phi_*$ defines a map of fans with attached semigroups $(N,
\Sigma, \Gamma) \to (N', \Sigma', \Gamma')$ then it gives rise to
a morphism: $ \bar{\phi} \colon T_{\Sigma'}^{\Gamma'} \to
T_{\Sigma}^{\Gamma} $ which extends $\phi \colon T^{M'} \to T^M$
and is equivariant with respect to $\phi$. Conversely, if $f
\colon T_{\Sigma'}^{\Gamma'} \to T_{\Sigma}^{\Gamma}$ is an
equivariant morphism with respect to $\phi$ then $\phi_*$ defines
a map of fans with  attached semigroups  $(N', \Sigma', \Gamma')
\to (N, \Sigma, \Gamma)$ and $f= \bar{\phi}$. In addition we have
a commutative diagram
\[
\begin{array}{ccc}
T_{\Sigma'} & \longrightarrow  & T_{\Sigma}
\\
\downarrow && \downarrow
\\
T_{\Sigma'}^{\Gamma'} & \longrightarrow & T_{\Sigma}^{\Gamma}
\end{array}
\]
 where the vertical arrows are
normalizations and the horizontal ones are the toric morphisms which
extend $\phi: T^{M'} \rightarrow T^M$.
\end{proposition}
\textit{Proof.} For any $\s' \in \Sigma'$ there exists a cone $\s
\in \Sigma$ such that the restriction of  $\phi^*$ determines a
semigroup homomorphism $\Gamma_\s \to \Gamma_{\s'}'$. The
corresponding homomorphism of $k$-algebras $k [\Gamma_\s] \to k
[\Gamma_{\s'}']$ defines a morphism:
\[
\bar{\phi}_{\s',\s} \colon T^{\Gamma_{\s'}'} \to T^{\Gamma_\s}
\mbox{ given on closed points by } \bar{\phi}_{\s',\s} (x) = x
\circ \phi^*_{| \Gamma_{\s}},
\]
where $x \in T^{\Gamma_{\s'}'}; x\colon \Gamma_{\s'}' \to k$  is a
homomorphism of semigroups. The morphism $\bar{\phi}_{\s',\s}$ is
equivariant through $\phi$ since for any  $y \in T^{M'}$, $y
\colon M' \rightarrow k^*$ group homomorphism  and  any $x \in
T^{\Gamma_{\s'}'}$ we get:
\[
\bar{\phi}_{\s',\s}  (y \cdot x) = (y\cdot x) \circ  \phi^*_{|
\Gamma_{\s}} = (y \circ \phi^*) \cdot (x \circ \phi^*  _{|
\Gamma_{\s}}) = \phi (y) \cdot \bar{\phi}_{\s',\s} (x).
\]
By gluing-up the affine pieces together we get a morphism
$\bar{\phi} \colon  T_{\Sigma'}^{\Gamma'} \to T_{\Sigma}^{\Gamma}$
which is equivariant with respect to $\phi$.

For the converse, since $f$ is assumed to be equivariant through
$\phi$ the image by $f$ of each orbit of the action of $T^{M'}$ on
$T^{\Sigma'}_{\Gamma'}$ is contained in one orbit of the action of
$T^M$ on $T_{\Sigma}^{\Gamma}$. If $\t' \leq \s'$ and $\s' \in
\Sigma'$ then the orbit $\orb (\s', \Gamma_{\s'}')$ is contained
in the closure of $\orb (\t', \Gamma_{\t'}')$ by Proposition
\ref{orbit-aff}. Then there exist $\s, \t \in \Sigma$ such that
\[ f (\orb (\s', \Gamma_{\s'}')) \subset \orb (\s, \Gamma_{\s}) \mbox{ and }
f (\orb (\t', \Gamma_{\t'}')) \subset \orb (\t, \Gamma_{\t}).\]
Since $f$ is continuous $\orb (\s, \Gamma_{\s})$ must be contained
in the closure of $\orb (\t, \Gamma_{\t})$, hence $\t$ is a face
of $\s$  by Proposition \ref{orbit-aff} and Lemma \ref{orb-clos}.
By (\ref{orbit}) it follows that $f(T^{\Gamma_{\s'}'}) \subset
T^{\Gamma_\s}$. The restriction $f_{|T^{\Gamma_{\s'}'}} \colon
T^{\Gamma_{\s'}'} \to T^{\Gamma_\s}$ is  equivariant with respect
to $\phi: T^{M'} \rightarrow T^{M}$. Hence $f_{|T^{\Gamma_{\s'}'}}
\colon T^{\Gamma_{\s'}'} \to T^{\Gamma_\s}$ is defined
algebraically by the homomorphism of $k$-algebras $k
[t^{\Gamma_\s}] \rightarrow k [t^{\Gamma_{\s'}'}]$, which is
obtained by restriction from the homomorphism of $k$-algebras $k
[t^{M}] \rightarrow k [t^{M'}]$ which maps $t^m \mapsto t^{\phi^*
(m)}$ for $m \in M$. This implies that $\phi^* (\Gamma_\s) \subset
\Gamma_{\s'}'$ and also that $f= \bar{\phi}$.

Since $\phi_*$ is a map of fans it defines a toric
morphism between the normalizations of $T^{\Gamma'}_{\Sigma'}$ and
 $T^{\Gamma}_{\Sigma}$. Finally, it is easy to check  that the diagram above
 is commutative.
 \hfill $\ {\Box}$

It is sometimes useful to consider morphisms of toric varieties which send
the torus of the source into a non dense orbit of the target:
Let $(N, \Sigma, \Gamma)$ and $( N', \Sigma', \Gamma')$ be two
triples defining toric varieties $T^{\Gamma}_{\Sigma}$ and
$T^{\Gamma'}_{\Sigma'}$. Let $\t $ be a cone of  $\Sigma$. Suppose
that we have  a morphism of algebraic tori
$ \phi:
T^{M'} \to T^{M(\t, \Gamma_\t)}$
such that $\phi_* : N' \to N(\t, \Gamma_\t)$ defines a map of fans with attached semigroups $(N', \Sigma', \Gamma') \to ( N   (\t,
\Gamma_\t) , \Sigma (\t) , \Gamma(\t))$. Then by Proposition
\ref{morphism} and Lemma \ref{orb-clos} we have a toric morphism
\[
\bar{\phi}: T_{\Sigma'}^{\Gamma'} \to T_{\Sigma(\t)}^{\Gamma(\t)}.
\]
Let us denote by $n \colon  T_{\Sigma} \rightarrow
T^{\Gamma}_{\Sigma}$ the normalization map and by  $\bar{i}_\t
\colon  T_{\Sigma(\t)} \to T_\Sigma $ the closed embedding of the
closure of $\orb (\t)$ in $T_\Sigma$. The following Proposition is
consequence of Proposition \ref{morphism} and Lemma
\ref{orb-clos}.

\begin{proposition} \label{morphism-emb}
The composite of $\bar{\phi}$ with the closed embedding $i_\t:
T_{\Sigma(\t)}^{\Gamma(\t)} \hookrightarrow T_\Sigma^\Gamma$ lifts
to the normalization of  $T_\Sigma^\Gamma$, i.e., there exists a
toric morphism $\psi: T_{\Sigma'}^{\Gamma'} \to  T_{\Sigma (\t)} $
such that
$
 i_\t \circ \bar{\phi} = n \circ \bar{i}_\t \circ \psi$
 if and only if
there is a lattice homomorphism $ \varphi^*: M(\t) \rightarrow M'
$ such that  $\varphi^*_{|  M(\t, \Gamma_\t)} = \phi^*$ and then
$\psi= \bar{\varphi}$.
\end{proposition}

\begin{example}
By Proposition \ref{morphism-emb} the map $ u \mapsto (u,0,0) $,
which parametrizes the singular locus of the Whitney umbrella $\{
x_1^2 x_2 - x_3^2 =0 \}$  does not lift to the normalization while
$u \mapsto (u^2,0,0) $ does.
\end{example}

\section{Abstract toric varieties}

We recall the \textit{usual} definition of toric variety.
\begin{definition} \label{usual}
A toric variety $X$ is an irreducible (separated) algebraic
variety equi\-pped with an action of an algebraic torus $T$
embedded in $X$ as a Zariski open set such that the action of $T$
on $X$ is morphism which extends the action of $T$ over itself by
multiplication.
\end{definition}

As stated in Proposition \ref{affine-inv} any affine toric variety
is  the spectrum of certain semigroup algebra. Gel{\cprime}fand,
Kapranov, and Zelevinsky have defined  and studied those
projective toric varieties which are equivariantly embedded in the
projective space, which is viewed as a toric variety, see
\cite{GKZ}, Chapter 5.

The following Theorem, which is consequence of a more general
result of Sumihiro, provides the key to establish a combinatorial
description of normal toric varieties.
\begin{theorem} (see \cite{Sumihiro-I}) \label{Sh}
Any normal toric variety $X$ has a finite  covering by $T$-invariant
affine normal toric varieties.
\end{theorem}
The statement of Theorem \ref{Sh} does not hold  if the normality
assumption is dropped.
\begin{example}
Let $C \subset \P^2_\C$ be the projective nodal cubic with
equation $y^2 z - x^2( x+ z)$. It is a rational curve with a node
singularity at $P = (0:0:1)$ and only one point $Q = (0:1:0)$ at
the line of infinity  $z=0$. The curve $C$ is rational and has a
parametrization $\p: \P^1_\C \rightarrow C$ such that $\p(0) =
\p(\infty) = P$ and $\p (1) = Q$. Then we have that $\p_{| \C^* }
\colon \C^* \rightarrow C \setminus \{ P \}$ is an isomorphism.
The multiplicative action of $\C^*$ on $\P^1_\C$ corresponds by
$\p$ to the group law action on the cubic hence it is algebraic.
It follows that $C$ is a toric variety with respect to Definition
\ref{usual}. Notice that $C$ is the only open set containing $P$
which is invariant by the action of $\C^*$. This example is also a
projective toric curve which does not admit any equivariant
embedding in the projective space (see \cite{Oda-M} page 4 and
\cite{GKZ} Chapter 5, Remark 1.6). 
\end{example}

\begin{definition} \label{good-action}
An action of a group on an algebraic variety $X$ is \textit{good} if $X$ is covered by a finite number of 
affine open subsets which are invariant by the action.
\end{definition}

We modify the abstract definition of toric varieties as follows:
\begin{definition}
\label{new} A toric variety $X$ is an irreducible separated
algebraic variety equi\-pped with a good action of an algebraic torus
$T$ embedded in $X$ as a Zariski open set such that the action of
$T$ on $X$ extends the action of $T$ on itself by multiplication.
\end{definition}

\begin{theorem} \label{cate}
If $X$ is a toric variety in the sense of Definition \ref{new}
with torus $T$, then there exists a triple $(N, \Sigma, \Gamma)$
as in Definition \ref{def-general-toric} and an isomorphism
$\varphi \colon T \to T^M$ such that the pair $(T,X)$ is equivariantly isomorphic
to $(T^M, T^{\Gamma}_\Sigma)$ with respect to $\varphi$.
\end{theorem}
\textit{Proof.} We denote by $M$ the lattice of characters of the
torus $T$ hence $T= T^M$ and $N$ is the dual lattice of $M$.

By Proposition \ref{affine-inv} an affine $T^M$-invariant open
subset is of the form
 $T^{\Gamma_\s}$ where ${\Gamma_\s}$ is a subsemigroup of finite type of $M$ such that
$\Z {\Gamma_\s} = M$, and $\s  \subset N_\R$ is the dual
 cone of $\check\s =  \R_ {\geq 0}  {\Gamma_\s} \subset M_\R$.
 By Lemma  \ref{open-embedding} the open affine $T^M$-invariant subsets of  $T^{ \Gamma_\s} $ are
$T^{ \Gamma_\t} $, for $\t \leq \s$, where
 $\Gamma_\t = \Gamma_\s + M (\t, \Gamma_\s )$.

 By definition $X$ is covered
by a finite number of $T^M$-invariant affine open subsets of the form
$\{ T^{\Gamma_\s} \}_{\s \in \Sigma }$. We can assume that if $\s \in \Sigma$ and if
$\t \leq \s$ then $\t \in \Sigma$. We are going to show that $\Sigma$ is a fan in $N_\R$,  hence
$T^ {\Gamma_\s} \ne T^{\Gamma_{\s'}}$ if  $\s \ne \s'$.

We have that for any $\s, \s' \in \Sigma$ the intersection $T^
{\Gamma_\s} \cap T^{\Gamma_{\s'}}$ is an affine open subset of the
separated variety $X$ (see Chapter 2 of \cite{Ha}). It is also a
$T^M$-invariant affine subset of both $T^ {\Gamma_\s} $ and $
T^{\Gamma_{\s'}}$, hence it is of the form $T^ {\Gamma_\t}$. By
Lemma \ref{open-embedding}  we obtain two inclusion of semigroups
$\Gamma_\s \to \Gamma_\t$ and $\Gamma_{\s'} \to \Gamma_{\t}$.
  Since $X$ is separated the diagonal map
$T^ {\Gamma_\t} \to  T^ {\Gamma_\s} \times T^{\Gamma_{\s'}}$  is  a closed embedding (see Chapter 2 of \cite{Ha}).
Algebraically, this implies the surjectivity of the homomorphism
\[ k [ t^{\Gamma_\s} ] \otimes_k k [t^{\Gamma_{\s'}} ] \to k [ t^{\Gamma_\t} ]
\mbox{ ,  determined by } t^\gamma \otimes t^{\gamma'} \mapsto t^ {  \gamma + \gamma'}.
\]
It follows that the homomorphism of semigroups $
 {\Gamma_\s}  \times {\Gamma_{\s'}} \to {\Gamma_\t} $,
 $(\gamma, {\gamma'}) \mapsto {  \gamma + \gamma'}$
is surjective.
This proves that $\Gamma_\t = {\Gamma_\s}  + {\Gamma_{\s'}}$ thus
\[
\R_{\geq 0} \Gamma_\t = \check{\t} =  \R_{\geq 0} ({\Gamma_\s}  + {\Gamma_{\s'}}) = \check{\s} + \check{\s'}
\]
By duality we deduce that $\t = \s \cap \s'$.  By Proposition
\ref{open-embedding}  we obtain
\[
 \Gamma_\t = \Gamma_\s + M (\t, \Gamma_\s) =  \Gamma_{\s'} + M (\t, \Gamma_{\s'}).
\]

In conclusion, $\Sigma$ is a fan in $N_\R$ and if $\Gamma := \{
\Gamma_\s \mid \s \in \Sigma  \}$ the triple $(N, \Sigma, \Gamma)$
satisfies the compatibility properties of Definition
\ref{def-general-toric} and the variety $T^\Gamma_\Sigma$ is
$T^M$-equivariantly isomorphic to $X$.
 \hfill $\ {\Box}$

The following corollary is consequence of  Proposition
\ref{morphism} and Theorem \ref{cate}.

\begin{corollary}  The category whose objects are the triples $(N,
\Sigma, \Gamma)$ of Definition \ref{def-general-toric} and
whose morphisms are the maps of fans with attached semigroups of
Definition \ref{map-fan} is equivalent to the category whose objects are
the toric varieties of Definition \ref{new} and whose morphisms are those
equivariant morphism which extend morphisms of the corresponding
algebraic tori; see Proposition \ref{morphism}.
\end{corollary}

\section{Invertible sheaves on toric varieties}

In this section we describe how some of the classical results in the
study of equivariant invertible sheaves on a normal toric variety
extend to the general case.

Let $T_\Sigma^\Gamma$ denote a toric variety defined by the triple
$(N, \Sigma, \Gamma)$. Recall that if $\s \in \Sigma$ we denote by
$T_{\s} = T^{\check{\s} \cap M}$ the normalization of the chart
$T^{\Gamma_\s}$ and by $T_\Sigma$ the normalization of
$T^{\Gamma}_\Sigma$.

A \textit{support function}  $ h : |\Sigma| \to
\R$ is a continuous function such that for each $\s \in \Sigma$
the restriction $h_{| \s}  \colon \s \to \R$ is linear.  We say
that $h$ is \textit{integral} with respect to $N$ if $h ( |\Sigma|
\cap N) \subset \Z$. We denote by $\mbox{SF} (N, \Sigma)$  the set
of support functions integral with respect to $N$. If $h$ is a
support function integral with respect to $N$ then for any $\s \in
\Sigma$ there exists $m_\s \in M$ such that
\[
h (\nu) = \langle \nu, m_\s \rangle, \quad \mbox{ for all } \nu
\in \s.
\]
Notice that by continuity we have that
\begin{equation}  \label{car-con-nor}
 m_\t = m_\s \mod M(\t)  = M \cap \t^\perp, \mbox{ for } \t \leq \s, \, \s \in \Sigma.
\end{equation}
The set $\{ m_\s \mid \s \in \Sigma \}$ determines $h$ but may not
be uniquely determined.

\begin{definition} \label{sup}
 A support function for the   triple $(N, \Sigma, \Gamma)$
is a support function $ h : |\Sigma| \to \R$ integral with respect
to $N$ which in addition has the compatibility property
\begin{equation} \label{car-con}
 m_\t = m_\s \mod M(\t,
\Gamma_\t) , \mbox{ for } \t \leq \s, \, \s \in \Sigma.
\end{equation}

\end{definition}
We denote by  $\mbox{SF} (N, \Sigma, \Gamma)$ the additive group
of support functions for the triple $(N, \Sigma, \Gamma)$. It is a
subgroup of $\mbox{SF} (N, \Sigma)$. A vector $m \in M$ defines an
element of $\mbox{SF} (N, \Sigma, \Gamma) $ hence we have a
homomorphism $M \to \mbox{SF} (N, \Sigma, \Gamma) $, which is
injective if the support of $\Sigma$ spans $N_\R$ as a real vector
space.

Any $h \in \mbox{SF} (N, \Sigma)$ determines
 $T^M$-invariant \textit{ Cartier divisor} $D_h$
on $T_\Sigma$ by \begin{equation} \label{cartier}
D_{h | T_\s } =
\mbox{{\rm div}} (t^{-m_\s}) \mbox{ for }\s \in \Sigma,
\end{equation}
 where $\mbox{{\rm div}} (g)$ denotes
the \textit{principal Cartier divisor} of the rational function
$g$ on an irreducible variety. Notice that $D_h$ is independent of
the possible choices of different \textit{Cartier data} $\{ m_\s
\mid \s \in \Sigma \}$ defining $h$. If $\s, \s' \in \Sigma$, $\t
= \s \cap \s'$ then $T_\t = T_\s \cap T_{\s'}$ and
(\ref{car-con-nor}) guarantees that $ t^{-m_\s + m_{\s'}}$ and $
t^{m_\s - m_{\s'}}$ are both regular functions on $T_\t$. Any
$T^M$-invariant Cartier divisor on $T_\Sigma$ is of the form $D_h$
for $h \in \mbox{SF} (N, \Sigma)$, i.e., it is defined by Cartier
data.

\begin{lemma} If $h \in  \mbox{SF} (N, \Sigma)$ is defined by the Cartier data $\{ m_\s \mid \s \in \Sigma \}$
then it defines a $T^M$-invariant Cartier divisor on
$T^{\Gamma}_\Sigma$ if and only if (\ref{car-con}) holds, that is,
if and only if $h \in  \mbox{SF} (N, \Sigma, \Gamma)$.
\end{lemma}
\textit{Proof.} The condition to determine a Cartier divisor is
that for any $\s, \s' \in \Sigma$, $\t = \s \cap \s'$ the
transition function $ t^{-m_\s + m_{\s'}}$ is an invertible
regular function on $T^{\Gamma_\t} = T^{\Gamma_\s} \cap
T^{\Gamma_{\s'}}$. By Lemma \ref{conditions-aff} this is
equivalent to (\ref{car-con}).
 \hfill $\ {\Box}$

We have shown that the group  $\mbox{\rm CDiv}_{T^M}
(T^\Gamma_\Sigma)$ of $T^M$-invariant Cartier divisors on
$T^\Gamma_\Sigma$ can be seen as a subset of  $\mbox{\rm
CDiv}_{T^M} (T_\Sigma)$. The set $\{ \mathrm{div} (t^m) \}_{m \in
M}$ is a subgroup of $\mbox{\rm CDiv}_{T^M} (T^\Gamma_\Sigma)$
consisting of principal Cartier divisors.

 The map
\[
\mbox{SF} (N, \Sigma, \Gamma) \longrightarrow  \mbox{\rm
CDiv}_{T^M} (T^\Gamma_\Sigma), \quad h \mapsto D_h.
\]
is a group isomorphism. The  inverse map sends a Cartier divisor
$D$ on $T_\Sigma^\Gamma$,  given by the Cartier data $\{ m_\s \mid
\s \in \Sigma \}$, to the function
\[
h_D := |\Sigma| \to \R, \quad h_D (\nu) = \langle \nu, m_\s
\rangle \mbox { if } \nu \in \s.
\]

A Cartier divisor on $T_\Sigma$  determines an \textit{invertible
sheaf} $\mathcal{O}_{T_\Sigma} (D)$. If $U$ is an affine open set
in which $D =  \mbox{{\rm div}} (g_U)$ for some rational function
$g_U$  then the set of sections $ H^0 ( U, \mathcal{O}_{T_\Sigma}
(D))$  consists of those rational functions $f$ such that
$fg_U$ is a regular function on $U$.

We denote by $\mathcal{O}_{T^\Gamma_\Sigma}$ the structure sheaf
on the toric variety ${T^\Gamma_\Sigma}$. The \textit{invertible
sheaf} of a $T^M$-invariant Cartier divisor  $D$ on
$T_\Sigma^\Gamma$ is the sheaf  of
$\mathcal{O}_{T^\Gamma_\Sigma}$-modules
$\mathcal{O}_{T_\Sigma^\Gamma} (D)$. By (\ref{cartier}) the set of
sections of this sheaf  on  $T^{\Gamma_\s}$ is
\begin{equation} \label{section}
H^0 ( {T^{\Gamma_\s}}, \mathcal{O}_{T_\Sigma^\Gamma} (D) )  =
t^{m_\s} k [t^{\Gamma_\s} ].
\end{equation}

We denote by $P_D^\Gamma$ the following subset of $M$:
\begin{equation} \label{PDG}
P_D^\Gamma : = \bigcap_{\s \in \Sigma} {m_\s} + \Gamma_\s.
\end{equation}

The set of global sections of the sheaf
$\mathcal{O}_{T^\Gamma_\Sigma}{(D)}$ is equal to
\begin{equation} \label{global-sec}
H^0 (T^\Gamma_\Sigma, \mathcal{O}_{T^\Gamma_\Sigma} (D) ) =
\bigcap_{\s \in \Sigma} t^{m_\s} k [t^{\Gamma_\s}] =  \bigoplus_{m
\in P_D^\Gamma} k t^m.
\end{equation}

\begin{remark} As in the normal case, a $T^M$-invariant Cartier divisor $D$ defines an equivariant line bundle $\mathcal{L}_{D}$ whose sections coincide with those of the invertible sheaf $\mathcal{O}_{T^\Gamma_\Sigma} (D)$. See \cite{Oda}, Chapter 2.\end{remark}
The \textit{Picard group} $\mathrm{Pic} (X)$ of a variety $X$
consists of the isomorphism classes of invertible sheaves in $X$.

\begin{lemma} Suppose that $|\Sigma | = N_\R$.
 For any Cartier divisor $D$ on the toric variety
$T^{\Gamma}_\Sigma$ we have an
$\mathcal{O}_{T^{\Gamma}_\Sigma}$-module isomorphism $
\mathcal{O}_{T^{\Gamma}_\Sigma} (D) \cong
\mathcal{O}_{T^{\Gamma}_\Sigma} (D_h)$ for some $h \in \mathrm{SF}
(N, \Sigma, \Gamma)$. The following are equivalent for $h \in
\mathrm{SF} (N, \Sigma, \Gamma)$.
 \begin{enumerate}
\item[{\rm i.}] $h \in M$

\item[{\rm ii.}] $D_h$ is a principal Cartier divisor.

\item[{\rm iii.}] $\mathcal{L}_{D_h}$ is a trivial line bundle.

\item[{\rm iv.}] The sheaf $\mathcal{O}_{T^{\Gamma}_\Sigma} (D_h)
$ is isomorphic to $\mathcal{O}_{T^{\Gamma}_\Sigma}$ as
$\mathcal{O}_{T^{\Gamma}_\Sigma}$-module.
 \end{enumerate}
\end{lemma}
\textit{Proof.} See Proposition 2.4 of \cite{Oda}.  \hfill
${\Box}$

\begin{proposition}
Suppose that $|\Sigma | = N_\R$. Then we have canonical
isomorphisms
\[
 \mathrm{SF}
(N, \Sigma, \Gamma) / M  \to \mathrm{Pic} (T^\Gamma_\Sigma)  \to
\mbox{\rm CDiv}_{T^M} (T^\Gamma_\Sigma) / \{ \mathrm{div} (t^m)
\}_{m \in M},
\]
from which we deduce a canonical injection $\mathrm{Pic} (T^\Gamma_\Sigma)
\rightarrow  \mathrm{Pic} (T_\Sigma)$.
\end{proposition}
\textit{Proof.} This follows by using the same arguments as in
Corollary 2.5 \cite{Oda}.  \hfill ${\Box}$

If $\r $ belongs to the $1$-skeleton
$\Sigma(1)$ of the fan $\Sigma$ we denote by $\nu_\r$ the
primitive integral vector for the lattice $N$ in the ray $\r$,
that is the generator of the semigroup $\r \cap N$. We associate
to  $h \in \mbox{SF} (N, \Sigma) $ the polyhedron
\begin{equation} \label{PD}
P_h: = \{ m \in M_\R \mid \langle \nu_\r , m \rangle \geq h
(\nu_\r),  \, \r \in \Sigma (1) \}.
\end{equation}
Recall that
\begin{equation}
P_{l h } = l P_h \mbox{ and } P_h + P_{h'} = P_{h+h'} \label{Plh}
\end{equation}
for any integer $l \geq 1$ and $h, h'  \in \mbox{SF} (N, \Sigma)
$.

\begin{proposition} \label{base-free}
Suppose that
 $|\Sigma |= N_\R$. The following are equivalent for $h \in \mbox{SF} (N, \Sigma,
 \Gamma)$ defining a Cartier divisor $D = D_h$.
 \begin{enumerate}
\item[{\rm i.}] The $\mathcal{O}_{T^\Gamma_\Sigma}$-module
$\mathcal{O}_{T^\Gamma_\Sigma} (D)$ is generated by its global
sections.

\item[{\rm ii.}] $h$ is upper convex, i.e., $h (\nu) + h(\nu')
\leq h (\nu + \nu')$ for all $\nu, \nu' \in N_\R$.

\item[{\rm iii.}]  The polytope $P_h$ has vertices  $\{ m_\s \mid
\s \in \Sigma \}$.
 \end{enumerate}
 If these conditions hold  the convex hull of the set $P_D^\Gamma$ is the
polytope $P_h$ and $h$ is the support function of the polytope
$P_h$.
\end{proposition}
\textit{Proof.}  The proof follows as in the normal case  (see Theorem 2.7 \cite{Oda}). \hfill
${\Box}$

If $|\Sigma| = N_\R$ the support function $h \in \mbox{SF} (N,
\Sigma,
 \Gamma)$, defined by the Cartier data $\{ m_\s \mid \s \in \Sigma
 \}$,
  is \textit{strictly upper convex}
if it is upper convex and in addition
\[
h( \nu) = \langle \nu, m_\s \rangle \mbox{ if and only if } \nu
\in \s \mbox{ , for } \s \in \Sigma.
\]

 Suppose that $h \in \mbox{SF}
(N, \Sigma,
 \Gamma)$ satisfies the equivalent conditions of
Proposition \ref{base-free}. Set $D = D_h$. If $P_{D}^\Gamma = \{
u_1, \dots, u_s \}$ we have a morphism
\begin{equation}\label{embedding}
\Phi_D \colon T^\Gamma_\Sigma \longrightarrow \P^{s-1}_k, \quad
\Phi_D = (t^{u_1} \colon \cdots \colon t^{u_s} )
\end{equation}
(defined in homogeneous coordinates of $\P^{s-1}_k$). The morphism
$\Phi_D$ is equivariant with respect the  map of tori $\Phi_{|T^M}
\colon T^M \to T^{M'}$, where $T^{M'}$ denotes the torus of
$\P^{s-1}_k$ with respect to the fixed coordinates.

\begin{proposition} \label{very-ample}
Suppose that
 $|\Sigma| = N_\R$. The following are equivalent for $h \in \mbox{SF} (N, \Sigma,
 \Gamma)$ defining a Cartier divisor $D = D_h$.
  \begin{enumerate}
\item[{\rm i.}] $D$ is very ample.

\item[{\rm ii.}]  $h$ is strictly upper convex and for all $\s \in
\Sigma (d)$ the set $\{ m - m_\s \mid m \in P_D^\Gamma \}$
generates the semigroup $\Gamma_\s$.
 \end{enumerate}
\end{proposition}
\textit{Proof.} Suppose that $h$ is not strictly upper convex.
Then there exists $d$-dimensional cones $\s, \s' \in \Sigma$ such
that $\t = \s \cap \s'$ is of dimension $d-1$ and $ m_{\s} =
m_{\s'}$. This implies that the section defined by $t^{m_\s}$ in
the open set $ U= T^{\Gamma_\s} \cup T^{\Gamma_{\s'}}$ is nowhere
vanishing.

By definition there exists $1 \leq i \leq s$ such that $m_\s =
u_i$.

 The restriction of $\Phi_D$ to $U$
 factors through the
affine open set $\C^{s-1}$, where the $i$-th homogeneous
coordinate does not vanish. It is of the form:
\[
\Phi_{|U} : U \to \C^{s-1} \mbox{, with }  \Phi_{|U} = (t^{u_1 -
m_\s} , \dots, t^{u_{i-1} - m_{\s}},  t^{u_{i+1} - m_{\s}}, \dots,
t^{u_s - m_{\s} }).
\]
 By Lemma \ref{orb-clos} the closure
of the orbit $\orb (\t, \Gamma)$ is a complete one-dimensional toric variety
contained in $U$. The restriction $\Phi_{| \overline{\orb (\t,
\Gamma)} }$ must be constant hence $\Phi$ is not an embedding.
This implies that if $D_h$ is very ample $h$ is strictly upper
convex.

Suppose that $h$ is  strictly upper convex.  If $\s \in \Sigma$ is
a $d$-dimensional cone then $m_\s$ belongs to  $\{u_i\}_{i=1}^s$,
say $m_\s = u_s$. The restriction of $\Phi$ to $T^{\Gamma_\s}$
factors though the affine open set of $\P^{s-1}_k$  where the last
homogeneous coordinate does not vanish. It is described
algebraically by the homomorphism of $k$-algebras:
\[
k [ y_1, \dots, y_{s-1} ] \to k [t^{\Gamma} ], \quad y_i \mapsto
t^{u_i - m_\s}, \, i =1, \dots, s-1.
\]
This maps defines a closed immersion if and only if it is
surjective. This happens if and only if the set of vectors $\{ u_i
- m_\s \}_{1 \leq i \leq s-1 }$ generate the semigroup
$\Gamma_\s$.
 \hfill ${\Box}$

\begin{proposition}   \label{ample}
Suppose that
 $|\Sigma | = N_\R$. The following are equivalent for $h \in \mbox{SF} (N, \Sigma,
 \Gamma)$.
 \begin{enumerate}
\item[{\rm i.}]  $D_h$ is ample

\item[{\rm ii.}] $h$ is strictly upper convex.
\end{enumerate}
\end{proposition}
\textit{Proof.}  If $D$ is ample then $l D$ is very ample for $l
\gg 0$. Since $l D = D_{l h}$ it follows that $h$ is strictly
upper convex if $l h$ is and  the assertion holds by Proposition
\ref{very-ample}.

Conversely suppose that $h$ is strictly upper convex. We prove
that $l D_h$ is very ample for $l \gg  0$. By Proposition
\ref{very-ample} it is sufficient to prove that there exists an
integer $l \gg 0$ such that for each $d$-dimensional cone $\s \in
\Sigma$ the semigroup $\Gamma_\s$ is generated by $\{ m - l m_\s
\mid m \in P^\Gamma_{D_{lh}} \}$.

If $\s' \in \Sigma$, $\dim \s' = d$,  $\t = \s' \cap \s$   we have
that $\Gamma_\t = \Gamma_\s + \Z_{\geq 0} (-u)$
for any $u \in \Gamma_\s$ in the relative interior of the cone
$\t^\perp \cap \check{\s}$ (see Lemma \ref{conditions-aff}). For
instance we take $u = m _{\s'} - m_{\s}$. We obtain similarly that
$
\Gamma_\t = \Gamma_{\s'} + \Z_{\geq 0} (u).
$

If $\g \in \Gamma_\s$ then $\g$ belongs to $\Gamma_\t$ and there
exists $\g' \in \Gamma_{\s'}$ and an integer $p \geq 0$  such that
$ \g = \g'   + p u$.  If $l \geq p $ we obtain:
\begin{equation} \label{g}
l m_{\s'} + \g' + (l - p) (m_\s - m_{\s'}) =  l m_\s + \g.
\end{equation}
If $l$ is big enough, a formula of the form (\ref{g}) holds for
any $\g$ in a finite set $G_{\s}$ of generators of $\Gamma_\s$
(where $p$ and $\g'$  vary with $\g$) and for any cone $\s' \in
\Sigma(d)$. Since $\g'$ and $m_\s - m_{\s'}$ belong to
$\Gamma_{\s'}$ this implies $t^{ l m_\s + \g }$ defines a section
in $H^0 ( {T^{\Gamma_{\s'}}}, \mathcal{O}_{T_\Sigma^\Gamma}
(D_{lh}) ) $ (see (\ref{section})) for  any cone $\s' \in \Sigma
(d)$. We deduce that for any $\g \in G_\s$ the vector $l m_\s + \g
$ belongs to the set $P^\Gamma_{D_{l h}}$ and $t^{l m_\s + \g}$
defines a global section of $\mathcal{O}_{T_\Sigma^\Gamma}
(D_{lh}) $. \hfill ${\Box}$

\begin{remark}
 Let $\mathcal{A} = \{ u_1, \dots, u_l \} $ be a subset of a
lattice $M$ such that $\Z \mathcal{A} = M$, i.e., $\mathcal{A}$
spans $M$ as a lattice. Gel{\cprime}fand, Kapranov, and Zelevinsky
\cite{GKZ} define a projective
 toric variety $X_\mathcal{A}$ as the closure of the image of the
 map
 \[
\varphi_{\mathcal{A}} =(t^{u_1} \colon \dots \colon t^{u_l})
\colon T^M \to \P^{l-1}_k.
\]
\end{remark}
Let us explain how their definition fits with our notion of
projective toric variety.  Let $P$ be the convex hull of
$\mathcal{A}$ in $M_\R$ and $\Sigma$ the dual fan of $P$. Each $\s
\in \Sigma$ of maximal dimension determines a vertex $m_\s$ of
$P$, which is necessarily an element of $\mathcal{A}$. We
associate to $\s$ the semigroup $\Gamma_\s : = \sum \Z_{\geq 0}
(u_i - m_\s)$. If $\t \leq \s$ we define $\Gamma_\t$ by
(\ref{Gamma_tau}). The set $\Gamma := \{ \Gamma_\theta \mid \theta
\in \Sigma \}$ is well-defined and the triple $(N, \Sigma,
\Gamma)$ defines a toric variety $T^{\Gamma}_\Sigma$ (the argument
is the same as the one used in the proof of Proposition
\ref{blowing}). The support function $h$ of $P$ belongs to
$\mbox{SF} (N, \Sigma, \Gamma)$ and is strictly upper convex. If
$D = D_h$ we deduce from the definitions that $\mathcal{A} 
\subset P_D^\Gamma
$. By Proposition \ref{very-ample} the Cartier divisor
$D$ is very ample, and the morphism (\ref{embedding}) is an
equivariant embedding of $T^\Gamma_\Sigma$ in the projective space. 
 It follows that $ X_{\mathcal{A}}$  and $T^\Gamma_\Sigma$ 
 are isomorphic toric varieties. 
 Notice that 
 the embedding defined by  (\ref{embedding}) may be \textit{degenerate}, that is, 
 the image may lie in a proper linear subspace.

\begin{remark}
If $F = \sum_{i=1}^s c_i t^{u_i} \in k [t^M]$ is a polynomial with
$c_1 \dots c_s \ne 0$, then $F$ defines a global section of
$\mathcal{O}_{T^\Gamma_\Sigma} (D)$ such that the closure of $\{ F
=0 \} \cap T^M$ in $T^\Gamma_\Sigma$ does not meet any
zero-dimensional orbit of $T^\Gamma_\Sigma$.
\end{remark}

\begin{proposition} \label{proj}
 Suppose that
 $|\Sigma | = N_\R$. Then the toric variety $T_\Sigma^\Gamma$ is projective if and only if its normalization
 $T_\Sigma$ is 
 projective. 
\end{proposition}
\textit{Proof.}  
Suppose that $T_\Sigma$ is projective. 
Then there exists a strictly upper convex function $h \in SF(N, \Sigma)$.
By definition, there exists an integer $k_0  \geq 1$ such that 
$k_0 h \in SF(N, \Sigma, \Gamma)$. It follows that  $T_\Sigma^\Gamma$ is projective by 
Proposition \ref{ample}. 

If $T_\Sigma^\Gamma$ is projective there exists a strictly upper convex function $h \in SF(N, \Sigma, \Gamma)$. 
Since $ SF(N, \Sigma, \Gamma) \subset SF (N, \Sigma)$ it follows that $T_\Sigma$ is projective.  
\hfill ${\Box}$

\begin{example}        \label{np}
We give an example of complete non-normal toric variety $T_\Sigma^\Gamma$, which is non projective.
We recall first the Example 6.1.17 of \cite{CLS} of a complete smooth toric variety $T_\Sigma$, which is  
non-projective. 
The maximal faces of the fan $\Xi$ associated to $(\P^1)^3$ are the eight orthants of $\R^3$. 
We denote the canonical basis of $\Z^3$ by $e, f, g$ and we set $N = \Z^3$. In term of this basis 
we consider the vectors 
$ a = (2,1,1)$, $ b = (1, 2,1)$, $c= (1,1, 2)$
and $d= (1,1,1)$.  
We define from $\Xi$ a complete fan $\Sigma$,  by subdividing the first orthant
$\R^3_{\geq 0} = \R_{\geq 0} e + \R_{\geq 0} f+ \R_{\geq 0}g$,  
adding the rays defined by $a, b, c, d$ in the way indicated in the Figure \ref{complex-figure}.

\setlength{\unitlength}{0.4 mm}
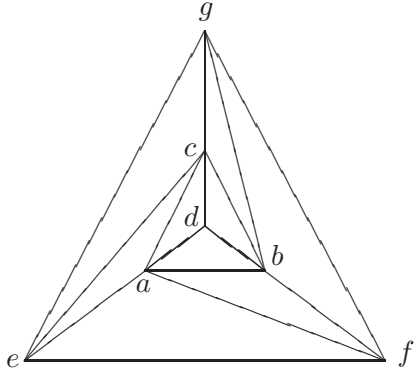
\begin{figure}
\begin{center}
\begin{picture}(120,120)(0,0)
\linethickness{0.1mm}

\drawline(0,0)(120,0)
\drawline(0,0)(40,30)
\drawline(0,0)(60,70)
\drawline(0,0)(60, 110)

\drawline(60, 110)(60,70)
\drawline(60, 110)(80,30)
\drawline(60, 110)(120,0)

\drawline(120,0)(80,30)
\drawline(120,0)(40,30)

\drawline(60,45)(40,30)
\drawline(60,45)(80,30)
\drawline(60,45)(60,70)

\drawline(60,45)(40,30)
\drawline(60,45)(80,30)
\drawline(60,45)(60,70)

\drawline(40,30)(80,30)
\drawline(40,30)(60,70)
\drawline(80,30)(60,70)

\jput(-6,-2){{$e$}}
\jput(37,23){{$a$}}
\jput(82,32){{$b$}}
\jput(53,45){{$d$}}
\jput(53,68){{$c$}}
\jput(58,115){{$g$}}
\jput(124,0){{$f$}}

\end{picture}
\end{center}

    \caption[]{The cones of $\Sigma$ subdividing the first orthant in Example \ref{np}}\label{complex-figure}
\end{figure}

Let us denote by $a^*, c^*, d^*$  (resp. $a', c', e'$) the dual basis of 
$a, c, d$ (resp. of $a, c, e$). In term of these basis of $M$ we introduce the following non-saturated 
semigroups:
\[
 \begin{array}{lclclclclcl}
 \Gamma_{acd} & := & a^* \Z_{\geq 0}  & + &   c^* \Z_{\geq 0} &  + & 
 (a^* + d^*)  \Z_{\geq 0}  &  + &  (c^* + d^*)  \Z_{\geq 0}  &  + &    2 d^* \Z_{\geq 0}, 
  \\
 \Gamma_{ac}  & := & a^* \Z_{\geq 0} & + &   c^* \Z_{\geq 0}  & + &  
 (a^* + d^*)  \Z_{\geq 0}    & + &  (c^* + d^*)  \Z_{\geq 0}   & + &   2 d^* \Z, 
 \\
 \Gamma_{ace} & := & a' \Z_{\geq 0} & + &   c' \Z_{\geq 0}  &+&  
 (a' + e')  \Z_{\geq 0}    & + &  (c' + e')  \Z_{\geq 0}  &  + &     2 e' \Z_{\geq 0}.
 \end{array}
\]
Let us denote by $\s_{acd}$ the cone  $\R_{\geq 0} a + \R_{\geq 0} c + \R_{\geq 0} d$. 
We use a similar notation to define the cones $\s_{ace}$ and $\s_{ac}$. 
Let us define a semigroup $\Gamma_\s$ associated to each cone $\s \in \Sigma$ by: 
\[
\Gamma_\s:= \left\{ 
 \begin{array}{lcccl}
         \Gamma_{acd} & \mbox{ if } & \s & = &       \s_{acd}
         \\
          \Gamma_{ace} & \mbox{ if } & \s & =  &   \s_{ace}
                    \\
           \Gamma_{ac} & \mbox{ if } & \s & = &  \s_{ac}
                     \\
          \check{\s}\cap M    & & & \mbox{otherwise}. &     
 \end{array}
\right.
\]
Then the triple $(N, \Sigma, \Gamma)$ satisfies the conditions in the definition \ref{def-general-toric}: 
It is immediate that $\Z \Gamma_\s = M $ and $\R_{\geq 0} \Gamma_\s = \check{\s}$, 
for any $\s \in \Sigma$. We check the compatibility conditions among those semigroups defining different charts. 
First, if $\t \ne \s_{ac}$ is a proper face of $\s_{acd}$ then we get that  
$\Gamma_{acd} + M (\t, \Gamma_{acd} ) = \check{\t} \cap M$  is a regular semigroup. 
The same assertion holds replacing $acd$ by $ace$. In these cases the compatibility conditions 
are the same as in the normal case. It remains to check what happens when $\t = \s_{ac}$ is the common face 
of $\s_{acd}$ and $\s_{ace}$. One gets that 
$\Gamma_{acd} + M(\t, \Gamma_{acd}) = \Gamma_{ac}$
while 
$\Gamma_{ace} + M(\t, \Gamma_{ace})$ is the semigroup $\Gamma_{ac}'$ generated by
$a', c', a' + e', c' + e', \pm 2 e' $. Since $c' = c^*$, $a' = a^* + d^*$ and $e' = - d^*$ 
it follows that $\Gamma_{ac} = \Gamma_{ac}'$,  hence all the compatibility conditions hold. 
By Proposition \ref{proj} we deduce that the complete toric variety $T_\Sigma^\Gamma$
is non-normal and non-projective.

\end{example}

\medskip

\begin{center}
{\textbf{Part II: Local uniformization of maximal rank monomial valuations on toric varieties by Semple-Nash modifications}}
\end{center}

\medskip

\medskip
In this Part we prove that a maximal rank monomial valuation dominating a
point of a toric variety admits a canonical local uniformization
by a finite number of iterated blowing ups of logarithmic jacobian
ideals. Recall that, as shown below, if $k$ is an algebraically
closed field of characteristic zero the blowing up of the
logarithmic jacobian ideal of an affine toric variety $T^\Gamma$
coincides with the Semple-Nash modification. This fact, originally
due to G\'erard Gonzalez Sprinberg in the normal case (\cite{GS-ENS}), is
our starting point.\par\medskip

The sequence of logarithmic jacobian blowing-ups of a toric
variety $T^{\Gamma^{(1)}}_{\Sigma^{(0)}}$ is a sequence of
toric varieties $T_{\Sigma^{(i)}}^{\Gamma^{(i+1)}}$ defined
by a sequence $\Sigma^{(i)}$ of refinements (or
subdivisions) of $\Sigma$ with attached families
$\Gamma^{(i+1)}$ of semigroups.
The center of a monomial valuation of maximal rank $d$ on the
toric variety  $T_{\Sigma^{(i)}}^{\Gamma^{(i+1)}}$ is a zero
dimensional orbit corresponding to a $d$-dimensional cone
$\t^{(i)}$ of the fan   $\Sigma^{(i)}$. The main result of this
part shows that the affine toric variety
$T_{\t^{(i)}}^{\Gamma^{(i+1)}}$ becomes smooth after
finitely many iterations of logarithmic jacobian blowing-ups (see
Theorem \ref{lu}).

In most proofs of
resolution the strategy is to attach to points an invariant which
takes its minimal value only for regular points  and then show
that it can be made to decrease by successive blowing-ups. Our
strategy is different: we show that the very nature of the
blowing-up of the logarithmic jacobian ideals forces the cones
distinguished by the valuation
in the successive refinements $\Sigma^{(i)}$ of $\Sigma$ to
stabilize for $i$ large enough, meaning that they are not
subdivided in the $\Sigma^{(j)}$ for $j\geq i$. If one can
stabilize the cones of maximal dimension the logarithmic jacobian
blowing-ups are finite morphisms from then on, and it is easy to
show that they resolve in finitely many steps (see Proposition
\ref{stable}). This stabilization is not measured by the constancy
of some local invariant. The basic idea is to show stabilization
by extending it from lower-dimensional cones to higher-dimensional
ones, so that if one really insists on having an invariant, it
should be the minimal codimension of  stable faces of the
cone $\t^{(i)}$ picked by the valuation; it is at most $d-1$
since edges are stable, and if it is zero, we are essentially
done. \par

Here is a quick description of the structure of the proof: first
we study the problem with respect to the monomial valuation
associated to a vector $\nu\in \s\cap N$, where $\s\in\Sigma(d)$.
For each $j$ such a vector determines a unique cone
$\theta^{(j)}\in\Sigma^{(j)}$ containing $\nu$ in its
relative interior, and the first observation is that this sequence
$(\theta^{(j)})_{j\geq 0}$ stabilizes for $j\geq j_1$ say; the
limit $\theta^{(\infty)}=\theta^{(j_1)}$ is by definition a stable
cone of $\Sigma^{(j_1)}$. This implies that the chart
$T^{\Gamma^{(j)}}_{\theta^{(\infty)}}$ is non-singular for $j\gg
0$ (see Propositions \ref{stable} and \ref{prev}).

One of the difficulties is that the logarithmic jacobian
blowing-up of the variety does not induce the logarithmic jacobian
blowing-up of its lower-dimensional orbit closures. A key point in
the proof is that given a stable cone $\eta$, a nested sequence of
cones $\zeta^{(j)}\in\Sigma^{(j)}$ containing $\eta$ as a
codimension one face necessarily stabilizes (Proposition
\ref{uno}). This uses the fact that for every one dimensional
orbit closure associated to a stable cone of codimension one, the
effect of the ambient blowing-up is very similar that of its
logarithmic jacobian blowing-up (see Claim \ref{oneorbit}). A
monomial valuation of maximal rank defines a nested sequence of
$d$-dimensional cones $\t^{(j)} \in \Sigma^{(j)}$. We show first
that this sequence  contains a stable cone $0 \ne \eta \leq
\t^{(j)}$ for $j \gg 0$. If $\eta \ne  \t^{(j)}$ we analyze the
the blowing up of the logarithmic jacobian ideal on the orbit
closure associated to $\eta$ on the chart picked up by the
valuation and we prove that after finitely many iterations we
obtain a stable cone $\eta \leq \theta \leq  \t^{(j)}$ of smaller
codimension.

In Sections \ref{EI} and \ref{EI2} we
give an interpretation of the main result  in terms of
the Zariski-Riemann space of a fan, introduced by Ewald and Ishida (\cite{Ewald-Ishida}).

The recent paper \cite{ALPPT} suggests that it would be interesting to develop
 an approach from a computational viewpoint to the iteration of Semple-Nash modification.

 \section{The Semple-Nash modification: preliminaries}

 In \cite{S}, Semple introduced the Semple-Nash modification
 of an algebraic variety and asked whether a finite number of iterations
 would resolve the singularities of the variety.
 The same question was apparently rediscovered by Chevalley and Nash in the 1960's,
 and studied notably by Nobile (see  \cite{No}), Gonzalez-Sprinberg (see  \cite{GS} and \cite{GS-ENS}),
  Hironaka (see  \cite{H}),  and Spivakovsky (see  \cite{Spivakovsky}).
  The best consequence so far of all this work is the Theorem, due to Spivakovsky,
  stating that by iterating the operation consisting of the Semple-Nash modification followed
  by normalization one eventually resolves singularities of surfaces over an algebraically
  closed field of characteristic zero.\par
Let $X$ be a reduced algebraic variety or analytic space,  which
we may assume of pure dimension $d$ for simplicity. Whenever we
speak of the Semple-Nash modification, we assume that we are
working over an algebraically closed field $k$ of characteristic
zero. Consider the Grassmanian $g\colon \hbox{\rm
Grass}_d\Omega^1_X\to X$; it is a proper algebraic map, which has
the property that its fiber over a point of $x$ is the Grassmanian
of $d$-dimensional subspaces of the Zariski tangent space
$E_{X,x}$. The map $g$ is characterized by the fact that
$g^*\Omega^1_X$ has a locally free quotient of rank $d$ and $g$
factorizes in a unique manner every map to $X$ with this property.
Let $X^o$ denote the non-singular part of $X$, which is
$d$-dimensional and dense in $X$ by our assumptions. Since the
restriction $\Omega^1_X\vert X^o$ is locally free the map $g$ has
an algebraic section over $X^o$ and the Semple-Nash modification
is defined as the closure $NX$ of the image of this section,
endowed with the natural projection $n_X\colon NX\to X$ induced by
$g$. The map $n_X$ is proper and is an isomorphism over $X^o$; it
is a modification. Like the Grassmanian of $\Omega^1_X$, it is
defined up to a unique $X$-isomorphism.\par A local description
can be given for a chart $X\vert U$ of $X$ embedded in affine
space ${\mathbf A}^N(k)$ by taking the closure in $(X\vert U)\times {\mathbf G}(N,d)$ of the graph of the
Gauss map $\gamma\colon (X\vert U)^o\to {\mathbf G}(N,d)$ sending each non-singular point to the class of its tangent space in the Grassmanian of $d$-dimensional vector subspaces in ${\mathbf A}^N(k)$. For any
point $x\in X$ the fiber $n_X^{-1}(x)$ is the subset of ${\mathbf
G}(N,d)$ consisting of limit positions at $x\in X$ of tangent
spaces to $X$ along sequences of non-singular points tending to
$x$. In this guise, the Semple-Nash modification appears in a complex-analytic
framework in the paper \cite{Wh} of Hassler Whitney in connection with equisingularity problems.

 \begin{proposition}\label{iso}{\rm (Nobile), see \cite{No} and \cite{T-hunting}}.
 Let $X$ be a reduced equidimensional space; if the map $$n_X\colon NX\to X$$ is an isomorphism,
 the space $X$ is non-singular.
\end{proposition}
For the convenience of the reader, we sketch the proof found in
\cite{T-hunting}:\par\noindent If the map $n_X$ is an isomorphism,
the sheaf $\Omega^1_X$ has a locally free quotient of rank $d$.
The problem is local, so it is enough to prove that the existence
of a surjective map $\phi\colon \Omega^1_{X,x}\to {\mathcal
O}^d_{X,x}$ implies, in characteristic zero,
 that ${\mathcal O}_{X,x}$
 is regular.
Passing to the completion and tensoring $\Omega^1_{X,x}$ by
$\hat{\mathcal O}_{X,x}$ we may assume that ${\mathcal O}_{X,x}$
is complete. We consider the linear map $e\colon{\mathcal
O}^d_{X,x}\to{\mathcal O}_{X,x} $ sending the first basis vector
to $1$ and the others to $0$. The composition of $e$ with the map
$\phi$ gives a surjective map, so that there has to be an element
$h\in {\mathcal O}_{X,x}$ such that the image of $dh$ in
${\mathcal O}_{X,x}$ by $e\circ\phi$ is equal to $1$, and then the
$k$-derivation $D\colon{\mathcal O}_{X,x}\to {\mathcal O}_{X,x}$
corresponding to $e\circ\phi$ is such that $Dh=1$. In
characteristic zero one can formally integrate this non vanishing
vector field using the formal expansion of $\mathrm{exp}(-hD)$ to
get an isomorphism ${\mathcal O}_{X,x}\simeq {\mathcal O}_1[[h]]$
where $ {\mathcal O}_1\simeq {\mathcal O}_{X,x}/(h)$. By
construction $ {\mathcal O}_1$ satisfies the same assumptions as
${\mathcal O}_{X,x}$ in one less dimension. By induction we are
reduced to dimension zero, but a reduced zero dimensional complete
equicharacteristic local ring is $k$ in our case. We refer to
\cite{T-hunting} for details, and to \cite{No} for the original
proof.
\begin{remark} We will see below in
Section \ref{sheaf}, Proposition \ref{jacoiso},  the
characteristic-free version of this statement, which is that if
the blowing-up of the logarithmic jacobian ideal is an
isomorphism, the toric variety is smooth. Note that the
Semple-Nash modification is defined in any characteristic but its
being an isomorphism does not imply regularity in positive
characteristic; it is the case for $y^p-x^q=0$ with $(p,q)=1$ in
characteristic $p$. See \cite{No}.\end{remark}
  \section{The Semple-Nash modification in the toric case}
 The following is an extension to the case of not necessarily normal toric varieties  of a result of
Gonzalez-Sprinberg (\cite{GS-ENS}; a summary of this work appeared in  \cite{GS-CRAS})
which was revisited by Lejeune-Jalabert and Reguera in the appendix to \cite{LJ-R}.
 \par\noindent
 Let $X$ be an affine toric variety over an algebraically closed field $k$. Using the notations of Section \ref{affine} we write its ring $$R=k[U_1,\ldots ,U_{r}]/P,$$ where $P$ is a prime binomial ideal $(U^{m^\ell}-U^{n^\ell})_{\ell\in\mathbf L}$ of the polynomial ring $k[U_1,\ldots ,U_{r}]$. Let $d$ be the dimension of $X$ and  denote by  ${\mathcal L}\subset \Z^r$ the lattice generated by the differences $(m^\ell -n^\ell)_{\ell\in \mathbf L}$; by \cite{E-S}, it is a direct factor of $ \Z^r$ since $X$ is irreducible and $k$ is algebraically closed. Setting $c=r-d$, we may identify $\mathbf L$ with $\{1,\ldots ,L\} $ with $L=\vert \mathbf L\vert$ in such a way that the lattice generated by $(m^1-n^1,\ldots ,m^c-n^c)$ has rank $c$.
 The quotient $\Z^r/\mathcal L$ is isomorphic to $\Z^d$ and we have an exact sequence
\begin{equation}\label{star} 0\to {\mathcal L}\stackrel{\psi}\rightarrow \Z^r\to \Z^d\to 0.\end{equation}
Our affine toric variety $X$ is $\hbox{\rm Spec} k[t^\Gamma]$,
where $\Gamma$ is the semigroup generated in $\Z^d$ by the images
$\gamma_1,\ldots ,\gamma_r$ of the basis vectors of $\Z^r$.
The \textit{logarithmic jacobian ideal} of $X$ is the ideal
of $R = k[t^\Gamma]$ generated by the images of the products
$U_{i_1}\dots U_{i_d}$ such that $\hbox{\rm
Det}(\gamma_{i_1},\ldots ,\gamma_{i_d})\neq 0$.
\par

\begin{proposition}\label{Nashbl}{\rm (Generalizing \cite{GS-ENS} \cite{GS-CRAS} and  \cite{LJ-R})}
Let $X$ be an affine toric variety over an algebraically
closed field of characteristic zero. The Semple-Nash modification
of $X$ is isomorphic to the blowing-up its logarithmic jacobian
ideal.
\end{proposition}
\begin{proof} Keeping the notations just introduced, a straightforward computation using logarithmic differentials shows that the jacobian determinant
$J_{K,\mathbf L'}$ of rank
$c=r-d$ of the generators $(U^{m^\ell}-U^{n^\ell})_{\ell\in\{1,\ldots , L\}}$ of our prime binomial ideal $ P\subset k[U_1,\ldots
,U_r]$, associated to a sequence $K=(k_1,\ldots , k_c)$ of distinct elements of $\{1,\ldots, r\}$ and a subset $\mathbf L'\subseteq \{1,\ldots , L\}$ of
cardinality $c$, satisfies the congruence
$$U_{k_1}\ldots U_{k_c}.J_{K,\mathbf L'}\equiv
\big(\prod_{\ell\in \mathbf L'}U^{m^\ell}\big)\hbox{\rm Det}_{K,\mathbf L'}\big( (\langle m-n\rangle )\big) \ \
\hbox{\rm mod.} P,$$ where $\big(\langle m-n\rangle\big)$ is the matrix of the
vectors $(m^\ell-n^\ell)_{\ell\in \{1,\ldots , L\}}$,  and $\hbox{\rm Det}_{K,\mathbf L'}$ indicates the minor
in question. By Lemma 6.3 of \cite{T-valuations}, the rank of the image in $k^{r\times L}$ of the  matrix $\big(\langle m-n\rangle\big)$ is equal to $c$.\par
  By (\cite{No}, proof of Th.1) the Nash modification of $X$ is isomorphic to the blowing up in $X$ of the ideal generated as $K=(k_1,\ldots , k_c)$ runs through the sets of $c$ distinct elements of $(1,\ldots ,r)$ by the elements $J_{K,\mathbf L_0}$ satisfying the congruences
$$U_{k_1}\ldots U_{k_c}.J_{K,\mathbf L_0}\equiv
\big(\prod_{\ell\in\mathbf L_0}U^{m^\ell}\big)\hbox{\rm Det}_{K,\mathbf L_0}\big( (\langle m-n\rangle )\big) \ \
\hbox{\rm mod.} P,$$ where $\mathbf L_0=(1,\ldots ,c)$ is, after renumbering of $\{1,\ldots, L\}$, a subset such that these jacobian determinants are not all zero; such subsets exist since the $J_{K,\mathbf L'}$ are not all zero. Remark the necessity that $J_{K,\mathbf L_0}=0$ whenever the determinant  on the right side is zero.\par
Remark also that by \cite{CCD}, we may not suppose that the first $c$ binomials define a complete intersection.\par\medskip
Now for each $K$ let us multiply both sides by  $U_{i_1}\ldots U_{i_d}$, where $I=(i_1,\ldots ,i_d)=\{1,\ldots ,r\}\setminus K$. We obtain for each $K$ the equality:\par
\begin{equation}\label{(**)}
U_1\ldots U_{r}.J_{K,\mathbf L_0}\equiv U_{i_1}\dots
U_{i_d}
\big(\prod_{\ell\in \mathbf L_0}U^{m^\ell}\big)\hbox{\rm Det}_{K,\mathbf L_0}\big( (\langle m-n\rangle )\big) \ \
\hbox{\rm mod.} P.
\end{equation}

Taking exterior powers for the map $\psi$ in the sequence (\ref{star}) gives an injection
$$0\to\stackrel{r-d}\Lambda {\mathcal L}\ \stackrel{\stackrel{r-d}\Lambda \psi}\longrightarrow \ \ \stackrel{r-d}\Lambda \Z^r$$
whose image is a primitive vector in $\stackrel{r-d}\Lambda \Z^r$ since it is a direct factor.\par\noindent Let $\L_0\subset \L$ be the lattice generated by the differences $(m^1-n^1,\ldots ,m^c-n^c)$, that is, corresponding to the first $c$ binomial equations. The image of its $(r-d)$-th exterior power is a non-zero multiple of the primitive vector $\stackrel{r-d}\Lambda {\mathcal L}$; all the $c\times c$  minors of the matrix $\big(\langle m-n\rangle\big)$ involving vectors $m^\ell-n^\ell$ with $\ell>c$ are rationally dependent upon those which do not.  Consider now the $d$-th exterior power of the map dual to the surjection $ \Z^r\to \Z^d\to 0$ of (\ref{star}):
$$0\to \stackrel{d}\Lambda \check{\Z}^d\to  \stackrel{d}\Lambda\check{\Z}^r.$$
The image of $\stackrel{d}\Lambda \check{\Z^d}$ is a primitive vector in
$ \stackrel{d}\Lambda\check{\Z^r}$.
 \par\noindent
 By the natural duality isomorphism between $ \stackrel{d}\Lambda\check{\Z^r}$ and $\stackrel{r-d}\Lambda \Z^r$ (see  \cite {B}  \S 11, No. 11, Prop. 12) deduced from the pairings
 $$\begin{array}{lr}\stackrel{d}\Lambda\check{\Z}^r\otimes \stackrel{d}\Lambda{\Z}^r\to \Z,\ \ \stackrel{d}\Lambda{\Z}^r\otimes\stackrel{r-d}\Lambda \Z^r\to \Z,\cr\end{array}$$ this vector correspond to the image of $\stackrel{r-d}\Lambda {\mathcal L}$ in such a way that the coordinate which corresponds to the determinant of the vectors $\gamma_{i_1},\ldots ,\gamma_{i_d}$ in $\Z^d$ is a rational multiple of the determinant $\hbox{\rm Det}_{K,\mathbf L_0}\big( (\langle m-n\rangle )\big)$, which is non-zero since our base field is of characteristic zero. \par
Equation \ref{(**)} now shows that the ideal of $R$ generated by the $J_{K,\mathbf L_0}$ differs from the ideal generated by the images of the products $U_{i_1}\dots U_{i_d}$ such that $\hbox{\rm Det}(\gamma_{i_1},\ldots ,\gamma_{i_d})\neq 0$  only by the product by invertible ideals, so that these two ideals determine isomorphic blowing ups, which proves the Proposition.\end{proof}
\begin{remark} The proof found in \cite {LJ-R} is valid in the non-normal case; the proof given here makes explicit the connection of the logarithmic jacobian ideal with the usual one.\end{remark}
\begin{remark}\label{idproduct}
The isomorphism of Proposition \ref{product} carries the
logarithmic jacobian ideal of $k[t^{\Gamma\times\Gamma'}]$ onto
the tensor product of the logarithmic jacobian ideals of the
factors.
\end{remark}
\begin{remark}\label{dimone}
In the one-dimensional case the logarithmic jacobian ideal is the maximal ideal corresponding to the closed orbit. It is a classical fact that iterating the blowing-up of the singular point resolves the singularities of any branch.
\end{remark}

\section{The sheaf of logarithmic jacobian ideals on a toric variety} \label{sheaf}

Let the pair $(\Sigma, \Gamma)$ define a toric variety
$T_{\Sigma}^\Gamma$ as in Definition \ref{def-general-toric}.

On the affine open set $T^{\Gamma_\s}$, $\s \in \Sigma$  we consider the
 ideal $\J_\s$ of $k [t^{\Gamma_\s}]$  generated by
monomials of the form  $t^{\a}$, where $\a$ belongs to the set
\[
| \J_\s | = \{  \a_1 + \cdots  + \a_d \, \mid \,  \a_1, \dots,
\a_d \in \Gamma_\s \mbox{ and } \a_1 \y  \cdots \y \a_d \ne 0 \}.
\]
As we saw in Proposition \ref{Nashbl} the ideal $\J_\s$ is
called the \textit{logarithmic jacobian ideal} of $T^{\Gamma_\s}$.

\begin{remark} \label{gen}
If $\g_1, \dots, \g_r$ are generators of the semigroup
$\Gamma_\s$ then the monomials $t^\a$, for $\a$ in
\begin{equation} \label{set-j}
 \{  \g_{i_1} + \cdots +  \g_{i_d} \, \mid \, \g_{i_1}
 \y  \cdots \y \g_{i_d} \ne 0, \, 1 \leq  i_1, \dots, i_d \leq r
 \},
\end{equation}
generate the ideal  $\J_\s$. Abusing notation we denote the
set (\ref{set-j})  with the same letter $\J_\s$, whenever the set
of generators of $\Gamma_\s$ is clear from the context.
\end{remark}

\begin{proposition}\label{jacsheaf}
The family $\{ \J_\s \mid \s \in \Sigma \}$ defines a
$T^M$-invariant sheaf of ideals  $\J$ on $T_\Sigma^\Gamma$, which
is called the sheaf of logarithmic jacobian ideals of
$T_\Sigma^\Gamma$.
\end{proposition}

{\textit{Proof.}} It is sufficient to check that if $\t \leq \s$,
$\s \in \Sigma$, then the ideal $\J_\t$ coincides with the
extension $\J_\s k [t^{\Gamma_\t}]$, induced by the inclusion
$k[t^{\Gamma_\s}] \hookrightarrow k [t^{\Gamma_\t}]$ defined by
$\Gamma_\s \subset \Gamma_\t$.

By Lemma \ref{conditions-aff} if $m \in \Gamma_\s$ belongs to the
relative interior of the cone $\check{\s} \cap \t^\perp$ then we
have that $\Gamma_\t = \Gamma_\s + \Z_{\geq 0} (-m)$.

If  $\g_1, \dots, \g_r$ are generators of $\Gamma_\s$ then $\g_1,
\dots, \g_r, -m$ are generators of $\Gamma_\t$. This implies the
inclusion $\J_\s \subset \J_\t$. By Remark \ref{gen}  an exponent
${\a}$ in $\J_\t$ which does not belong to the set $\J_\s$ is of
the form: $ \a= \g_{i_1} + \cdots + \g_{i_{d-1}} - m, \mbox{ with
} \g_{i_1} \y  \cdots \y \g_{i_{d-1}}  \y (-m) \ne 0$. Then, the
element $\beta:= \g_{i_1} + \cdots + \g_{i_{d-1}} +  m$ belongs to
$\J_\s$ and we obtain that: $ t^{\a} = t^{-2m} t^{\b} \in \J_\s k
[t^{\Gamma_\t}], \mbox{ and } \J_\s k [t^{\Gamma_\t}] = \J_\t$.
\hfill $\ {\Box}$

   \begin{proposition} \label{jacoiso} The toric variety $T^\Gamma_\Sigma$ is non-singular
 if and only if the blowing up of the logarithmic jacobian ideal is an isomorphism.
\end{proposition}
\textit{Proof.}
We only have to prove that if the blowing up of the logarithmic jacobian ideal of an affine toric variety is an
isomorphism the variety is smooth.
We deal first with the case of a semigroup $\Gamma$
such that the cone $\check\s$ generated by $\Gamma$
is strictly convex, or equivalently that the cone $\s$ is of dimension $d ={\rm rank}M \geq 1$.
In this situation the semigroup $\Gamma$ has a unique
minimal system of generators  $\gamma_1, \ldots , \gamma_d,\gamma_{d+1},\ldots $  (see \cite{Ewald}, Chapter V, Lemma
3.5, page 155. The result is proved there for $\check\sigma\cap M$ but the same argument applies to $\Gamma$).  If there are more than $d$ generators, we may assume that the first $d$ generators are linearly independent. 
Then $\gamma_{d+1}$ is linearly dependent on the previous ones
which gives us another element $m=\gamma_1+\cdots
+\gamma_{i-1}+\gamma_{d+1}+\gamma_{i+1}+\cdots +\gamma_d$ of our
ideal. Our assumption ensures that
$m-m^{(1)}=\gamma_{d+1}-\gamma_i$ (or its opposite) is in $\Gamma$, which contradicts the
assumption of minimality since the appearance of $\gamma_{d+1}$ in the right hand side of an expression $\gamma_{d+1}-\gamma_i=\sum a_k\gamma_k,\ a_k\in \N$, would contradict the strict convexity of the cone $\check\sigma$ by implying either that some positive multiple of $-\gamma_{d+1}$ is in $\Gamma$ or that $-\gamma_i$ is in $\Gamma$. The same argument works for $\gamma_i-\gamma_{d+1}$. Therefore $\Gamma$ has $d$ independent
generators which generate $M$ and $k[t^\Gamma]$ is a polynomial
ring.\par
If the dimension of $\s$ is $< d$ we deduce from the assumption and  Proposition \ref{prev}, ii) below
that $M(\s, \Gamma) = M(\s)$. Then we reduce to the case $\dim \s = d$  by  Lemma \ref{qzero} below. Neither of those two results uses this proposition.
 \hfill $\ {\Box}$\par\medskip

\begin{lemma} \label{pw}
There is a continuous  piecewise linear function $\ord_{\J} \colon
|\Sigma | \rightarrow \R$ such that for each $\t \in \Sigma$ the
function $\ord_{\J_{\t}}$ is the restriction of $\ord_{\J}$ to
$\t$.
\end{lemma}
{\textit{Proof.}} This follows from the definition of
$\ord_{\J_\s}$ (see (\ref{ord})), by using that $\J$ is a sheaf of
monomial ideals.

\begin{remark}
Note that Lemma \ref{pw} holds more generally if we replace $\J$
by any sheaf of monomial ideals $\mathcal{I}$ on
$T_\Sigma^\Gamma$.
\end{remark}

The will need the following lemma  in Section \ref{iteration}.
\begin{lemma}\label{induced}
Let $\theta^\perp\cap\Gamma$ be a face of the finitely generated
semigroup  $\Gamma\subset M$. The logarithmic jacobian ideal
$\tilde\J$ of the image $\tilde\Gamma$ of $\Gamma$ in the lattice
$M/M(\theta)$ is equal to the image of the logarithmic jacobian
ideal $\J$ of $\Gamma$.
\end{lemma}
\textit{Proof} Let us denote by $\tilde\gamma_i$ the images in
$M/M(\theta)$ of the generators $\gamma_i$ of $\Gamma$ and by $p$
 the rank of the lattice $M/M(\theta)$. If $\tilde\gamma_{i_1}, \ldots
,\tilde\gamma_{i_{p}}$ are linearly independent in $M/M(\theta)$,
then $\gamma_{i_1}, \ldots ,\gamma_{i_{p}}$ must be linearly
independent from $\theta^\perp$.  Remark that
  since $\theta^\perp\cap\Gamma$ is a face
it must contain $d-p$ generators of $\Gamma$ which are linearly
independent, since  $\theta^\perp\cap\Gamma$ spans the rank $d-p$
lattice $M(\theta, \Gamma)$.  Choosing linearly independent
generators $\gamma_{i_{p+1}} , \dots,
\gamma_{i_d}\in\theta^\perp\cap\Gamma$ of $\Gamma$ gives us a
generator $\gamma_{i_1}+ \cdots +\gamma_{i_{p}}+ \gamma_{i_{p+1}}
+ \cdots + \gamma_{i_d}$ of $\J$ whose image is
$\tilde\gamma_{i_1}+ \cdots +\tilde\gamma_{i_{p}}$, showing that
$\tilde\J$ is contained in the image of $\J$. If we now take $d$
independent generators $\gamma_{i_1}, \ldots
,\gamma_{i_{p}},\gamma_{i_{p+1}} , \dots, \gamma_{i_d}$ of
$\Gamma$, since they generate $M$, there must exist $d-p$
independent elements in their images, say $\tilde\gamma_{i_1},
\ldots ,\tilde\gamma_{i_{p}}$. Then the image $\tilde\gamma_{i_1}+
\cdots +\tilde\gamma_{i_{p}}+\tilde\gamma_{i_{p+1}} + \cdots +
\tilde\gamma_{i_d}$ belongs to the logarithmic jacobian ideal
$\tilde\J$, which shows that the image of $\J$ is equal to
$\tilde\J$.\hfill $\ {\Box}$

\section{Iterating the blowing-up of the logarithmic jacobian ideal}\label{iteration}

Let $\Gamma \subset M$ a finitely generated subsemigroup of a rank
$d$ lattice $M$ such that $\Z \Gamma =M$. We assume in addition
that the convex rational cone $\check{\sigma}: = \R_{\geq 0} \Gamma$, which is $d$-dimensional since $\Z \Gamma =M$, is strictly convex, which is equivalent to saying that the dual cone $\s \subset N_\R$ is strictly convex
of dimension $d$. The semigroup $\Gamma$ determines the affine
toric variety $T^{\Gamma} = \mathrm{Spec } k [ t^{\Gamma} ]$.
We fix a finite set of generators $\g_1, \dots, \g_r$ of $\Gamma$.
We consider the set
\[
\J := \{ \g_{i_1} + \cdots + \g_{i_d} \mid \g_{i_1} \y \dots \y
\g_{i_d} \ne 0, \, 1 \leq i_1, \dots, i_d\leq r  \}
\] defining the logarithmic jacobian ideal of $T^\Gamma$.

The Newton polyhedron $\Newton_\s (\J)$ of the monomial ideal
$\J$ (see Section \ref{blow}),  is contained in the interior of $\check{\s}$, since
the elements of $\J$ are sums of $d$-linearly independent elements
in the $d$-dimensional cone $\check{\s}$. The set $\J$ determines
the \textit{order function} defined by (\ref{ord}).
The maximal cones  $\t \subset \s$ of linearity of the function
$\ord_{\J}$ form the $d$-skeleton of a fan $\Sigma$ supported on $\s$.
The map
\begin{equation} \label{m-tau}
\t \mapsto m  \mbox{ if } \ord_\J (\nu) = \langle \nu, m \rangle
\mbox { for all } \nu \in \t.
\end{equation} is a bijection between the set $\Sigma(d)$ of $d$-dimensional cones of $\Sigma$ and the set of vertices of the polyhedron $\Newton_\s (\J)$.

We now consider the blowing up of the monomial ideal $\J$. A cone
$\t^{(1)} \in \Sigma (d)$ determines a vertex $m^{(1)}$ of
$\Newton_\s (\J)$ by (\ref{m-tau}) and also the finitely generated semigroup
\[
\Gamma_{\t^{(1)} }^{(2)} : = \Gamma + \sum_{m \in \J} \Z_{\geq 0 }
(m-m^{(1)})\subset \check\t^{(1)}\cap M.
\]
In view of the description recalled above of $\Sigma(d)$ in terms of $\Newton_\s (\J)$, the cone $\R_{\geq 0} \Gamma_{ \t^{(1)} }^{(2)} $ is
$\check{\t}^{(1)}$. The affine toric variety $T^{ \Gamma_{\t^{(1)}}^{(2)}}$ is a chart of the blowing up of $\J$ and this
toric variety is covered by charts of this form (see Section \ref{blow}).
    The semigroup $\Gamma_{\t^{(1)}}^{(2)}$ is generated by $\{ \g_1,
\dots, \g_r\} \cup \{  m -m^{(1)} \}_{m \in \J}$.

\begin{lemma}    \label{repre-gen}
If we choose a representation $m^{(1)} = \g_{i_1} + \cdots + \g_{i_d}$ as a sum of linearly independent vectors in $\Gamma$
then the semigroup $\Gamma_{\t^{(1)}}^{(2)}$ is generated by
 \begin{equation} \label{set-gen}
 \{ \g_{i_1}, \dots, \g_{i_d} \} \cup \{ \g_l - \g_{i_s} \mid 1 \leq s \leq d,
 1\leq l \leq r, l \ne i_s, \g_l \y \bigwedge_{j=1, \dots, d}^{j\ne i_s}  \g_{i_j}  \ne 0 \}.
\end{equation}
\end{lemma}
\textit{Proof.}    To simplify the notations we can assume that  $i_s = s$ for $s =1, \dots,d$.
First notice that if $\g_l$, $1 \leq l \leq r$,
has the property that $\g_l \y \bigwedge_{i=1, \dots, d}^{i\ne i_0} \g_i$ then the vector
$m' = \g_l + \sum_{i=1, \dots, d}^{i\ne i_0} \g_i $ belongs to $\J$  hence
$m' - m^{(1)} = \g_l - \g_{i_0}$ is a generator of   $\Gamma_{\t^{(1)}}^{(2)}$. Since
$\g_1, \dots, \g_d \in  \Gamma_{\t^{(1)}}^{(2)}$ by construction we get that
$\g_l \in      \Gamma_{\t^{(1)}}^{(2)}$ is not in the minimal set of generators of the semigroup
$\Gamma_{\t^{(1)}}^{(2)}$.

Suppose now that $\g_{p_1}, \dots, \g_{p_d}$ are linearly independent vectors.
We denote by $s_0$ the integer such that $\g_{p_1}, \dots, \g_{p_{s_0}} \in \{ 1, \dots,d \}$ and
$\g_{p_{s_0 +1}}, \dots, \g_{p_d} \in \{ d+1, \dots, r\}$. We can permute the vectors $\g_{p_{s_0 +1}}, \dots, \g_{p_d}$ in such
a way that the $i$-th coefficient of the expansion of $\g_{p_i}$ in terms of the basis $\g_1, \dots,\g_d$ of $N_\Q$, is non-zero for
$i = s_0 +1, \dots, d$ (otherwise we would get that       $\g_{p_1}\y \cdots \y \g_{p_d} =0 $  contrary to the assumption).
Then the vector $m= \sum_{i=1}^d \g_{p_i}$ belongs to $\J$ and we deduce that
$m - m^{(1)} = \sum_{i= s_0 +1}^d (\g_{p_i} - \g_i)$ hence the vectors
(\ref{set-gen}) generate the
semigroup      $\Gamma_{\t^{(1)}}^{(2)}$.
\hfill $\ {\Box}$

We denote also  by
$\J_{\t^{(1)}}^{(2)}$ the finite subset of
$\Gamma_{\t^{(1)}}^{(2)}$ corresponding to the monomials
generating the logarithmic jacobian ideal of
$T^{\Gamma_{\t^{(1)}}^{(2)}}$, by the same symbol this last ideal of $k[t^{\Gamma_{\t^{(1)}}^{(2)}}]$ , and by
$ \ord_{\J_{\t^{(1)}}^{(2)}} \colon \t^{(1)} \rightarrow \R $ the
corresponding order function.
\begin{remark} \label{principal}
  On the chart $T^{  \Gamma_{ \t^{(1)} }^{(2)}  }$
the pull back of the ideal $\J$ by the blowing up of $\J$ is the
principal ideal $t^{m^{(1)}} k [ t^{\Gamma_{ \t^{(1)} }^{(2)}}] =
t^\J k [ t^{\Gamma_{ \t^{(1)} }^{(2)}}]$. The Newton polyhedron
\[
\Newton_{\t^{(1)}} ( \J) := \J +  \check{\t}^{(1)} = m^{(1)} +
\check{\t}^{(1)}
\]
 of $t^\J k [ t^{\Gamma_{\t^{(1)} }^{(2)}}]$ is \textit{principal}, i.e.,
 it has only one vertex $m^{(1)}$.
\end{remark}

\begin{lemma} \label{global}
There is a continuous  piecewise linear function $\ord_{\J^{(2)}} \colon
\s \rightarrow \R$ such that for each $\t^{(1)}\in \Sigma(d)$
the function
$\ord_{\J_{\t^{(1)}}^{(2)}}$ is the restriction of $\ord_{\J^{(2)}}$ to
$\t^{(1)}$.
\end{lemma}
\textit{Proof.} This follows from Lemma \ref{pw}.
 \hfill $\
{\Box}$

As above, the maximal cones  $\t \subset \s$ of linearity of the function
$\ord_{\J^{(2)}}$ form the $d$-skeleton of a fan $\Sigma^{(2)}$
supported on $\s$ and subdividing the fan $\Sigma$.
 In particular, if
 $\t^{(2)}  \in \Sigma^{(2)} (d)$ is contained in $ \t^{(1)}   \in \Sigma (d)$
 then we denote by $m^{(2)}$ the vertex of the
Newton polyhedron  $\Newton_{\t^{(1)}} ( \J_{\t^{(1)}}^{(2)} ) $
of $\J_{\t^{(1)}}^{(2)}$ such that
\[
\ord_{\J^{(2)}}  (\nu) =  \langle \nu, m^{(2)} \rangle \mbox{ for
all } \nu \in \t^{(2)}.
\]

By iterating this construction we obtain a sequence of piecewise
linear functions $\ord_{\J^{(j)}}$ on $\s$, together with the
corresponding  fans $\Sigma^{(j)}$, with $\J = \J^{(1)}$ and
$\Sigma^{(1)} = \Sigma$, and such that $\Sigma^{(j)}$ is a
subdivision of $\Sigma^{(j-1)} $ for all $j \geq 2$.

By definition a cone $\t^{(j)} \in \Sigma^{(j)}(d)$ is contained
in a unique cone $\t^{(l)} \in \Sigma^{(l)} (d)$, for $0 \leq l
\leq j-1$, where we set $\t^{(0)} := \s$. Then we have unique
vectors $m^{(l)} \in M$ such that
\[
\ord_{\J^{(l)}}  (\nu) =  \langle \nu, m^{(l)} \rangle \mbox{ for
all } \nu \in \t^{(j)} \mbox{ and } 1 \leq l \leq j.
\]

The cone $\t^{(j)}$ corresponds to a chart of the  blowing up of
the logarithmic jacobian ideal $\J^{(j)}_{\t^{(j-1)}}$ of $k [ t
^{ \Gamma_{\t^{(j-1)}}^{(j)}}]$. This chart is the affine toric
variety defined by the semigroup
\[
\Gamma_{\t^{(j)}}^{(j+1)} = \Gamma_{\t^{(j-1)}}^{(j)} + \sum_{m
\in \J^{(j)}_{\t^{(j-1)}} } \Z_{\geq 0} (m - m^{(j)} ).
\]
By induction this procedure also provides a system of generators
of each semigroup $\Gamma_{\t^{(j)}}^{(j+1)}$. We  use also the
notation $\J_{\t^{(j)}}^{(j+1)}$ to refer to the finite set of
generators of the logarithmic jacobian ideal of $k [ t^{
\Gamma_{\t^{(j)}}^{(j+1)} }]$ (see Remark \ref{gen}). The
following inclusions, for $j\geq 2$, are consequence of the
definitions:
\begin{equation} \label{essential}
\Gamma_ {\t^{(j-1)}} ^{(j)} \subset \Gamma_{\t^{(j)}}^{(j+1)},
\quad  k[t^{\Gamma_ {\t^{(j-1)}} ^{(j)}}]\subset k[t^{
\Gamma_{\t^{(j)}}^{(j+1)}}], \quad \J_{\t^{(j-1)}}^{(j)} k[t^{
\Gamma_{\t^{(j)}}^{(j+1)}}] \subset \J_{\t^{(j)}}^{(j+1)} k[t^{
\Gamma_{\t^{(j)}}^{(j+1)}}].
\end{equation}
By
(\ref{essential}) we have that
\begin{equation} \label{essential2}
\ord_{\J^{(j+1)}} (\nu) \leq  \ord_{\J^{(j)}} (\nu)  \mbox{ for
all } \nu \in \s.
\end{equation}

\begin{remark} \label{principal-k}
For $1 \leq l \leq j$ we deduce from Remark \ref{principal} that $
\J^{(l)}_{\t^{(l-1)}}  k [ t^{ \Gamma^{(l+1)}_{\t^{(l)} }} ] =
t^{m^{(l)}}  k [ t^{ \Gamma^{(l+1)}_{\t^{(l)} } } ]$, hence the
Newton polyhedron
 $\Newton_{\t^{(l)}} ( \J^{(l)}_{\t^{(l-1)}} ) =
\J^{(l)}_{\t^{(l-1)}} + \check{\t}^{(l)}  =   m^{(l)} +
\check{\t}^{(l)}   $ has only one vertex $m^{(l)}$. \end{remark}

\begin{notation} We denote the Newton polyhedron $\Newton_{\t^{(j-1)}} ( \J^{(j)}_{\t^{(j-1)}} )$ simply by $\Newton( \J^{(j)}_{\t^{(j-1)}} )$ since there is no risk of confusion.
\end{notation}

\begin{proposition} \label{finite}
The following assertions are equivalent:
\begin{enumerate}
\item[{\rm i.}] $\t^{(j)} = \t^{(j-1)}$

\item[{\rm ii.}] The blowing up of the ideal $\J _ {\t^{(j-1)}} ^
{(j)} $ of $T^{ \Gamma_ {\t^{(j-1)}} ^ {(j)} }$ is a finite
morphism.
\end{enumerate}

\end{proposition}
\textit{Proof.} The hypothesis {\rm i.} is equivalent to the
following fact: the semigroups $\Gamma_ {\t^{(j-1)}} ^ {(j)}$ and
$\Gamma_ {\t^{(j)}} ^ {(j+1)}$ have the same saturation in the
lattice $M$; it is equal to $\check{\t}^{(j-1)} \cap M   =
\check{\t}^{(j)} \cap M $. This is equivalent to the following
geometric statement: the composite of the normalization of $T^{
\Gamma_ {\t^{(j)}} ^ {(j+1)} }$ with the blowing up of the
logarithmic jacobian ideal of $T^{ \Gamma_ {\t^{(j-1)}} ^ {(j)} }$
is the normalization map of $T^{ \Gamma_ {\t^{(j-1)}} ^ {(j)} }$
and therefore this blowing up is finite. Conversely, if {\rm ii.}
holds, the blowing up morphism $T^{ \Gamma_ {\t^{(j)}} ^ {(j+1)}
}\to T^{ \Gamma_ {\t^{(j-1)}} ^ {(j)} }$ induces an isomorphism of
the normalizations, from which {\rm i.} follows in view of Remark
\ref{norma}. \hfill $\ {\Box}$
\begin{remark}\label{vertex}
The conditions of the Lemma are also equivalent to the fact that the Newton polyhedron of the ideal $\J _ {\t^{(j-1)}} ^
{(j)} $ has only one vertex $m^{(j)}$.
\end{remark}
\begin{definition} \label{function}
For any integer $j \geq 1$ we introduce  a function
\[
f^{(j)} : \{ \t \subset \s \mid  0 \ne \t \mbox { convex rational
polyhedral cone } \} \rightarrow \Z_{\geq 1}.
\]
If $\nu_1, \dots, \nu_s$ are the primitive integral vectors for
the lattice $N$ which span the edges of $\t$, then the value of
$f^{(j)}(\t)$ is defined by
\[
f^{(j)} (\t) := \sum_{i=1}^s \ord_{\J^{(j)}} (\nu_i).
\]
\end{definition}

\begin{remark} \label{key}
Notice that if $0 \ne \t$ is any rational  polyhedral cone
contained in $\t^{(j)} \in \Sigma^{(j)} (d)$ then $f ^{(j)} (\t) =
\sum_{i=1}^s \langle \nu_i , m^{(j)} \rangle$ and if $\t$ is of
dimension $d$ then $f ^{(j)} (\t) \geq d$. Moreover, by
(\ref{essential2}) we obtain that
\begin{equation} \label{ine}
f^{(j)} (\t) \leq  f^{(j-1)} ( \t).
\end{equation}
\end{remark}

\begin{lemma} \label{smooth}
The following conditions are equivalent for $j\geq 1$:
\begin{enumerate}
\item[{\rm i.}]  The equality $f^{(j)} ( \t^{ (j-1) }) = d$ holds.

\item[{\rm ii.}] The cone $\t^{ (j-1) }$ is regular for the
lattice $N$ and $\Gamma ^{(j)}_{\t^{ (j-1) } }  = \check{\t}^{
(j-1) } \cap M$.

\item[{\rm iii.}] The toric variety $T^{\Gamma^{(j)}_{\t^{ (j-1)
}}} $ is smooth.
\end{enumerate}
\end{lemma}
Note that if the conditions of the Lemma are satisfied, the polyhedron $\Newton_{\t^{(j-1)}}
(\J^{(j)}_{\t^{ (j-1) } } )$ has only one vertex $m^{(j)}$.\par\noindent
\textit{Proof.} It is clear that {\rm ii.} and {\rm iii.} are
equivalent. It is enough to prove the result for $j=1$. Suppose
first that {\rm i.} holds. By hypothesis the fan $\Sigma^{(1)}$ is
the fan consisting of the faces of $\s$. If $\nu_1, \dots, \nu_s$
are the primitive vectors for the lattice $N$ which span the cone
$\s $ then $\langle \nu_i, m \rangle
>0$, $i=1, \dots, s$ since $m=m^{(1)} $ belongs to the interior of
$\check{\s}$. Since $f(\s) = d = \sum_{i=1}^s \langle \nu_i, m
\rangle$ we get that $s= d$ and $\langle \nu_i, m \rangle = 1$.

By definition of $\J$ the vector $m$ is sum of $d$ generators of
$\Gamma$ which are linearly independent, say $m = \g_1 + \dots +
\g_d$. Since $\sum_{j=1}^d \langle \nu_i, \g_j \rangle = 1$ for
$i=1, \dots, d$ we obtain that, up to relabeling the  $\nu_i$,
the vectors  $\nu_1, \dots, \nu_d$ in $N_\R$ form the dual basis
of $\g_1, \dots, \g_d$ in $M_\R$. Finally, notice that the
parallelogram generated by the  primitive vectors $\g_1, \dots,
\g_d$  in $M_\R$ contains no integral points different from the
vertices. It follows that $\g_1, \dots, \g_d$ form a basis of $M$.

 Conversely, if {\rm ii.} holds then we check from the
 definitions that
{\rm i.} holds.
 \hfill $\ {\Box}$

\begin{proposition} \label{eq} Suppose that $\t^{(j)} \in \Sigma^{(j)} (d)$ is
contained in $\t^{(j-1)} \in \Sigma^{(j-1)} (d)$. The following
equalities are equivalent:
\begin{enumerate}
\item[{\rm i.}] $f^{(j)} ( \t^{(j)} ) = f^{(j-1)} ( \t^{(j)} )$,

\item[{\rm ii.}] $m^{(j)} = m^{(j-1)}$.
\end{enumerate}
\end{proposition}
\textit{Proof.} Notice that if $m^{(j)} = m^{(j-1)}$ then {\rm i.}
follows by Remark \ref{key}.

Suppose that the equality {\rm i.} holds. By Remark
\ref{principal-k} we have that
\[
\Newton_{ \t^{(j-1)} } ( \J^{(j-1)}_{\t^{(j-2)}} ) = m ^{(j-1)} +
\check{\t}^{(j-1)} \mbox{ and } \Newton_{ \t^{(j)} } (
\J^{(j)}_{\t^{(j-1)}} ) = m ^{(j)} + \check{\t}^{(j)}.
\]
Since $\t^{(j)}$ is contained in $\t^{(j-1)}$ we get that
$\check{\t}^{(j-1)} \subset  \check{\t}^{(j)} $ and then
$
\Newton_{ \t^{(j)} } ( \J^{(j-1)}_{\t^{(j-2)}} ) = m ^{(j-1)} +
\check{\t}^{(j)}
$.
By (\ref{essential}) we get
\begin{equation} \label{inc}
 m ^{(j-1)} +
\check{\t}^{(j)} \subset  m ^{(j)} + \check{\t}^{(j)}.
\end{equation}
Let $\nu_1, \dots, \nu_s$ be the primitive integral vectors for
the lattice $N$ which span the cone $\t^{(j)}$. The vector $\nu: =
\sum_{i=1}^s \nu_i$ belongs to the interior of the cones
$\t^{(j)}$ and $\t^{(j-1)}$. By Remark \ref{key} and the
hypothesis we deduce
\[
f^{(j-1)} ( \t^{(j)} ) = \langle \nu, m^{(j-1)} \rangle = f^{(j)}
( \t^{(j)} ) = \langle \nu, m^{(j)} \rangle.
\]
This equality and the inclusion (\ref{inc}) imply that
$m^{(j-1) } = m^{(j)}$.
 \hfill $\ {\Box}$

\begin{proposition} \label{bound}
There exists an integer $l \geq 1$ such that for any cone $\t \in
\Sigma^{(l)} (d) $ if  $f^{(1)} (\t) > d$ then $f^{(1)} (\t)
> f^{(l)} (\t)$.
\end{proposition}
\textit{Proof.} Let us assume that the assertion of the
Proposition does not hold. This implies that there exists a
infinite sequence of convex polyhedral cones
\begin{equation} \label{sequence}
\s = \t^{(0)} \supset \t^{(1)} \supseteq  \t^{(2)} \supseteq
\cdots \supseteq \t^{(j)} \supseteq \cdots,
\end{equation}
such that $\t^{(j)} \in \Sigma^{(j)} (d)$ and
\begin{equation} \label{assumption}
 f^{(j)} (\t^{(j)}) = f^{(1)} (\t^{(j)} )
> d  \mbox{,  for all }j \geq 2.
\end{equation}

By Remark \ref{key} we have that $f^{(j)} (\t^{(j)}) = f^{(j-1)}
(\t^{(j)})$ for all $j \geq 2$. Proposition \ref{eq} implies  then
that $m^{(j)} = m^{ (j-1)}$ for all $j \geq 2$.

\begin{claim} \label{strict}
There exists a strictly increasing sequence $(i_j)_{j \geq 1}$ of
integers $\geq 0$ such that $\t^{(i_{j})} \ne \t^{(i_{j}+1)}$,
that is, the inclusion $\t^{(i_{j})} \supset \t^{(i_{j}+1)}$ is
strict, for $j \geq 1$.
\end{claim}
\textit{Proof of the claim.} Assume that the claim does not hold.
This implies that $\t^{(j)} = \t^{(j-1)}$ for all $j \geq 1$. By
Proposition \ref{finite} the  blowing up of the ideal $\J _
{\t^{(j-1)}} ^ {(j)} $ of $T^{ \Gamma_ {\t^{(j-1)}} ^ {(j)} }$ is
a finite morphism, dominated by the normalization of $T^\Gamma$,
 for all $j \geq 1$. It follows that for $j \gg
0$ the variety $T^{ \Gamma_
{\t^{(j-1)}} ^ {(j)} }$ is normal. By Proposition \ref{jacoiso}, this variety is also smooth.
 By Lemma
\ref{smooth} it follows that $f^{(j)} (\t^{(j)}) = d$ for $j \gg
0$. This is a contradiction with (\ref{assumption}). \hfill $\
{\Box}$

\medskip

 Let us fix a
representation for $m = m^{(1)}$ in terms of the generators of $\Gamma$:
\[
m  = \g_1 + \cdots  + \g_d \mbox{ with } \g_1 \y \cdots \y \g_d
\ne 0,
\]
(up to an eventual relabeling  of the generators $\{\g_i
\}_{i=1}^r$ of the semigroup  $\Gamma$).

By Claim \ref{strict} we can suppose without loss of generality
that $i_1 =0$, that is, the Newton polyhedron $\Newton_\s (\J)$
has at least two different vertices $m$ and $n$.

\begin{lemma}\label{end}
Given one of the $\gamma_j$ which appear in the decomposition of
$m$, say $\gamma_d$, for any $j \geq 0$ the vector $n_j := n -
j \g_d$  has the property that
\[
n_j  \in \J_{\t^{(j)}}^{(j+1)} \mbox{ and } (n_j -m) \y \g_1
\y  \cdots \y \g_{d-1} \ne 0.
\]
\end{lemma}
\textit{Proof.} We prove the assertion by induction on $j$. Notice
that for $j=0$ the vector $n_0 = n$ belongs to $\J$ by hypothesis.
We suppose by induction that $n_l \in \J_{\t^{(l)}}^{(l+1)}$, $1
\leq l \leq j$.

We  prove first that:
\begin{equation} \label{independent}
(n_l -m) \y \g_1  \y \cdots \y \g_{d-1} \ne 0 \mbox{,  for } 0
\leq l \leq j.
\end{equation}

Assume on the contrary that (\ref{independent}) does not hold for
some $0 \leq l \leq j$.  After relabeling the vectors $\g_1,
\dots, \g_{d-1}$ if necessary, we have an expansion of the form:
\begin{equation} \label{decomposition}
n_l -m = a_1 \g_1 + \cdots + a_h \g_h \mbox{ with } h \leq d-1,
\end{equation}
and in addition the coefficients of (\ref{decomposition}) are
non-zero rational numbers which are not of the same sign, that is,
\[
\left\{
\begin{array}{ccccccl} a_i & > & 0 & \mbox{for} & i & = & 1, \dots, s
\\
a_i & < &  0 & \mbox{for} & i &  = & s+1, \dots, h.
\end{array}
\right.
\]
Indeed, if all coefficients $a_i$ in (\ref{decomposition}) are
$\geq 0$  we obtain that
\[
n = m + a_1 \g_1 + \cdots + a_h \g_h + l \g_d \subset m + \check\s,
\]
contradicting that $n \ne m$ is a vertex of $\Newton_\s (\J)$. In
particular, we have that $n_l \ne m$. Similarly, if all the
coefficients $a_i$ are smaller than zero we get that
\[
m = n_l -  a_1 \g_1 - \cdots - a_h \g_h \subset  n_l +
\check\s \subset n_l +  \check{\t}^{(l)}.
 \]
This implies that $m$ is not a vertex of the Newton polyhedron
$\Newton_{\t^{(l)}}  (\J_{\t^{(l)}}^{(l+1)} )$,  since $n_l \in
\J_{\t^{(l)}}^{(l+1)}$ and $n_l \ne m$.

If $\a \in \R$ we denote by  $\lceil \a \rceil$ the smallest
integer $p$ such that $\a \leq p$.

\begin{claim} \label{Q}
If ${q} := \sum_{i=1}^s \lceil a_i \rceil$, $0 \leq {p} \leq q$
and $b_i$ are integers such that $0 \leq b_i \leq \lceil a_i
\rceil$, $i=1, \dots, s$ and $\sum_{i=1}^s b_i = {p}$ then the
vector $\beta_{l,{p}} := n_l - \sum_{i=1}^s b_i \g_i$ belongs
to $\J_{\t^{(l+{p} )}}^{(l+{p}+1)}$.
\end{claim}
\textit{Proof of the claim.}  We prove the assertion by induction
on ${p}$. For ${p} =0$ we have $\beta_{l, 0} = n_{l}$ hence
the assertion holds by assumption. Suppose that $\beta_{l,{p}}
\in \J_{\t^{(l+{p} )}}^{(l+{p}+1)}$ for $0 \leq {p} < {q}$. The
vector
\[
\beta_{l,{p}}  - m = n_l - m - \sum_{i=1}^s b_i \g_i =
\sum_{i=1}^s (a_i - b_i) \g_i + \sum_{i =s+1}^h a_i \g_i
\]
belongs to $\Gamma_{\t^{(l+{p}+1 )}}^{(l+{p}+2)}$. Since ${p} <
{q}$ there is a strictly positive coefficient in this expansion of
$\beta_{l,{p}}  - m $, say $a_1 - b_1$, for instance. We get $
(\beta_{l,{p}}-m) \y \g_2 \y \cdots \y \g_d \ne 0$, hence the vector
\[ \beta_{l,{p}+1} := \beta_{l,{p}} - m + \g_2 + \cdots + \g_{d} =
\beta_{l,{p}} - \g_1
\] belongs to $\J_{\t^{(l+{p}+1 )}}^{(l+{p}+2)}$.  \hfill $\ {\Box}$

\medskip

By Claim \ref{Q}  the expansion
\[
\beta_{l, {q}} - m = \sum_{i=1}^s  (a_i - \lceil a_i \rceil) \g_i
+ \sum_{i = s+1}^h a_i \g_i
\]
has only coefficients $\leq 0$ and  $\beta_{l, {q}} \in
\J_{\t^{(l+{q}+1 )}}^{(l+{q}+2)}$. We get also that $m \ne
\beta_{l, {q}}$ since the coefficients $a_{s+1}, \dots, a_h$ are
non-zero and the vectors $\g_1, \dots , \g_h$ are linearly
independent. We deduce from this that
\[
m = \beta_{l, {q}} - \sum_{i=1}^s  (a_i - \lceil a_i \rceil) \g_i
-  \sum_{i = s+1}^h a_i \g_i \, \in \,  \beta_{l, {q}} + \check\s\, \subset \, \beta_{l, {q}} + {\check{\t}^{(j+{q}+1 )}}.
\]
This contradicts the assumption, $m$ being a vertex of the Newton
polyhedron of $\J_{\t^{(l+{q}+1 )}}^{(l+{q}+2)}$.

Finally, we have proven that (\ref{independent}) holds hence
\[ n_{j+1} = (n_j -m) + \g_1 +  \cdots +  \g_{d-1}  \in  \J_{\t^{(j+1)}}^{(j+2)}.\]
 This concludes the induction in the proof of Lemma \ref{end}. \hfill $\ {\Box}$

\medskip

The cone
\[
\t^{(\infty)} = \bigcap_{l \geq 1} \t^{(l)} = \bigcap_{j \geq 1}
\t^{(i_j)}
\]
is a closed convex subset of $\s$ different from $ 0 $. A vector
$0 \ne w \in \t^{(\infty)}$ defines a \textit{monomial valuation}
$\w$ of the field of fractions of $k [t^\Gamma]$, which verifies
that if $ 0 \ne \sum a_\g t^\g \in k [t^\Gamma]$ then
$
\w (  \sum a_\g t^\g  ) = \min_{a_\g \ne 0}  \langle w, \g
\rangle
$.
By definition this valuation is non-negative in the subrings $k [
t^{\Gamma_ {\t^{(j-1)}} ^ {(j)}}]$ for all $j \geq 1$. Notice that
the vector $w \in N_\R$ is not necessarily an element of $N_\Q$
and it may lie in a face of $\s$ (different from $0$).
We remark that for all $j\geq 1$ we have that
$ \hbox{\rm min}  \{ \langle w,\gamma\rangle   \mid   \gamma\in
\J^{(j)}_{\t^{(j-1)}}
   \} = \langle w,m\rangle  $
 since $w$ takes non negative values on $\J^{(j+1)}_{\t^{(j)}}$
 which contains the  set $\{\gamma-m\mid \gamma\in\J^{(j)}_{\t^{(j-1)}} \}$.

Since $\g_1 \dots, \g_d$ span $M_\R$ at least one of the vectors
$\g_i$ verifies that  $\langle w, \g_i \rangle \ne 0$. Suppose for
instance that $\langle w, \g_d \rangle > 0$.

By Lemma \ref{end}, for any integer $j\geq 0$  the vector $ n_j
= n - j \g_{d}$ belongs to $\J_{\t^{(j)}}^{(j+1)}  \subset
\Gamma_{\t^{(j)}}^{(j+1)} $. This implies that
\[
\w( t^{n_j}) = \langle w, n_j \rangle = \langle w, n \rangle
- j \langle w, \g_d \rangle
\]
becomes strictly negative for $j$ large enough. This is a
contradiction since $t^{n_j} \in k [
t^{\Gamma_{\t^{(j)}}^{(j+1)}}]$ and the valuation $\omega$ is non
negative on  the ring $k [ t^{\Gamma_{\t^{(j)}}^{(j+1)}}]$. \hfill
$\ {\Box}$

\begin{corollary} \label{bound2}
With the previous notations, given any sequence of the form
(\ref{sequence}) if $T^{\Gamma}$ is not smooth it is not possible
that $m^{(1)} = m^{(j)}$ for all $j \geq 2$.
\end{corollary}
\textit{Proof.} This is now a consequence of Proposition \ref{eq}.

\begin{definition} \label{def-sta}\par\noindent
\begin{enumerate}
\item[i.] If $0 \ne \eta \subset \s$ is a cone we denote by
$\nu_{\eta}$ the sum of the primitive  vectors, for the lattice
$N$
 in the edges of the cone $\eta$.
\item[ii.]A cone $\eta \subset \s$ is \textit{stable} if there is
an integer $I \geq 1$ such that $\eta \in \Sigma^{(j)}$ for all $j
\geq I$.

\item[iii.]  The \textit{stability problem} for the toric variety
$T^\Gamma$ consists of determining if the sequence of fans
$(\Sigma^{(j)})_{j\geq 0}$ stabilizes, that is $\Sigma^{(j)} =
\Sigma^{(j+1)}$ for $j \gg 0$.
\item[iv.]  If $\theta \in
\Sigma^{(l)}$ the \textit{stability problem} for the cone $\theta$
consists of determining if the sequence of fans,
 $\{ \theta^{(j)} \in \Sigma^{(j)} \mid \theta^{(j)} \subset   \theta \}$ $j \geq l$, stabilizes.
\end{enumerate}
\end{definition}
For instance,  if $\r \in \Sigma^{(j)}$ is of dimension one then  $\r$ is stable.

\begin{lemma} \label{qzero}
If  $\theta    \in \Sigma^{(j-1)}$ is a cone of codimension $>0$ and
 $M ( \theta, \Gamma^{(j)}_\theta) = M(\theta)$ (see Notation \ref{ind}), then the stability problem
for the cone $\theta \subset N_\R$ is equivalent to a  stability problem for the cone $\theta$, viewed in $(N_\theta)_\R$,
with respect to the sequence
of iterated Semple-Nash modifications of another toric variety of dimension equal to $\dim \theta$.
\end{lemma}
\textit{Proof}.
The sublattice $M (\theta) = M \cap \theta^\perp$ of $M$ is obviously saturated, so that
it is a direct summand of $M$. Let us consider a sublattice $M'$
of $M$ such that $M = M (\theta) \oplus M'$. Such a sublattice
$M'$ is spanned by vectors $v_{q_0 +1 }, \dots, v_d$, completing a
basis $v_1, \dots, v_{q_0}$ of $M (\theta)$ to a basis of $M$. Any
$\g \in M$ can be written in a unique way as $\g = \a_1(\g)  +
\a_2(\g) $ with $\a_1 (\g)  \in M (\theta)$ and $\a_2 (\g) \in
M'$. The restriction $\beta \colon M' \mapsto M/ M (\theta)$ of
the canonical map $M \to M/M (\theta)$ to the sublattice $M'
\subset M$ is an isomorphism. The image
$\tilde{\Gamma}_{\theta}^{(j)}$ of ${\Gamma}_{\theta}^{(j)}$ by the canonical map $M \mapsto M/ M (\theta)$ is
a semigroup of finite type, generates the lattice $M/ M (\theta)$
and spans a strictly convex cone $\R_{\geq 0}
\tilde{\Gamma}_{\theta}^{(j)}$.

The semigroup  $\Gamma^{(j)}_\theta  \cap  \theta^{\perp}$, which  is the minimal face of
$\Gamma^{(j)}_\theta$, is a rank $q_0$ lattice by Lemma \ref{conditions-aff}.
Using the equality    $\Gamma^{(j)}_\theta  \cap  \theta^{\perp} = M(\theta)$
one checks directly that the map
\begin{equation}       \label{spliting}
{\Gamma}_{\theta}^{(j)} \longrightarrow  M (\theta) \times
\tilde{\Gamma}_{\theta}^{(j)},  \quad  \g \mapsto (\a_1
(\g), \beta \circ \a_2 (\g) )
\end{equation}
is a semigroup isomorphism.
 The map (\ref{spliting}) determines an
isomorphism $ T^ {{\Gamma}_{\theta}^{(j)}}
\stackrel{\simeq}{\longrightarrow } \orb (\theta,
{\Gamma}_{\theta}^{(j)} ) \times {
T^ {\tilde{\Gamma}_{\theta}^{(j)}}} $.
 In particular, the
variety $T^{{\Gamma}_{\theta}^{(j)}}$ is smooth if and
 only if $T^{\tilde{\Gamma}_{\theta}^{(j)}}$ is smooth.

 According to Remark
\ref{idproduct} and Lemma \ref{induced} the blowing up of
logarithmic jacobian ideals commutes with the splitting defined by
(\ref{spliting}).      The assertion follows from this since the cone $\theta$,  viewed in the $\R$-linear subspace
$(N_\theta)_\R$ it spans in $N_\R$, is the dual cone of $\R_{\geq 0}
\tilde{\Gamma}_{\theta}^{(j)}$.
\hfill $\ {\Box}$

 \begin{remark}
Geometrically, we see that the semigroup   $\tilde\Gamma_{\theta}^{(j)}$  in Lemma \ref{qzero}
corresponds to
 the toric variety of dimension $\dim \theta$ which is
 a transverse linear section of $T^{\Gamma^{(j)}_\theta}$
 at the point $(1,\ldots, 1)$ of the orbit corresponding to $\theta$.
 \end{remark}

\begin{proposition}  \label{stable}
If   $\theta$ is a stable cone then there is an integer $I
\geq 1$ such that the variety $T^{\Gamma^{(j)}_\theta}$
 is smooth
for all $j \geq I$. If $T^{\Gamma^{(j)}_\theta}$
 is smooth the cone $\theta$ is stable.

\end{proposition}
\textit{Proof}. By Lemma \ref{qzero}  and ii. of Proposition \ref{prev} below, we can assume that $\codim \theta = 0$.
The  blowing up of the ideal $\J _ {\theta} ^ {(j)} $ of $T^{
\Gamma_ {\theta} ^ {(j)} }$ is a finite morphism, dominated for all $j \gg 0$
by the
normalization of $T^{ \Gamma_ {\theta} ^ {(j)} }$, which is equal to that of $T^{ \Gamma_ {\theta} ^ {(1)} }$
 (see Proposition \ref{finite}). It follows that for $j \gg
0$ the map $T^{ \Gamma_ {\theta} ^ {(j+1)} }\to T^{ \Gamma_ {\theta} ^ {(j)} }$ is an isomorphism.
 By
Proposition \ref{jacoiso}, this variety is smooth. The second assertion follows directly from the definitions.
 \hfill $\ {\Box}$

\begin{remark}  By Proposition \ref{stable}
the stability problem is equivalent to the problem of determining if the iteration of the
blowing ups of logarithmic jacobian ideals
eventually resolves the singularities of the toric variety $T^\Gamma$.
\end{remark}

\begin{proposition} \label{prev}
Given $0 \ne \nu \in \s \cap N$,   for any $j \geq 1$ there exists a unique cone $\theta^{(j)} \in \Sigma^{(j)}$ such that
$\nu \in \int(\theta^{(j)} )$.  Then for $j \gg 0$ we have:
\begin{enumerate}
\item[{\rm i.}] If $q_0 = \codim \theta^{(j-1)}$ and if $m^{(j)} \in \J^{(j)}_{\theta^{(j-1)}}$ is such that
$\ord_{\J^{(j)}} (\nu) = \langle \nu , m^{(j)} \rangle$, then for any representation
\begin{equation}      \label{repre}
m^{(j)} = \g_1^{(j)} + \cdots + \g_d^{(j)}
\end{equation}
as a sum of linearly independent elements in the
semigroup $\Gamma_{\theta^{(j-1)}}^{(j)}$, there are exactly
$q_0$ of them in $ M(\theta^{(j-1)})$.
\item[{\rm ii.}]  $M(\theta^{(j-1)} , \Gamma_{\theta^{(j-1)}}^{(j)}) =  M(\theta^{(j-1)}) = M( \theta^{(j)} ) $ (cf. Notation \ref{ind}).
\end{enumerate}
Moreover, the sequence of cones   $(\theta^{(j)})$ stabilizes.
\end{proposition}
\textit{Proof}.   Notice that the vector
$\nu$ is in a subdivision of $\theta^{(j)}$ induced by
$\Sigma^{(j+1)}$, so that we have $\theta^{(j+1)}\subseteq
\theta^{(j)}$.
   The sequence $(\codim_{N_\R}  \ \theta^{(j)} )_j$ is increasing, thus there exists an integer $0 \leq q_0 \leq d-1$
such that  $ \codim  \ \theta^{(j)} = q_0 $, for $j \gg 0$.

Since $\nu \in \int (\theta^{(j)}) \cap N$, we have the equalities
\begin{equation} \label{perp}
 \Gamma_{\theta^{(j-1)}}^{(j)} \cap ({\theta^{(j-1)}})^\perp  = \Gamma_{\theta^{(j-1)}}^{(j)} \cap  \nu^\perp =
   M( \theta^{(j-1)} , \Gamma_{\theta}^{(j)} ).
\end{equation}
The lattice $M( \theta^{(j-1)} , \Gamma_{\theta^{(j-1)}}^{(j)} ) $
is a sublattice of finite index $i ( \theta^{(j-1)} ,
\Gamma_{\theta^{(j-1)}}^{(j)} )$ of $M( \theta^{(j-1)})$ (see
Notation \ref{ind} and Lemma \ref{conditions-aff}). By definition
$M( \theta^{(j-1)} )$ is  a saturated subsemigroup of $M$, and it
is also a rank $q_0$ lattice. By (\ref{essential})  we deduce that
$M( \theta^{(j-1)} ) = M (\theta^{(j)} )$ and $ M( \theta^{(j-1)}
, \Gamma_{\theta}^{(j)} ) \subset M( \theta^{(j)} ,
\Gamma_{\theta}^{(j+1)} )$. Then, the sequence of indices  $( i (
\theta^{(j-1)} , \Gamma_{\theta^{(j-1)} }^{(j)} ))_j$ stabilizes,
hence          $ M( \theta^{(j-1)} , \Gamma_{\theta}^{(j)} ) = M(
\theta^{(j)} , \Gamma_{\theta}^{(j+1)} )$ for $j \gg 0$.
 The
lattice  $ M( \theta^{(j-1)} , \Gamma_{\theta}^{(j)} )$  is
\textit{a priori} a sublattice of finite index
 of $M(\theta^{(j-1)})$.

We deal first with the proof of (i).
Up to relabeling the vectors we can assume that those $\g_i^{(j)}$ appearing in (\ref{repre}),
which belong to $M(\theta^{(j-1)})$ are
 $\g_1^{(j)} \dots, \g_{s}^{(j)}$ for $0 \leq s \leq q_0$. By (\ref{perp})
 these vectors belong to $M(\theta^{(j-1)}, \Gamma_{\theta^{(j-1)}}^{(j)} )$.
Suppose that $s\ne q_0$.
 Since, the images of the $\g_i^{(j)}$,
$i=s+1 , \dots,d$,
  generate a rank $d- q_0$ sublattice of
 $M/M(\theta^{(j-1)})$, we get that $d- q_0$ of them, say for $i=q_0 +1 ,
 \dots,d$, are linearly independent modulo $M(\theta^{(j-1)})$.
We can find vectors $\tilde{\g}^{(j)}_{s+1}, \dots, \tilde{\g}^{(j)}_{q_0}
\in \Gamma_{\theta^{(j-1)}} ^{(j)} \cap \nu^\perp $, such that
$\g_1^{(j)}\y \dots \y \g_s^{(j)} \y \tilde{\g}^{(j)}_{s+1}\y  \dots \y
\tilde{\g}^{(j)}_{q_0} \ne 0$.
 Then the
vector
\[ m' := \g_1^{(j)}+ \dots +  \g_{s}^{(j)}+  \tilde{\g}^{(j)}_{s+1} + \dots +
\tilde{\g}^{(j)}_{q_0}  + \g_{q_0+1}^{(j)} + \dots +  \g_{d}^{(j)}
\]
verifies that  $  m'  \in \J_{\theta^{(j-1)}}^{(j-1)}$. Since $0 =
\langle \nu,  \tilde{\g}^{(j)}_j \rangle <  \langle \nu, {\g}^{(j)}_j
\rangle $, for $ s+1\leq j\leq q_0$
 we would have $\langle \nu, m' \rangle  < \langle \nu,m^{(j)} \rangle$, a contradiction.

Suppose that (ii) does not hold. Then, since $
\Gamma_{\theta^{(j-1)}}^{(j)}$ generates the lattice $M$, there
exist $\g, \g' \in  \Gamma_{\theta^{(j-1)}}^{(j)} $ such that $\g
-\g ' \in M(\theta^{(j-1)}) \setminus
   M(\theta^{(j-1)}, \Gamma_{\theta^{(j-1)}}^{(j)} )
$. In view of (\ref{perp}) we get $0 \ne \langle \nu, \g  \rangle
= \langle \nu, \g' \rangle$. Since $\g_{q_0+ 1}^{(j)}, \dots,
\g_d^{(j)}$ define linearly independent vectors in the lattice $M
/ M(\theta^{(j-1)})$ which is of rank $d - q_0$, there exists an
integer $q_0+ 1 \leq i_0 \leq d$ such that $ \g$ (resp. $\g'$)
together with  $\g_{q_0+ 1}^{(j)} ,\dots,  \g_{i_0-1}^{(j)} ,
\g_{i_0+ 1}^{(j)} , \dots, \g_d^{(j)} $ are linearly independent
modulo $M(\theta^{(j-1)})$. Suppose without loss of generality
that  $i_0 = d$. Then the vectors $ \tilde{\g} :=  \g + \sum_{i =
1}^{d-1} \g_i^{(j)}  $ and $ \tilde{\g}' :=  \g' + \sum_{i =
1}^{d-1} \g_i^{(j)} $ belong to $\J_{\theta^{(j-1)}}^{(j)} $,
hence $\tilde{\g} -m^{(j)} = \g - \g_d^{(j)}$ and $\tilde{\g}
'-m^{(j)} = \g' - \g_d^{(j)}$  are both elements of
$\Gamma_{\theta^{(j)}}^{(j+1)} $. Since $\langle \nu,
\g_d^{(j)} \rangle  > 0$,
 it follows that $
\tilde{\g} - \tilde{\g}' = \g - \g' $ and $ \langle \nu,
\tilde{\g} \rangle  = \langle \nu, \tilde{\g}' \rangle  < \langle
\nu, {\g} \rangle  = \langle \nu, {\g}' \rangle$. By repeating
this construction, since $\nu\in \s\cap N$, in a finite number of
steps we reduce to the case when   $\langle \nu, \g \rangle =
\langle \nu, \g' \rangle     = 0 $.
 As we remarked before this implies that $\g - \g' \in  M(\theta^{(j-1)}, \Gamma_{\theta^{(j-1)}}^{(j)} )$, a contradiction.

By Lemma \ref{qzero} and ii. we can assume that $q_0 = 0$. By
(\ref{essential2})  the sequence  of positive integers  $( \langle
\nu, m^{(j)} \rangle)_{j \geq 1} $ is decreasing, hence it
stabilizes.
 Suppose that the sequence $(\theta^{(j)})$ does not
stabilize. By Proposition \ref{stable} this implies that the toric
variety $T^{\Gamma_{\theta^{(j-1)}}^{(j)}}$ is not smooth for any
$j\geq 1$.

 Let us fix an integer $j_1 \gg 0$. By  Proposition
\ref{bound} there exists a
 smallest
 integer $j_2 > j_1$ such that $m^{(j_2)} \ne m^{(j_1)}$.
We consider a representation of $m^{(j_2)} \in \J_{\theta
^{(j_2-1)}}^{(j_2) }$ of the form (\ref{repre}). The vector
$m^{(j_1)} - m^{(j_2)}$  belongs to the semigroup $\Gamma_{\theta
^{(j_2)}}^{(j_2 + 1) }$  since   $m^{(j_1)} \in \J_{\theta
^{(j_2-1)}}^{(j_2) }$ by (\ref{essential}). We obtain that
$\langle \nu, m^{(j_1)} - m^{(j_2)} \rangle =0$, hence $ m^{(j_1)}
- m^{(j_2)} \in \Gamma_{\theta ^{(j_2)}}^{(j_2 + 1) } \cap
\nu^\perp$. Since $\nu \in \int \theta^{(j_2)}$ and $\dim
\theta^{(j_2)} = d$ we deduce also that $\Gamma_{\theta
^{(j_2)}}^{(j_2 + 1) } \cap \nu^\perp = \{ 0 \}$, a contradiction,
which ends the proof of Proposition \ref{prev}. \hfill $\ {\Box}$
\par\medskip

If $\eta \subset N_\R$ is a rational convex polyhedral cone the
duality between the lattices $N$ and $M$ induces a duality between
the lattices $N_\eta$ and $M_\eta = M/ M(\eta)$ and also a duality
between $N(\eta) = N/ N_\eta$ and $M(\eta) = M \cap \eta^\perp$
(cf. Notations \ref{ind}).

If $0 \ne \eta$ is a stable cone one can consider for $j \gg 0$
the orbit closure $T^{\Gamma^{(j)}}_{\Sigma^{(j-1)}} (\eta) $
associated to $\eta \in \Sigma^{(j-1)}$. By convenience we recall the notations to
describe this toric variety in this case. See Lemma
\ref{orbit-closure} and Notations \ref{ind}.
\begin{notation} \label{clo-st} If $\eta\in \Sigma^{(j-1)}$ is a stable cone the variety
$T^{\Gamma^{(j)}}_{\Sigma^{(j-1)}} (\eta)$ is covered by charts
defined by the semigroups $\Gamma^{(j)}_{\theta^{(j-1)}} \cap
\eta^\perp$, for $\eta \leq \theta^{(j-1)}$ and $\theta^{(j-1)}
\in \Sigma^{(j-1)}$.  Notice that the semigroup
$\Gamma^{(j)}_{\theta^{(j-1)}} \cap \eta^\perp$ is a face of
$\Gamma^{(j)}_{\theta^{(j-1)}}$, and it spans the lattice $M(\eta,
\Gamma_\eta^{(j)}) = \Gamma^{(j)}_\eta \cap \eta^\perp$ by Lemma
\ref{conditions-aff}. By Proposition \ref{prev} ii. this lattice
is equal to $M(\eta)$ if $j \gg 0$. The fan $\Sigma^{(j-1)}
(\eta)$ consists of the images $\theta^{(j-1)} ( \eta) $ of cones
$\theta^{(j-1)} \in \Sigma^{(j-1)}$ in $N(\eta)_\R$.  The cone
$\theta^{(j-1)} (\eta)$ is the dual cone of $\R_{\geq 0} (
\Gamma^{(j)}_{\theta^{(j-1)}} \cap \eta^\perp)$. We denote by
$\J^{(j)}_{\theta^{(j-1)}} ( \eta ) $ the logarithmic jacobian
ideal of $k [ t^{\Gamma^{(j)}_{\theta^{(j-1)}} \cap \eta^\perp
}]$.
\end{notation}

The following technical lemma will be useful.
\begin{lemma}      \label{basis}
Let $0 \ne \eta$ be a stable cone of codimension $d_0 < d$. We
denote by
\[ \p: M_\R \to M_\R / M(\eta)_\R, \quad \a \mapsto \tilde{\a} \] the canonical
projection. If $ (\theta^{(j)})_{j}$ is a sequence  such that
\begin{equation}           \label{cond}
          \theta^{(j)} \in
\Sigma^{(j)}, \ \theta^{(j)} \supset \theta^{(j+1)}  \  \mbox{ and
}  \ \eta \leq  \theta^{(j)},
\end{equation}
then  for $j \gg 0$ we have:
\begin{enumerate}
\item[i.]
$\p ( \Gamma_{\theta^{(j-1)}}^{(j)} ) = \p (
\Gamma_{\theta^{(j)}}^{(j+1)}    )$
 and
the semigroup $\tilde{\Gamma}_\eta := \p (
\Gamma_{\theta^{(j-1)}}^{(j)}    )  $ is generated by a basis
$\tilde{e}_{d_0 +1} , \dots, \tilde{e}_{d}$ of $M/M(\eta)$.

\item[ii.] If $m^{(j)}$ is the vertex of the polyhedron
$\mathcal{N} (\J^{(j)}_{\theta^{(j-1)}})$ such that
\begin{equation} \label{numj}
\ord_{{\J}^{(j)}} (\nu) = \langle \nu, m^{(j)} \rangle, \quad
\forall  \nu \in \theta^{(j)}, \end{equation} then for any
representation of the form (\ref{repre}) of $m^{(j)}$ as a sum of
linearly independent vectors in $\Gamma^{(j)}_{\theta^{(j-1)}}$
then exactly $d_0$ of the $\g_i^{(j)}$ belong to
$\Gamma^{(j)}_{\theta^{(j-1)}}\cap\eta^\perp$, say for $i =1, \dots, d_0$, while
$\langle \nu_\eta, \g_i^{(j)} \rangle = 1$ for $i =d_0 +1, \dots,
d$.

\item[iii.] The vector $m^{(j)}$ belongs to the face
$\mathcal{F}_{\theta^{(j-1)}}^{(j)}$ of $\mathcal{N}
(\J^{(j)}_{\theta^{(j-1)}})$ determined by $\nu_\eta$ (cf.
Definition \ref{def-sta}). The face
$\mathcal{F}_{\theta^{(j-1)}}^{(j)}$ is the Minkowski sum
\[ \mathcal{N} (\J^{(j)} _{\theta^{(j-1)}} (\eta) ) +
\mathcal{P}^{(j)}_{\theta^{(j-1)}},\]
 where
$\mathcal{P}^{(j)}_{\theta^{(j-1)}}$ is the convex hull of the set
$ \bigcup \, \delta_{i_{d_0 +1}}^{(j)}  + \cdots + \delta_{i_{d}}^{(j)}
\, + \, (\check\theta^{(j-1)}
  \cap \eta^{\perp})$,
for    $\delta_{i_{d_0 +1}}^{(j)} , \dots, \delta_{i_{d}}^{(j)} \in
\Gamma_{\theta^{(j-1)}}^{(j)}$ such that
$
       \delta_{i_{d_0 +1}}^{(j)} \y \cdots \y \delta_{i_{d}}^{(j)} \ne 0$,   and   $\langle \nu_\eta, \delta_{i_{l}}^{(j)} \rangle =
       1 $ for $l=d_0+1,\ldots, d $.
\end{enumerate}
\end{lemma}
\textit{Proof}. Since
$\eta$ is a stable cone we get that $M(\eta) = M(\eta,\Gamma_\eta^{(j)})$,
by applying Proposition \ref{prev} to the constant sequence of cones $\eta_j :=\eta$.
By Lemma \ref{qzero} the cone $\eta \subset (N_\eta)_\R$ is a stable cone for the semigroup
$     \p ( \Gamma_{\eta}^{(j)}    ) $ for $j \geq j_0 \gg 0$.  By Proposition \ref{stable}
the sequence of semigroups $(\p ( \Gamma_{\eta}^{(j)}    ))_{j\geq j_0}$ stabilizes and
$\p ( \Gamma_{\eta}^{(j)}    )$  is generated by a basis of $M/M(\eta)$,
for $j \gg 0$.
By Lemma \ref{conditions-aff}
we have that
$ \Gamma_\eta^{(j)} = \Gamma_{\theta^{(j-1)}}^{(j)}  + \Z_{\geq 0} (-u_j) $, for any
 $u_j \in \Gamma_{\theta^{(j-1)}}^{(j)}$  which belongs to $\int (\check{\theta}^{(j-1)} \cap \eta^\perp)$.
We deduce that   $\p ( \Gamma_{\theta^{(j-1)}}^{(j)}    )  = \p (
\Gamma_{\eta}^{(j)}) $ for $j \gg 0$.

Since $\eta \leq \theta^{(j)}$ we get from (\ref{numj}) that
$\ord_{\J^{(j)}} (\nu_\eta) = \langle \nu_\eta, m^{(j)} \rangle$.
This implies that $m^{(j)}$ belongs to
$\mathcal{F}_{\theta^{(j-1)}}^{(j)}$. Notice also that $m^{(j)}$
belongs to $\J_{\eta}^{(j)}$. By Proposition \ref{prev} i., for
any representation of $m^{(j)}$ of the form (\ref{repre}), $d_0$
of the $\g_i^{(j)}$ belong to $\eta^{\perp}$, say for $i =1,
\dots, d_0$. By Lemma \ref{induced} the vectors $\p (\g^{(j)}_i) =
\tilde{\g}^{(j)}$, $i =d_0 +1, \dots, d$ are linearly independent
elements of $\Gamma_\eta$, such that $\tilde{m}^{(j)} = \sum_i
\tilde\g^{(j)}_i $ belongs to the logarithmic jacobian ideal $
\tilde{\J}_\eta$ of $k[t^{\tilde{\Gamma}_\eta}]$. We know that
$\langle \nu, m \rangle = \langle \nu, \tilde{m} \rangle$ for any
$m \in M$ and $\nu \in \eta$,  thus the vector $\tilde{m}^{(j)}$
is in the face of $\mathcal{N} ( \tilde{\J}_\eta)$ determined by
$\nu_\eta$. By i. the semigroup $\tilde{\Gamma}_\eta$ is regular,
hence we deduce that $\tilde{m}^{(j)} = \sum_{i=d_0+1}^d
\tilde{e}_i$ and, up to relabeling,  $\tilde{e}_i =
\tilde{\g}^{(j)}_i$, $i = d_0 +1, \dots,d$. This ends the proof of
ii.~ and also shows that $m^{(j)} \in \mathcal{N} (\J^{(j)}
_{\theta^{(j-1)}} (\eta) ) + \mathcal{P}^{(j)}_{\theta^{(j-1)}}$.

We have also shown the equalities
\[ \ord_{\J^{(j)}} ( \nu_\eta) = \langle
\nu_\eta, m^{(j)} \rangle = \langle \nu_\eta, \tilde{m}^{(j)}
\rangle = d-d_0.  \] Finally, we remark that the argument given
above applies more generally for any vector $m \in
\J^{(j)}_{\theta^{(j-1)}}$ in the face
$\mathcal{F}^{(j)}_{\theta^{(j-1)}}$. This implies that
iii.~holds. \hfill $\ {\Box}$

\begin{proposition} \label{uno}
If $0 \ne \eta$ is a stable cone and if  $ (\theta^{(j)})_{j}$ is a sequence of cones of the form (\ref{cond})
such that $ \theta^{(j)} $ contains $\eta$ as a face of codimension one,  then
\begin{equation} \label{ret}
M(\theta^{(j-1)}, \Gamma_{\theta^{(j-1)}} ^{(j)} ) = M (\theta^{(j-1)}) \mbox{ for } j \gg 0,
\end{equation}
and the sequence of cones $(\theta^{(j)})_{j
\geq I}$ stabilizes.
\end{proposition}
\textit{Proof}.  We denote by $q_0$ the integer $\codim_{N_\R}  \ \theta^{(j)}$ for $j \gg 0$ and
by $d_0$ the codimension of $\eta$.
Notice that $0 \leq q_0 < d-1$ since $0 \ne \eta \leq \theta^{(j)}$.

We deal first with the proof of (\ref{ret}). By the argument given
in the proof of Proposition  \ref{prev} we get $M(\theta^{(j-1)} )
= M(\theta^{(j)})$ and $M(\theta^{(j-1)} ,
\Gamma^{(j)}_{\theta^{(j-1)}}) = M(\theta^{(j)} ,
\Gamma^{(j+1)}_{\theta^{(j)}})$ for $j \gg 0$.
 We recall that for $j \gg 0$
the lattice spanned by the face $\Gamma^{(j)}_{\theta^{(j-1)}}
\cap \eta^\perp$ is equal to $M(\eta)$ (see Notations
\ref{clo-st}).

If (\ref{ret}) were not true then there exists $\g, \g' \in   \Gamma^{(j)}_{\theta^{(j-1)}}$ such that
$\g - \g' \in   M (\theta^{(j-1)}) \setminus  M(\theta^{(j-1)}, \Gamma_{\theta^{(j-1)}} ^{(j)} )$. Notice then that
$\langle \nu_\eta, \g \rangle = \langle \nu_\eta, \g' \rangle $ since $M (\theta^{(j-1)}) \subset M(\eta)$ by duality.

By Lemma \ref{basis}, if  $m^{(j)} \in \J_{ \theta^{(j-1)}}^{(j)}$
is such that $\ord_{\J^{(j)}} (\nu) = \langle \nu, m^{(j)}
\rangle$ for any $\nu \in \theta^{(j)}$ then for any
representation of $m^{(j)}$ of the form (\ref{repre}),  $d_0$ of
the $\g_i^{(j)}$ belong to $M(\eta)$.

By applying the argument of Proposition \ref{prev} we can assume,
replacing $j$ by a bigger number, that $\langle \nu_\eta, \g
\rangle  =0$, that is $\g$ and $\g'$ belong to the face $
\Gamma^{(j)}_{\theta^{(j-1)}} \cap \eta^\perp$.

By hypothesis $\eta \leq \theta^{(j-1)}$ is a face of codimension
one. The image of
 $\check{\theta}^{(j-1)} \cap \eta^\perp$    in $M(\eta)_\R / M(\theta^{(j-1)})_\R$ is the dual cone of
the image   $\bar{\theta}^{(j-1)}$  of $\theta^{(j-1)}$  in $ (N_{\theta^{(j-1)})_\R} / (N_\eta)_\R$.
Notice that the lattice  $N_{\theta^{(j-1)}} / N_{\eta}$  and the cone $\bar{\theta}^{(j-1)}$
  are independent of $j$ for
$j\gg 0$.
We denote by  $\bar{\nu}$ the generator of the semigroup $\bar{\theta}^{(j-1)} \cap (N_{\theta^{(j-1)}} / N_{\eta})$
and by $\bar{\a}$ the class of $\a \in M(\eta)$ modulo $M(\theta^{(j-1)})$.

Among those $\g_i^{(j)}$ which belong to $\eta^{\perp}$ there exists at least one, say $\g_d^{(j)}$, which does not belong to
$M(\theta^{(j-1)})$, hence $ \langle \bar{\nu}, \bar{\g}_d^{(j)} \rangle \ne 0$. Notice also that
$0 \ne \langle \bar{\nu}, \bar{\g} \rangle =  \langle \bar{\nu}, \bar{\g}' \rangle$. Then,
we apply the same argument as in
the proof of Proposition \ref{prev} i. to get that  $\g - \g_d^{(j)}$ and $\g' - \g_d^{(j)}$ belong to
$\Gamma_{\theta^{(j)}}^{(j+1)} \cap \eta^\perp$ and $\langle \bar{\nu},  \bar{\g} - \bar{\g}_d^{(j)} \rangle < \langle
\bar{\nu}, \bar{\g} \rangle$. By iterating this procedure,  replacing $\g$ and $\g'$ by $\g - \g_d^{(j)}$ and
$\g - \g_d^{(j)}$, respectively, we reduce to the case $0 =
\langle \bar{\nu}, \bar{\g} \rangle  =   \langle \bar{\nu}, \bar{\g}' \rangle$. But this implies that
$\g, \g'$ belong to $M( \theta^{(j-1)}, \Gamma^{(j)}_{\theta^{(j-1)}} )$, contradicting  the hypothesis.
This ends the proof of the equality
(\ref{ret}).

\medskip

We prove now that the sequence $(\theta^{(j)})$ stabilizes. Since
(\ref{ret}) holds, by Lemma \ref{qzero} we can assume that $q_0 =
0$. Then, by hypothesis the lattice $M(\eta)$ is of rank one. The
cone $\check\theta^{(j-1)}\cap \eta^{ \perp} \subset M(\eta)_\R$
is a one dimensional face  of $\check\theta^{(j-1)}$ for all $j
\gg 0$. Since $\check{\theta}^{(j-1)} \subset
\check{\theta}^{(j)}$ we get that the cone $\check\vartheta:=
\check\theta^{(j-1)}\cap \eta^{ \perp}$ is independent of $j$, for
$j \gg 0$.

With notations of Lemma \ref{basis}, the Minkowski sum $\p^{-1}
(\tilde{e}_i) \cap \Gamma^{(j)}_{\theta^{(j-1)}} +
\check\vartheta$ is an affine one dimensional cone with only one
vertex $\g_i^{(j)}$, which belongs to
$\Gamma^{(j)}_{\theta^{(j-1)}}$, for $i=2, \dots, d$. We denote
by $\g_1^{(j)}$ the smallest generator of the rank one semigroup
$\Gamma^{(j)}_{\theta^{(j-1)}}\cap \check\vartheta$. Notice that
$\g_1^{(j)}$ is independent of $j$ for $j \gg 0$.

By Lemma \ref{basis} the vector $m^{(j)} = \g_1^{(j)} + \cdots +
\g_{d}^{(j)} \in \J^{(j)}_{\theta^{(j-1)}}$  is the unique vertex
of the face $\mathcal{F}^{(j)}_{\theta^{(j-1)}} = m^{(j)}+
\check\vartheta$. Since $\nu_\eta \in \theta^{(j)}$ the cone
$\theta^{(j)}$ is dual to the cone spanned by $\{ \g - m^{(j)}
\mid \g \in \J^{(j)}_{\theta^{(j-1)}} \}$.

\begin{claim}\label{oneorbit} The semigroup
$\Gamma^{(j)}_{\theta^{(j-1)}} \cap \eta^\perp $  is regular for $j \gg 0$.
 \end{claim}
\textit{Proof of the claim.}  By Proposition \ref{prev} the
semigroup $\Gamma^{(j)}_{\theta^{(j-1)}}\cap\eta^{ \perp} =
\Gamma^{(j)}_{\theta^{(j-1)}}\cap \check\vartheta$
 generates the  group
$M(\eta)$,  for $j \gg 0$. If it has only one generator the result
follows directly. Assume that it has at least two generators. Let
us denote by $\d > \g_1^{(j)}$ the second element by order of
size. The element $n^{(j)}= \gamma^{(j)}_2+\cdots
+\gamma^{(j)}_1+ \delta$ belongs to
$\J^{(j)}_{\theta^{(j-1)}}$ so that $n^{(j)}-m^{(j)}=\delta
-\gamma_1^{(j)}$ is in $ \Gamma^{(j+1)}_{\theta^{(j)}}$. We see
that after finitely many steps the smallest generator of our
semigroup has decreased, so in the end we reach the generator of
the regular semigroup $\check\vartheta \cap M$, which proves the
result. This is similar to the resolution process for one
dimensional affine toric varieties by Semple-Nash
modifications.\hfill $\ {\Box}$

By Lemma \ref{basis} and Claim \ref{oneorbit} the elements
$\g_1^{(j)}, \dots, \g_d^{(j)}$ form a basis of $M$.

The expansion of $\g \in \Gamma^{(j)}_{\theta^{(j-1)}}$  in terms of
the basis  $\g_1^{(j)}, \dots, \g_{d}^{(j)}$ is of the form
\begin{equation} \label{exp-gamma}
\g = a_1\g_1^{(j)} + \cdots + a_d \g_d^{(j)},
\end{equation}
where $a_i \in \Z_{\geq 0}$ and $a_d \in \Z$. If $\g$ is of the
form (\ref{exp-gamma}) then we have $\langle \nu_\eta, \g \rangle
= \sum_{i=2}^{d} a_i$. In particular, we get that
$\ord_{\J^{(j)}} (\nu_\eta) = \langle \nu_\eta, m^{(j)} \rangle =
d-1$.

We denote by $G^{(j)}$ the minimal generating system of the
semigroup $\Gamma^{(j)}_{\theta^{(j-1)}}$, and by   $g^{(j)}$ the
maximum of $\nu_\eta$ on the set $G^{(j)}$.

Notice that the elements $\g_1^{(j)}, \dots, \g_d^{(j)}$ belong to
$G^{(j)}$ by our assumptions.  If a generator $\g \in \G^{(j)}$ is different
from $\g_1^{(j)}, \dots, \g_d^{(j)}$ then the coefficient $a_1$
 in (\ref{exp-gamma}) is $< 0$, and $\langle \nu_\eta, \g \rangle  \geq 2$. Otherwise $\g$ would
be in the semigroup generated by the $\g_i^{(j)}$, contradicting
the minimality of the generating system $G^{(j)}$. We deduce that
the equality $g^{(j)} = 1$  implies that $G^{(j)} = \{ \g_i^{(j)}
\}_{i=1}^d$, hence in this case the semigroup
$\Gamma^{(j)}_{\theta^{(j-1)}}$ is regular.

\begin{claim}    \label{hache}
If $g^{(j)} > 1$ there exists an integer $t_0 \geq 1$ such that
$g^{(j+t_0)} < g^{(j)}$.
\end{claim}
\textit{Proof of the claim}.
By Lemma \ref{repre-gen} the semigroup
$\Gamma_{\theta^{(j)}}^{(j+1)}$   is generated by $\g_1^{(j)}, \dots, \g_d^{(j)}$ and
vectors of the form $\g - \g_l^{(j)}$ , where $\g \in G^{(j)}$, $\g \ne \g_l^{(j)}$
and the $l$-th coordinate of $\g$ in the basis $\g_1^{(j)}, \dots, \g_d^{(j)}$ is non-zero
(in particular $\langle \nu_\eta, \g \rangle \geq 2$). We say that vectors of this form are
 \textit{followers} of $\g$.
Notice that $\g - \g_1^{(j)}$ is always a follower of $\g$ with
$\langle \nu_\eta, \g \rangle  = \langle \nu_\eta,  \g -
\g_1^{(j)} \rangle$, while $\langle \nu_\eta, \g - \g_l^{(j)}
\rangle  < \langle \nu_\eta, \g \rangle$ if $l \ne 1$.

We deduce  also that $m^{(j)} \ne m^{(j+1)}$ if and only if there
is an element $\g \in G^{(j)}$ with $\langle \nu_\eta, \g \rangle
= 2$. Indeed, if $\langle \nu_\eta, \g \rangle  = 2$ then there is
$2 \leq i \leq d$ such that $\g - \g_i^{(j-1)}$ is a follower of
$\g$ with $\langle \nu_\eta,  \g - \g_i^{(j-1)} \rangle =1$.  This
implies that $\g_i^{(j)}$ is in the semigroup generated by $\g -
 \g_i^{(j)}$ and $\g_d^{(j)}$, thus $\g_i^{(j)}$ does not belong
 to $G^{(j+1)}$ hence $\g_i^{(j)} \ne \g_i^{(j+1)}$ and $m^{(j)}
 \ne m^{(j+1)}$. The converse is deduced similarly.

Assume that $m^{(j)} = m^{(j+1)} = \cdots  = m^{(j+t_0 -1)}$ for
some $t_0 \geq 1$. Let $\g \in G^{(j)}$  be such that $\langle
\nu_\eta, \g \rangle  = g^{(j)}$. If $\g$ is of the form
(\ref{exp-gamma}) the followers of $\g$ after at most  $t_0$
iterations are $\g_\a = \sum_{i=1}^d (a_i - \a_i) \g_i^{(j)}$,
where $\a= (\a_1, \dots, \a_d) \in \Z_{\geq 0}^d$ verify that
$\sum_{i=1}^{d} \a_i \leq t_0$, $\a_i \leq a_i$ for $i=2, \dots,
d$ and $\langle \nu_\eta, \g_\a \rangle = \sum_{i=2}^{d} (a_i
-\a_i) \geq 1$. Those $\g_\a$ with $\langle \nu_\eta, \g_\a
\rangle = \langle \nu_\eta, \g \rangle$ are precisely $\g - l
\g_1^{(j)}$ for $0 \leq l \leq t_0$. The elements  $\g - l
\g_1^{(j)}$, $l =0, \dots, t_0 -1$ do not belong to $G^{(j+t_0)}$
since they are in the semigroup generated by $\g_1^{(j)}$ and $\g
- t_0 \g_1^{(j)}$.

Assume that $t_0
> 1$ is  the smallest integer such that $m^{(j)} \ne m^{(j+t_0)}$.
Notice that  $t_0 \leq  g^{(j)} -1$, otherwise we would get a
follower $\g_\a$ of $\g$ with $\langle \nu_\eta, \g_\a \rangle  =
1$, which is necessarily different from the  $\g_i^{(j)}$, $i= 2,
\dots, d$, and then $m^{(j)} \ne m^{(j+ g^{(j)} - 1 )}$, a
contradiction.

There exists $0 \ne (p_2, \dots, p_d) \in \Z^{d-1}_{\geq 0}$
such that
\[
\g_i^{(j + t_0 ) } = \g_i^{(j)} - p_i \g_1^{(j)}, \ i =2, \dots,
d, \  \mbox{ and } \ \g^{(j+ t_0)}_1  = \g^{(j)}_1.
\]
We deduce the expansion
\begin{equation} \label{g-alpha}
\g_\a = \sum_{i=2}^{d} (a_i - \a_i) \g_i^{(j + t_0 ) } + (a_1 -
\a_1  + \sum_{i=2}^{d} p_i (a_i - \a_i) ) \g_1^{(j + t_0 ) }.
\end{equation}

If $a_i p_i = 0$ for $i=2, \dots, d$ we iterate this procedure
replacing $\g$ by $\g - t_0 \g_1^{(j)}$. In at most $\langle
\nu_\eta, \g \rangle -1$ steps we get to the situation where at
least one of the $a_i p_i$ is non-zero, say $a_d p_d \ne 0$ for
simplicity. We prove that $\a_0 = (t_0 -1, 0, \dots, 0, 1)$
defines a follower of $\g$ such that $\g - t_0 \g_1^{(j)}$ belongs
to the semigroup generated by $\g_{\a_0} $ and $\g_i^{(j + t_0 )
}$, $i=1, \dots, d$, in particular $\g - t_0 \g^{(j)}_1$ does not
belong to $G^{(j+t_0)}$. We check this assertion by verifying that
the coefficient of the term $\g_i^{(j+t_0)}$ in the expansion
(\ref{g-alpha}) of $\g_{\a_0}$ is less than or equal to the
corresponding coefficient in the expansion (\ref{g-alpha}) of $\g
- t_0 \g_1^{(j)}$ for $i=1, \dots, d$. If $i =d$ we get a strict
inequality. The inequality is trivial for $i=2, \dots, d-1$. If
$i= 1$ we have to show the inequality:
\[
a_1 - (t_0 -1) - p_d + \sum_{i=2}^{d} p_i a_i \leq a_1 - t_0 +
\sum_{i=2}^{d} p_i a_i.
\]
This inequality is equivalent to  $p_d \geq 1$, which holds since
$a_d p_d \ne 0$.

Since there is a finite number of $\g \in G^{(j)}$ with $\langle
\nu_\eta, \g \rangle = g^{(j)}$ there exists an integer $t_2 \geq
0$ such that $g^{(j+t_2) } < g^{(j)}$ as claimed.
 \hfill $\ {\Box}$

Using this claim and induction there exists an integer $t_1 \geq
1$ such that $g^{(j+t_1)} = 1$, hence the semigroup
$\Gamma^{(j)}_{\theta^{(j-1)}}$ is regular. This ends the proof of
Proposition \ref{uno}. \hfill $\ {\Box}$

\begin{definition}
  The nested sequence of cones $(\theta^{(j)})_j$ is
\textit{distinguished} if there exist a stable cone $0 \ne
\eta$, an integer $I $, and   a sequence of faces of
$\theta^{(j)}$ for $j \geq I$,
\begin{equation} \label{faces}
\eta = \z_0^{(j)}  \leq \z_1^{(j)} \leq \cdots \leq \z_{l_0}^{(j)}
= \theta^{(j)},
\end{equation}
such that $\dim \z_i^{(j)} = \dim \eta + i$ and $\z_i^{(j)}
\supset \z_i^{(j+1)}$ for $i =0, \dots, l_0$ and $l_0 \leq \codim
\eta$.
\end{definition}

 \begin{proposition}\label{fin}
If the sequence $ (\theta^{(j)})_j$ is distinguished then it
stabilizes.
 \end{proposition}
\textit{Proof}. Let $(\theta^{(j)})_j$ be a distinguished sequence
as above. We consider the sequence of faces $(\z_1^{(j)})_j$.
Since $\eta $ is a face of codimension one of $\z_1^{(j)}$ we get
that the sequence $(\z_1^{(j)})_j$ stabilizes by Proposition
\ref{uno}. We proceed replacing $\eta = \z_0^{(j)}$ and
$\z_1^{(j)}$ by $\z_1^{(j)}$ and $\z_2^{(j)}$ respectively, in the
previous argument, and then the result follows by induction on the
length $l_0$ of the sequence (\ref{faces}).
 \hfill $\ {\Box}$

\begin{remark} \label{inclu}  If $0 \ne \eta$ is a stable cone, for $j \gg 0$,
the blowing up of the logarithmic
jacobian ideal induces a proper toric modification $\xi_\eta^{(j)}
: T^{\Gamma^{(j+1)}}_{\Sigma^{(j+1)}} (\eta) \to
T^{\Gamma^{(j)}}_{\Sigma^{(j)}} (\eta)$.
 As we explain below, the map $\xi_\eta^{(j)}$  is not
necessarily equal to the
 blowing up of the logarithmic jacobian ideal of
 $    T^{\Gamma^{(j)}}_{\Sigma^{(j)}}  (\eta)$.
\end{remark}

\begin{proposition}     \label{orbitq}
Given a stable cone $\eta$ of codimension $d_0$, a vector $0 \ne w
\in N(\eta)$ and a sequence of cones $\theta^{(j)} (\eta) \in
\Sigma^{(j)} (\eta)$ such that $\theta^{(j)} (\eta) \supset
\theta^{(j+1)} (\eta)$ for $j \gg 0$.  Then, if $w \in \int (
\theta^{(j)} (\eta) )$ for $j \gg 0$, we have that
\begin{equation} \label{ret2}
M(\theta^{(j-1)} (\eta) , \Gamma_{\theta^{(j-1)}} ^{(j)}  \cap
\eta^\perp ) = M (\theta^{(j-1)} (\eta) ) \mbox{ for } j \gg 0,
\end{equation}
and the sequence $\theta^{(j)} (\eta)$ stabilizes.
\end{proposition}
\textit{Proof}. We use notations \ref{clo-st}. We denote by $q_0$
the codimension of $\theta^{(j)} (\eta)$ for $j\gg 0$. If
$\theta^{(j-1)} \in \Sigma^{(j-1)}$ is a cone such that  $\eta
\leq  \theta^{(j-1)} $ and its image in $N(\eta)$ is equal to
$\theta^{(j-1)} (\eta)$ then by definition we get
$M(\theta^{(j-1)} (\eta) , \Gamma_{\theta^{(j-1)}} ^{(j)} \cap
\eta^\perp ) = M(\theta^{(j-1)} , \Gamma_{\theta^{(j-1)}} ^{(j)}
\cap \theta^{{(j-1)}\perp} ) \subset M( \eta, \Gamma_{\eta}
^{(j)}) $ and $ M ( \theta^{(j-1)} (\eta))  = M (\theta^{(j-1)})
\subset
 M (\eta)$.

We get that $M(\theta^{(j-1)} (\eta) , \Gamma_{\theta^{(j-1)}}
^{(j)} \cap \eta^\perp ) = M(\theta^{(j)} (\eta) ,
\Gamma_{\theta^{(j)}} ^{(j+1)}  \cap \eta^\perp )$ and  $ M
(\theta^{(j-1)} (\eta) )  = M (\theta^{(j)} (\eta) )$ for $j \gg
0$. Then, the proof of formula (\ref{ret2}) follows similarly as
that of (\ref{ret}) and of Proposition \ref{prev} ii.

We deduce from (\ref{ret2}) and Lemma \ref{basis} that it is
enough to prove the result in the case $\codim \theta^{(j)} (\eta)
= 0$. In this case, the cone $\theta^{(j-1)} (\eta)$ is the image
of a cone $\theta^{(j-1)} \in \Sigma^{(j)} (d)$ by the projection
$\p \colon N_\R \to N(\eta)_\R$.

The charts of the blowing up of the logarithmic jacobian ideal of
$T^{\Gamma^{(j)}}_{\theta^{(j-1)}}$ which are defined by those $d$
dimensional cones which contain $\eta$, are in bijection with the
vertices  of the
 face
$\mathcal{F}^{(j)}_{\theta^{(j-1)}}$  of the Newton polyhedron of
the logarithmic jacobian ideal $\J^{(j)}_{\theta^{(j-1)}}$
determined by the vector $\nu_\eta$.
One of these charts  $T^{\Gamma^{(j+1)}}_{\theta^{(j)}}$ ,
corresponding to a vertex  $m^{(j)}$ of
$\mathcal{F}^{(j)}_{\theta^{(j-1)}}$ is such that the image of the
cone $\theta^{(j)}$ in $N(\eta)$ is equal to the cone
$\theta^{(j)} (\eta)$ in the sequence. 

By Lemma \ref{basis} we can choose a  representation $ m^{(j)}  =  \g_1^{(j)} + \cdots + \g_{d}^{(j)}$ as sum of linearly independent elements of
$\Gamma^{(j)}_{\theta^{(j-1)}}$, 
such that 
\begin{equation} 
\label{repre3}
\langle \nu_\eta, \g_i^{(j)} \rangle = 
\left\{ 
\begin{array}
{ccl}
0 & \mbox{if} &  1 \leq i \leq d_0,
\\ 
1 & \mbox{if} &  d_0 + 1 \leq i \leq d. 
\end{array}
\right.
\end{equation}
Then we have that $ m^{(j)} = \bar{m}^{(j)} +   m'^{(j)}$ where 
\begin{equation}  \label{repre2}
\bar{m}^{(j)} = \g_1^{(j)} + \cdots + \g_{d_0}^{(j)} \mbox{ and }
m'^{(j)} =  \g_{{d_0 +1}}^{(j)} + \cdots + \g_{{d}}^{(j)}. 
\end{equation}
By hypothesis $w \in
N(\eta)$ is in the relative interior of $\theta^{(j)} (\eta)$.
We get that 
$\bar{m}^{(j)}$
is the vertex of the Newton polyhedron of the  logarithmic jacobian
ideal $\J^{(j)}_{\theta^{(j-1)}} (\eta)$ of $k
[t^{\Gamma^{(j)}_{\theta^{(j-1)}}}]$ determined by $w$ (see Lemma  \ref{basis}).

Remark that the chart associated to $\bar{m}^{(j)}$ of the blowing
up of the logarithmic jacobian ideal $\J^{(j)}_{\theta^{(j-1)}}
(\eta)$ of $k [t^{\Gamma^{(j)}_{\theta^{(j-1)}}}]$ is defined by
the semigroup $S^{(j+1)}$ generated by
$\Gamma^{(j)}_{\theta^{(j-1)}} \cap \eta^\perp$ and vectors of the
form $\g - \bar{m}^{(j)}$ for $\g$ in $\J_{\theta^{(j-1)}}^{(j)}
(\eta)$. The inclusion
$S^{(j+1)} \subset \Gamma^{(j+1)}_{\theta^{(j)}} \cap
\eta^\perp$,
may be strict (cf. Remark \ref{inclu}). The reason is the
following. If $\g \in \Gamma^{(j)}_{\theta^{(j-1)}}$ and if $\g \y
\g_1^{(j)} \y \cdots \y \g_{i-1} ^{(j)} \y \g_{i+1}^{(j)} \y
\cdots \y \g_{d}^{(j)} \ne 0$ then $\d := \g + m^{(j)} -
\g_i^{(j)}$ belongs to $\J_{\theta^{(j-1)}}^{(j)}$ and  $\d -
m^{(j)} = \g - \g_i^{(j)} \in \Gamma^{(j+1)}_{\theta^{(j)}}$. If
in addition, $\langle \nu_\eta, \g \rangle = \langle \nu_\eta,
\g_i^{(j)} \rangle$ we obtain that $\g -  \g_i^{(j)}  $ belongs to $
\Gamma^{(j+1)}_{\theta^{(j)}} \cap \eta^\perp$, even if $\g \notin
\eta^\perp$.

Notice that the sequence $(\langle w, \bar{m}^{(j)} \rangle )$
is stationary since $w$ belongs to the
interior of $\theta^{(j-1)} (\eta)$. There exists $j_1\geq 0$, such that the sequence 
$(\langle w, \bar{m}^{(j)} \rangle )$ is constant for $j \geq j_1$.

 If the sequence
$(\theta^{(j)} (\eta))$ does not stabilize then the toric variety
$T^{\Gamma_{\theta^{(j-1)}}^{(j)}}$ is not smooth for any integer
$j$. 
By Proposition
\ref{bound}, if the toric variety
$T^{\Gamma_{\theta^{(j_1-1)}}^{(j_1)}}$ is not smooth then there
is a smallest $j_2 > j_1$ such that $m^{(j_1)} \ne m^{(j_2)}$. 
The vector ${m}  ^{(j_1)} - {m}^
{(j_2)}$,   
 which can be
written as
\[
   {m}  ^{(j_1)} - {m}^ {(j_2)}  = ( \bar{m}  ^{(j_1)} - \bar{m}^ {(j_2)} )  + ( {m'}  ^{(j_1)} - {m'}^ {(j_2)}),
\]
belongs to $\Gamma^{(j_2+1)}_{\theta^{(j_2)}} $. By (\ref{repre3}),  both terms     
$\bar{m}  ^{(j_1)} - \bar{m}^ {(j_2)}$ and      ${m'}  ^{(j_1)} - {m'}^ {(j_2)}$
belong to the face $ \Gamma^{(j_2+1)}_{\theta^{(j_2)}} \cap \eta^\perp$ of  $\Gamma^{(j_2+1)}_{\theta^{(j_2)}} $, 
and at least one of them is non-zero.

If   $\bar{m}  ^{(j_1)} - \bar{m}^ {(j_2)} \ne 0 $ then we get
that $\langle w , \bar{m}  ^{(j_1)} - \bar{m}^ {(j_2)} \rangle  =
0$. This implies that $\bar{m}  ^{(j_1)} - \bar{m}^ {(j_2)} \in
(\Gamma^{(j_2+1)}_{\theta^{(j_2)}} \cap \eta^\perp) \cap w^\perp$,
but since $w  \in \int (\theta^{(j_2)} (\eta))$ the face
$\Gamma^{(j_2+1)}_{\theta^{(j_2)}} \cap \eta^\perp \cap w^\perp$
is reduced  to zero, a contradiction.

We deduce that $\bar{m}^{(j)} = \bar{m}^{(j_1)}$ for all $j \geq
j_1$.  Since $w \in \mathrm{int} (\theta^{(j)} (\eta))$ determines
 the  vertex $\bar{m}^{(j)} $ of the Newton polyhedron of the 
logarithmic jacobian ideal $\J^{(j)}_{\theta^{(j-1)}} (\eta)$, we can assume, up to relabeling,  that 
\begin{equation} \label{repre4}
\g_i^{(j)} = \g_i^{(j_1)}  \mbox{ for } =1, \dots, d_0 \mbox{ and }  j > j_1. 
\end{equation} 

We denote by $\vartheta$ the cone spanned by the vectors $\g_i^{(j_1)}, \dots, \g_{d_0}^{(j_1)}$. 
We prove below that the cone 
$\R_{\geq 0} ( \Gamma^{(j)}_{\theta^{(j-1)}} \cap \eta^\perp)$ is equal to $\vartheta$ for any $j \geq j_1$.  
This implies that the sequence of cones $(\theta^{(j-1)} (\eta))$   stabilizes since 
$\R_{\geq 0} ( \Gamma^{(j)}_{\theta^{(j-1)}} \cap \eta^\perp)$ 
is the dual cone of  $\theta^{(j-1)} (\eta)$. 

Otherwise, there exists  a vector $\delta^{(j)}   \in   \Gamma^{(j)}_{\theta^{(j-1)}} \cap \eta^\perp $ such that 
$\delta^{(j)} \notin \vartheta$. Then, we have an expansion of the form:
\[
\delta^{(j)} = a_1  \g_1^{(j_1)} + \cdots + a_{d_0}  \g_{d_0}^{(j_1)} , \mbox{ with } a_i \in \Q, 
\]
and with some $a_i$ less than zero, say $a_1$. Since $\g_1^{(j_1)} = \g_1^{(j)}$  by (\ref{repre4}), we deduce from 
the definitions that $\delta^{(j+1)} := \delta_j - \g_1^{(j)}$ belongs to  $\Gamma^{(j+1)}_{\theta^{(j)}} \cap \eta^\perp$ 
and $\delta^{(j+1)} \notin \vartheta$. By induction we get that  
$\delta^{(j+k)} := \delta_1 - k \g_1^{(j)}$  belongs to  $\Gamma^{(j+k)}_{\theta^{(j+ k-1)}} \cap \eta^\perp$ 
and $\delta^{(j+k)} \notin \vartheta$, for $k >1$. It follows that $\langle w, \delta^{(j+k)} \rangle = \langle w, \delta_1 \rangle - k \langle w, \g_1^{(j)} \rangle$ 
becomes negative for $k \gg 0$. This cannot happen since $w $ belongs to the interior of $\theta^{(j)} (\eta)$, for all $j \gg 0$.
This contradiction  ends the
proof of Proposition \ref{orbitq}.
  \hfill $\ {\Box}$

The statement of the main result below is given in terms of certain orders on the lattice $M$.

\begin{definition} \label{po}
 A relation $\leq$ on the lattice $M$ is said to be a \textit{preorder} if it satisfies the following conditions:
\begin{itemize}
\item For any $m, n\in M$ one has either $m\leq n$ or $n\leq m$.
\item The inequalities $m\leq n$ and $n\leq p$ imply $m\leq p$.
\item If $m\leq n$ holds, then $m+p\leq n+p$ for all $p\in M$.
\end{itemize}
\end{definition}

\begin{definition}   \label{pre-def}
If $N$ is the dual lattice of $M$ and if $\underline{\nu} = (\nu_1, \dots, \nu_s)  \in (N_\R)^s$, for $ s \leq d = \mathrm{rk N}$, 
one can define a preorder $\leq_{\underline{\nu}}$ on $M$ by 
setting for $n, m \in M$ that
\begin{equation} \label{pre-order}
n \leq_{\underline{\nu}}  m 
\Leftrightarrow 
(\langle n, \nu_1\rangle, \dots, \langle n, \nu_s\rangle) \leq_{lex} (\langle m, \nu_1\rangle, \dots, \langle m, \nu_s \rangle),
\end{equation}
where $\leq_{lex}$ denotes the lexicographical order on $\R^s$.
If   $\underline{\nu} = ( \nu_1, \dots, \nu_d ) \in N^s$ we say that $\leq_{\underline{\nu}}$  
is a \emph{rational preorder}; if $s = d$ and $\underline{\nu}$ is an ordered basis of $N$ then 
the rational preorder $\leq_{\underline{\nu}}$ is an order,
we say in this case that $\leq_{\underline{\nu}}$ is a \emph{rational order}. 
\end{definition}
 If $\leq_{\underline{\nu}}$   is a rational order then the relation (\ref{pre-order}) 
 means that the set of coordinates of $n$ with respect to the dual basis of $\underline{\nu}$ 
 is less than or equal to, for the lexicographic order, the analogous set of coordinates of $m$. If $\nu_1 = \dots = \nu_s =0 $ then the preorder $\leq_{\underline{\nu}}$
 is trivial, that is  $m\leq_{\underline{\nu}}n\ \forall \ m,n\in M$.

\begin{remark} 
A basis $\underline{\nu}$ of $N$ defines similarly a valuation of the field  $k (t^M)$ of functions
of $k[t^M]$, of maximal rank $d$, with values in the totally ordered group $\Z^d$ equipped with the lexicographical
order. This valuation is defined as the monomial valuation whose value on a monomial
$t^\g$ is $ (\langle \nu_1, \g \rangle , \dots, \langle \nu_d,  \g \rangle )$.
\end{remark}

Let ${\underline{\nu}}$ be a basis of the lattice $N$ defining the rational order $\leq_{{\underline{\nu}}}$ on $M$. 
We denote by $M_{\geq _{\underline{\nu}} 0}$ the set $\{ m \in M \mid m \geq _{\underline{\nu}} 0 \}$.
 If $\Gamma \subset M_{\geq _{\underline{\nu}} 0}$
then we can use the order $\underline{\nu}$  to distinguish a $d$-dimensional cone $\t^{(j)}$ of the the fan $\Sigma^{(j)}$ defining
the $j$-th
iterated blowing up of the logarithmic jacobian ideal.
First, we relabel the minimal system of  generators $G = \{ \g_1, \dots, \g_{n} \}$  of $\Gamma := \Gamma^{(1)}$
in such a way that $\g_1 = \min_{\leq_{\underline{\nu}}} G$,
and then $\g_i = \min_{\leq_{\underline{\nu}}} \{  \g  \in G \mid \g_1 \y \dots \y \g_{i-1} \y \g \ne 0 \} $ for $i= 2, \dots, d$.
The semigroup $\Gamma^{(2)}$  is generated by:
\[ \{ \g_1, \dots, \g_d \} \cup \{ \g_l - \g_i
\mid 1 \leq i \leq d < l \leq n , \ \g_1 \y \dots \y \g_{i-1} \y \g_l \y \g_{i+1} \y \dots \y \g_d \ne 0 \}. \]
The cone $\check{\t}^{(1)}$ spanned by $\Gamma^{(2)}$ is strictly convex of dimension $d$ and its dual cone
$\t^{(1)}$ belongs to $\Sigma^{(1)}$. With the notations of the begining of Section \ref{iteration}
this means that $\Gamma^{(2)} = \Gamma^{(2)}_{\t^{(1)}}$
(see Lemma \ref{repre-gen}).
By construction   the semigroup $\Gamma^{(2)}$ is contained in
$M_{\geq _{\underline{\nu}} 0}$,   hence we can repeat this procedure replacing $G$ by the minimal set of generators $G^{(1)} = \{ \g_1^{(1)}, \dots, \g^{(1)}_{n^{(1)}} \}$ of the semigroup
$\Gamma^{(2)}$. By repeating this procedure inductively, the order $\underline{\nu}$ defines a nested sequence cones $(\t^{(j)})$
of the form
(\ref{sequence}) and such that $\Gamma^{(j)} =  \Gamma^{(j)}_{\t^{(j-1)}}$ for all $j \geq 1$.

\begin{theorem} \label{lu} If $\underline{\nu}$  is a basis of $N$ defining a rational order $\leq_{\underline \nu}$ such that
$\Gamma \subset M_{\geq _{\underline{\nu}} 0}$ then:
\begin{enumerate}
\item[{\rm (i)}] For any $j \geq 0$ there
exists a unique $d$-dimensional cone  $\t^{(j)} \in \Sigma^{(j)} $ such that
$\Gamma^{(j+1)}_{\t^{(j)}} \subset  M_{\geq _{\underline{\nu}} 0}$.
\item[{\rm (ii)}]There exists an integer $l_0$, depending on $\underline{\nu}$,
such that $\t^{(j)} = \t^{(l_0)}$ for $j \geq l_0 -1$ and the semigroup
$\Gamma^{(l_0)}_{\t^{(l_0 -1)}}$ is regular.
\end{enumerate}
\end{theorem}
\textit{Proof}. The first assertion is consequence of definitions and the previous discussions.

Suppose that the second assertion of the Theorem does not hold.
Then the sequence of cones $(\t^{(j)})$ does not stabilize by Proposition \ref{stable}. By construction
we have that the vector $\nu_1$ belongs to $\t^{(j)} \cap N$ for all $j \geq 0$. By Proposition \ref{prev}
there exists an integer $j_0 \geq 0$ and a cone $\eta$  such that $ \nu_1 \in \mathrm{int}({\eta}) $ and $\eta \leq \t^{(j)}$ for $j \geq j_0$. The cone $\eta$ is stable.
If the sequence $(\t^{(j)})$ does not stabilize then there exists an integer $j_1 \geq j_0$ and a stable face $\theta$ such that
$\eta \leq \theta \leq \t^{(j)}$  for $j \geq j_1$,  $\codim \, \eta = d_0  > 0$ and any face $\vartheta^{(j)} \leq \t^{(j)}$ such that
$\theta \lneq \vartheta^{(j)}$ is not stable.

If $\nu \in N$ we denote by $\bar{\nu}$ the canonical image of $\nu$ in $N(\theta)$.
Then there exists an integer $1 < s \leq d_0 + 1$ such that $\bar{\nu}_i =0$ for $i =1, \dots, s-1$ and
$\bar{\nu}_s \ne 0$.  By definition of the order $\leq_{\underline{\nu}}$ the linear form $\bar{\nu}_s$ is non negative on
the semigroup $\Gamma^{(j)}_{\t^{(j-1)}} \cap \theta^{\perp} \subset \Gamma^{(j)}_{\t^{(j-1)}}$, that is,
the vector $\bar{\nu}_s $  is a non-zero element of  $\t^{(j-1)} (\theta) \cap N (\theta)$
for $j > j_1$.

The vector $\bar{\nu}_s $ belongs to the relative interior of a unique face $\z^{(j)} (\theta)$ 
of $\t^{(j)} (\theta)$ for $j \geq j_1$.
 By Proposition \ref{orbitq} there exists an integer $j_2 \geq j_1$ such that
$\z^{(j)} (\theta) = \z^{(j_2)} (\theta)$ for $j \geq j_2$.  We denote  by
$\z^{(\infty)} (\theta)$ the cone  $ \z^{(j_2)} (\theta)$ and by
 $d_1 > 0$ its dimension.

For any $j \geq j_2$ we fix a maximal chain of faces
\[
0 = \z_0 (\theta)  \leq \z_1 (\theta)  \leq \dots \leq \z_{d_1} (\theta) = \z^{(\infty)} (\theta),
\]
with $\dim \z_l (\theta) = l $ for $l= 0, \dots, d_1$. Then we get a chain of faces of $\t^{(j)}$ containing
$\theta$ of the form:
\[
\theta = \z_0^{(j)} \leq \z_1 ^{(j)}  \leq \dots \leq \z_{d_1} ^{(j)},
\]
where $\z_l^{(j)}$ is a face of codimension one of $\z_{l+1}^{(j)}$ for $l=0, \dots, d_1 -1$.
This implies that the sequence $( \z_{d_1} ^{(j)} )$ is distinguished. By Proposition \ref{fin} it follows that this sequence
stabilizes in a cone, which contains $\theta$ as a proper face. This contradiction ends the proof of Theorem \ref{lu}.
 \hfill $\ {\Box}$

\section{The Zariski-Riemann space of a fan, according to Ewald-Ishida (\cite{Ewald-Ishida})}   \label{EI}
In this section we recall some properties of a space which plays for proper birational toric maps of toric varieties the role played for proper birational maps of algebraic varieties by the Zariski-Riemann manifold.

As in \textit{loc.cit.}, we denote by $ZR(M)$ the set of all preorders on $M$ (see Definition \ref{po}).
Given $w\in ZR(M)$ we note by $m\leq_w n$ the corresponding preorder and by $L(w)$
the subsemigroup $\{m\in M\vert m\geq_w 0\}$. The set $L^0(w)=L(w)\cap (-L(w))$ of the elements that are equivalent
to zero for the equivalence relation $m=_w n$ deduced from $\leq_w$ is a saturated sublattice of $M$.\par
By construction, the preorder $\leq_w$ induces a total order on the free abelian group $M/L^0(w)$ and the
canonical quotient map $\lambda(w)\colon M\to M/L^0(w)$ determines the preorder $w$ from this order, so we see that every preorder on $M$
is given by a total order on a free quotient of $M$.
We denote by $o(M)$ the trivial preorder corresponding to the case where the free quotient is the zero group.\par
The topology on $ZR(M)$ is defined by a basis of open sets
$$ \bigl(\mathcal{U}(\theta)=\{w\in ZR(M)\vert \check \theta \cap M \subset L(w)\}\bigr)_\theta$$ where
$\theta$ runs through the set of rational polyhedral cones in $N_\R$. 
If $\Sigma$ is a rational fan in $N_\R$ the Zariski-Riemann manifold of $\Sigma$
is defined to be
$$ZR(\Sigma)=\bigcup_{\sigma\in \Sigma}\mathcal{U}(\sigma).$$
By \cite{Ewald-Ishida}, Proposition 2.10, it depends only on the support $\vert\Sigma\vert$ of $\Sigma$,
and by Theorem 2.4 of \textit{loc.cit.}, it is quasi-compact.\par
For each $w\in ZR(M)$, since $M/L^0(w)$ is a totally ordered group of finite rational rank,
it is of finite (real) rank (or height); there is a maximal sequence of distinct convex subgroups
$$(0)=\Psi_s(w)\subset \Psi_{s-1}(w)\subset\cdots\subset \Psi_1(w)\subset M/L^0(w)$$
such that the order on $M/L^0(w)$ induces on each of its quotients by the $\Psi_i (w)$ an order such that the quotient map is monotonous.
The totally ordered quotients $M_i=M/\lambda(w)^{-1}(\Psi_i)$ corresponding to these subgroups define via the surjections $\lambda_i\colon M\to M_i$ a sequence
$$w=w_0,w_1,\ldots, w_{s-1}, w_s=o(M)$$ of preorders on $M$, with which $w$ is said to be
\textit{composed}, as valuations are composed. The integer $s$ can be called the (real) \textit{rank} or the \textit{height} of the preorder. To them corresponds a sequence of distinct subsemigroups of $M$:
$$L(w)\subset L(w_1) \subset \cdots \subset L(w_{s-1})\subset L(w_s) = M.$$
and a sequence of distinct subgroups of $M$:
$$L^0(w)\subset L^0(w_1) \subset \cdots \subset L^0(w_{s-1})\subset L^0(w_s)=M.$$
Compare with \S 3.1 of \cite{T-valuations}. The semigroups $L(w_i)$ are the analogues in this situation of the valuation rings of a field containing a given valuation ring.
It is also shown in \cite{Ewald-Ishida}, Lemma 5.1, that for each nontrivial preorder there
is an element $y_0\in N_\R$ such that $L(w)\subset \{m\in M\vert \langle m,y_0\rangle\geq 0\}$,
which is unique up to multiplication by a positive real number.
In fact $\{x\in M_\R \vert \langle x,y_0\rangle\geq 0\}$ is the closure in
$M_\R$ of the convex closure of $L(w)$. To all preorders with which $w$ is composed correspond the same $y_0$.
It follows from what we have seen that, with the notations introduced above, $L^0(w_{s-1})$ is the intersection
with the lattice $M$ of $\{x\in M_\R\vert  \langle x,y_0\rangle = 0\}$.\par\noindent
Each quotient $\tilde{S}_i=L^0(w_i)/L^0(w_{i-1})$ ($1\leq i\leq s$) is a totally ordered lattice of real rank one whose order
is given by an embedding into $\R$. The images in $\R$ of the basis vectors of this lattice determine in $T_{i\R}$, where $T_i$ is the dual lattice of $\tilde{S}_i$, a weight vector which can be represented by an element $y_i\in N_\R$ as in \textit{loc.cit.}
The preorder $w$ on $M$ appears as the composition with $\lambda(w)$ of the lexicographic product of the rank one orders on the quotients
$\bigl(L^0(w_i)/L^0(w_{i-1})\bigr)_{1\leq i\leq s}$. Thus, for each preorder $w$ there exist elements
 $y_0,\ldots , y_{s-1}$ in $N_\R$ such that $m\geq_w 0$ if and only if the sequence
$( \langle y_0, m \rangle, \ldots , \langle y_{s-1}, m \rangle)$ is
$\geq 0$ for the lexicographic order. 
Each element $y_i$ is
defined up to positive homothety and modulo the $\R$-vector subspace of $N_\R$ generated by the previous ones. In particular, if $w$ is a rank
$d$ preorder it is a rational order with respect to
some ordered basis of $M$.

\begin{remark} \label{pre-form}
As a consequence of the discussion above, any preorder $w$ of $M$
is of the form $\leq_{\underline{ y}}$ for $\underline{y} = (y_0, \dots, y_{s-1}) \in N_\R^d$ 
and any rational preorder $w$  of $M$ is of the form $\leq_{\underline{y}}$ 
for $\underline{y} = (y_0, \dots, y_{s-1} )  \in N^s$  with $y_0, \dots, y_{s-1}$ 
linearly independent over $\R$ (compare with Definition \ref{pre-def}).
By induction on the height one can show that any non-trivial preorder is of the form
$\leq_{\underline{ y}}$ for $\underline{y'} = (y_0', \dots, y_{s-1}')$ and 
$y_0', \dots, y_{s-1}'$ vectors which are part of a basis of $N$. 
\end{remark}

Note that we can also, just as for valuations, define the
\textit{rational rank} of a preorder $w$ on $M$: it is the (rational)
rank of the free abelian group $M/L^0(w)$. The analogue of the
residual transcendence degree is then the rank of $L^0(w)$, so
that for preorders on $M$, the analogue of Abhyankar's equality
${\rm rat.rank}w+ {\rm rank}L^0(w)={\rm rank}M$ always
holds.\par\noindent A preorder which is an order, which means
$L^0(w)=0$ is the analogue of a valuation with algebraic residue field
extension, also called zero dimensional valuations (their centers are zero dimensional). 
If we are speaking of valuations dominating a local ring with algebraically closed residue field, 
they are the valuations with trivial residue field extension, which the second author calls \textit{rational valuations}
 in \cite{T-valuations}. Note also that since the preorders corresponding to vectors of $N$ are dense in $ZR(M)$ 
for the topology defined by Ewald-Ishida, while the set of orders is closed: if a preorder $w$ is not an order, 
one choose a rational cone $\theta$ such that $\check\theta\cap M\subset L(w)$ and $\check\theta\cap L^0(w)$ 
is a non-zero lattice. Then no preorder in the open set $\mathcal{U}(\theta)$, which contains $w$, can be an order.\par\noindent
It is shown in \cite{Si} that the closed subset of $ZR(M)$ consisting of orders on $M$, with the induced topology, is homeomorphic to a Cantor set when the rank of the lattice $M$ is $\geq 2$.\par\medskip\noindent
To summarize:
\begin{itemize}
\item All preorders are "Abhyankar".
\item A preorder given by a weight vector $w\in N$ is the analogue of a divisorial valuation on a noetherian local ring. Such preorders are dense.
\item A preorder which is an order is the analogue of a zero dimensional valuation. The set of orders is closed in $ZR(M)$ and homeomorphic to a Cantor set.
\item A preorder of maximal (real) rank is a rational order with respect to some basis of $N$ and is the analogue of a valuation of real rank equal to the dimension of the ring.
\end{itemize}
\par\medskip
Ewald and Ishida define the fact that a preorder \textit{dominates} a cone $\theta\subset N_\R$ by
the following two properties: $\check \theta \cap M \subset L(w)$ and $ \check \theta  \cap L^0(w)=
\theta^\perp\cap M$.
They show that given a fan $\Sigma$, a given preorder dominates at most one cone of $\Sigma$ (\cite{Ewald-Ishida},
Proposition 2.6)
 and that one has the inclusion $ \check \theta \cap M\subset L(w)$ if and only if
$w$ dominates a face of $\theta$ (\cite{Ewald-Ishida}, Lemma 2.7), which is then unique.\par
 Denoting by ${\rm dom}(\sigma)$ the set of elements $w\in ZR(M)$ dominating a given strictly convex
rational polyhedral cone $\sigma\subset N_\R$, it is non empty and one has moreover,
$F(\sigma)$ being the fan consisting of the faces of $\s$:
 $$ZR (F(\sigma))=\bigcup_{\zeta\in F(\sigma)}{\rm dom}(\zeta).$$
 If $\s$ is a strictly convex cone and $w\in ZR (F(\sigma))$, in each refinement $\Sigma^{(i)}$ of $F(\s)$ there is exactly one cone which is dominated by $w$.\par
 \begin{definition} Given an affine toric variety $T_{F(\s)}^\Gamma$ as above and a preorder $w$ on $M$ such that $\check\s\cap M\subset L(w)$, the \textit{semigroup} of $w$ on $k[t^\Gamma]$ is the image of $\Gamma$ in the quotient $M/L^0(w)$; it is a semigroup contained in the positive part of $M/L^0(w)$.\end{definition} If $w$ dominates $\s$ this semigroup is the same as the image of $\Gamma$ in the quotient $M/M(\s)$ which was used in the proof of Lemma \ref{qzero}.
 
 \section{Interpretation of the main result with the Zariski-Riemann Manifold of a fan}    \label{EI2}
Now let $\s$ be a strictly convex rational cone in $N_\R$ and consider the sequence $\Sigma^{(i)}$ of fans obtained by iterating the blowing-up of logarithmic jacobian ideals, with $\Sigma^{(0)}=F(\s)$. For each $w\in ZR(\Sigma^{(i)})=ZR(\Sigma^{(0)})$, and each $i, 0\leq i<{\rm rank}(w)$, the cones dominated by the partial orders $w_j$ with which $w$ is composed form a sequence of distinct cones of $\Sigma^{(i)}$
\begin{equation} \label{faces2}\zeta^{(i)}_{s-1}\subset \zeta^{(i)}_{s-2}\subset\cdots \subset \zeta^{(i)}_0\end{equation}
and since $\Sigma^{(i+1)}$ is a refinement of $\Sigma^{(i)}$, we have for each $j,\ 0\leq j\leq s-1$
the inclusion $\zeta^{(i+1)}_j\subset \zeta^{(i)}_j$.\par
A cone $\zeta\in \Sigma^{(i)}$ is dominated by $w_{s-1}$ if and only if $y_0$ is in the relative interior of $\zeta$ (see also \cite{Ewald-Ishida}, the four lines before 2.7, which also show that ${\rm dom}(\s)$ is not empty).\par
On the other hand, if $w$ is in fact an order, which means that $L^0(w)=0$, then a rational cone $\tau$ dominated by $w$ must satisfy $\check\tau\subset L(w)$ and ${\rm dim}\tau={\rm rank}M=d$ since we must have $\tau^\perp=(0)$; the rank of $w$ must be $\leq d$.

\begin{definition}Let us keep the notations just introduced. Given a preorder $w\in ZR (F(\sigma))$, we say that the sequence of fans $\Sigma^{(j)}$, or the sequence of toric varieties $T_{\Sigma^{(i)}}^{\Gamma^{(i)}}$, \textit{stabilizes at the preorder} $w$ if
there exists an integer
$i\geq 1$ such that the cone $\tau^{(j)}$  of $\Sigma^{(j)}$ dominated by $w$ is regular and the corresponding semigroup
is also regular, in particular, it  coincides with $\check\tau^{(j)}\cap M$.
\end{definition}
\begin{proposition}
\label{pre-eq}
Given the sequence of fans $\Sigma^{(j)}$ as above, the following statements are equivalent:
\begin{enumerate}
\item[{\rm (1)}]  There exists an integer $i$ such that the toric variety $T_{\Sigma^{(j)}}^{\Gamma^{(j)}}$ is regular.
\item[{\rm (2)}]  For every preorder $w$ of $ZR (F(\sigma))$ the sequence $\Sigma^{(j)}$ stabilizes at $w$.
\end{enumerate}
\end{proposition}
\begin{proof} We only have to prove that ${\rm (2)}$ implies ${\rm (1)}$. By construction if the sequence stabilizes at
 $w$ and $w$ dominates the regular simplex $\tau^{(j)}$ it also stabilizes for the same index for all preorders
 $w'\in\mathcal{U}(\tau^{(j)})$ since the faces of $\tau^{(j)}$ as well as the corresponding semigroups are regular.
This gives us an open covering of $ZR(\Sigma^{(0)})=ZR (F(\sigma))$. Since it is quasi-compact (see \cite{Ewald-Ishida}, Theorem 2.4)
one can extract a finite covering. Taking the maximum $i$ of the indices of stabilization corresponding to
those open sets gives us a fan $\Sigma^{(i)}$ covered by regular cones, and so ${\rm (1)}$.
\end{proof}
We can now reformulate the main result as follows, with the usual notations:
\begin{theorem} Given an affine toric variety $T^\Gamma$ and a rational order $w\in ZR(F(\s))$ with
respect to some ordered basis of $N$, the sequence $\Sigma^{(i)}$ of the Semple-Nash refinements of the cone
$\sigma$ stabilizes for $w$.
\end{theorem}
\begin{proof} Use the construction at the beginning of this section and Theorem \ref{lu}.
\end{proof}

\begin{corollary} 
For every rational preorder $w$ of $ZR (F(\sigma))$ the sequence $\Sigma^{(j)}$ stabilizes at $w$.
\end{corollary}
\begin{proof} By Remark \ref{pre-form},  
the rational preorder $w$ is of the form  $\leq_{\underline{\nu}}$, 
where $\underline{\nu} = (\nu_1, \dots, \nu_s)$ 
for $s \leq d$ and $\nu_1, \dots, \nu_s \in N $ are vectors which belong to a basis of $N$.  
Then we can complete this sequence to a basis of $N$ and consider the associated rational order $\bar{w}$. 
It follows that the sequence $\Sigma^{(j)}$ associated to $\bar{w}$ stabilizes at every
preorder with which such an order is composed, in particular with $w$.
\end{proof}
\begin{remark}
Any
preorder given by an integral vector $0 \ne \nu\in N$ is also
represented by a primitive vector which generates a direct factor
of $M$ so that there exist rational orders which are composed
with it. In this case this corollary is the analogue of Hironaka's result in \cite{H} 
since preorders corresponding to integral vectors are the analogues of divisorial valuations.
\end{remark}

\begin{remark}
As a conclusion, we prove the desired result for the restricted class of rational preorders  of $ZR(F(\s))$,
 while in order to prove resolution of
 toric varieties by blowing up logarithmic
jacobian ideals one should prove a similar result for all
preorders  in $ZR(F(\s))$.  
\end{remark}
We end the paper with the verification of the analogue in our situation of Zariski's construction of the Zariski-Riemann space in \cite{Z}.\par
\begin{lemma}\label{injj} Let $w,w'\in ZR(F(\s))$ be different preorders.
There exists a refinement $\Sigma'$ of  $F(\s)$ such that the two cones of $\Sigma'$
dominated respectively by $w$ and by $w'$ are distinct.
\end{lemma}
\begin{proof} Since the trivial preorder can dominate only the cone $\{0\}$ in a fan, we may assume that neither
preorder is trivial. To $w$ and $w'$ let us attach sequences $y_0,\ldots , y_{s-1}$ and $y'_0,\ldots , y'_{s'-1}$
of elements of $N_\R$ as in the previous section. Let us denote by $M^i(w)$ the intersection of $M$ with the subspace
$( \langle y_0, m \rangle= \cdots = \langle y_{i-1} , m \rangle)=0$, with the convention $M^0(w)=M$. It is proved in
\cite{Ewald-Ishida} that $w$ dominates a rational convex cone $\tau\subset N_\R$ if and only if
for $i=0,\ldots, s-1$ we have $y_i\in \tau+(M^i(w))^\perp$, and $\check\tau\cap M^s(w)=M\cap\tau^\perp$.\par\noindent
The quotients $ \tau+(M^i(w))^\perp/(M^i(w))^\perp$ are the dual cones of $\check\tau\cap M^i(w)_\R$. If $y_0$ and
$y'_0$ are not homothetic there is certainly a refinement of the fan $F(\s)$ in which they are contained in different
cones. This also settles the case where $M$ is of rank one. Otherwise we have $M^1(w)=M^1(w')$ and the preorders
induced by $w$ and $w'$ on $M^1(w)$ are different. We can now argue by induction on the rank of $M$ since a refinement
of the fan $ \Sigma+(M^1(w))^\perp/(M^1(w))^\perp$ will induce a refinement of $\Sigma$.
\end{proof}

Let us denote by $\{T_{\Sigma}^{\Gamma}\}$ the set of irreducible $T^M(k)$-invariant subvarieties of
$T_{\Sigma}^{\Gamma}$ endowed with the topology induced by the Zariski topology.
Let us consider all possible 
equivariant blowing-ups $T_{\Sigma'}^{\Gamma'}\to T_\Sigma^\Gamma$. 
They form a projective system of $T_\Sigma^\Gamma$-schemes and the $\{T_{\Sigma'}^{\Gamma'}\}$ also  form a
 projective system of topological spaces.\par\noindent
Given a preorder $w\in ZR(\Sigma)$ it dominates a cone $\tau'$ in each refinement $\Sigma'$ of $\Sigma$
and defines an irreducible $T^M(k)$-invariant subvariety corresponding to the prime ideal generated by
the monomials of $k[t^{\Gamma'_{\tau}}]$ whose exponents are $>_w 0$. This defines a map
$$Z\colon ZR(\Sigma)\to \limproj\{T_{\Sigma'}^{\Gamma'}\}.$$
\begin{proposition} The map $Z$ is an homeomorphism.
\end{proposition}
\begin{proof} The map is injective by Lemma \ref{injj}. By the construction of the blowing-ups of toric varieties an element
 $e$ of the projective limit determines in each $T_{\Sigma'}^{\Gamma'}$ a cone $\tau'$ and a semigroup
$\Gamma'\subset \check\tau'\cap M$. Let us try to define a preorder on $M$ as follows:
$m\geq_e 0$ if $m \in \Gamma'$ for some $\Gamma'$ picked by the element $e$. We have to prove that given $m, n\in M$
either $m-n\geq_e0$ or $n-m\geq_e0$. Since the $\Gamma'$ generate $M$ we may assume that
$m, n \in \Gamma'$ for some $\Gamma'$. Consider the ideal generated by $t^m,t^n$ and let
$T_{\Sigma''}^{\Gamma''}\to T_{\Sigma'}^{\Gamma'}$ be its blowing-up.
The effect on the fan $\Sigma'$ divides the cone $\tau'$ in two by the hyperplane dual to the vector $m-n$.
The element $e$ picks one of the two charts, say $\tau''$, and either $m-n$ or $n-m$ is in the corresponding $\Gamma''$,
which proves the desired result. The fact that the semigroups $\Gamma'$ form an inductive system suffices to show the
transitivity of the preorder.\par\noindent
The fact that the map just defined is the inverse of the map $Z$ and determines an homeomorphism follows directly
from the definitions. The trivial preorder corresponds to $T^M \subset T^{\Gamma'}_{\Sigma'}$.
\end{proof}

\bibliographystyle{amsplain}

\def\cprime{$'$}
\providecommand{\bysame}{\leavevmode\hbox
to3em{\hrulefill}\thinspace}


\providecommand{\bysame}{\leavevmode\hbox to3em{\hrulefill}\thinspace}
\providecommand{\MR}{\relax\ifhmode\unskip\space\fi MR }
\providecommand{\MRhref}[2]{%
  \href{http://www.ams.org/mathscinet-getitem?mr=#1}{#2}
}
\providecommand{\href}[2]{#2}
\begin{thebibliography}{}

\end{thebibliography}


\begin{thebibliography}{10}

\bibitem{ALPPT}
A.~Atanasov, C.~Lopez, A.~Perry, N.~Proudfoot, and M.~Thaddeus, \emph{Resolving
  toric varieties with Nash blow-ups}, 
  Experimental Math. \textbf{20}  (2011), no.~3, 288--303.
  

\bibitem{B}
N.~Bourbaki, \emph{\'{E}l\'ements de math\'ematique. {VII}. {P}remi\`ere
  partie: {L}es structures fondamentales de l'analyse. {L}ivre {II}:
  {A}lg\`ebre. {C}hapitre {III}: {A}lg\`ebre multilin\'eaire}, Actualit\'es
  Sci. Ind., no. 1044, Hermann et Cie., Paris, 1948.



  \bibitem{CCD} E.~ Cattani, R.~ Curran and A.~Dickenstein, \emph{Complete intersections in toric ideals},
 Proc. Amer. Math. Soc. \textbf{135} (2007), 329--335.



   \bibitem{CHWW} G.~ Corti\~nas, C.~ Haesemeyer, M.~E.~  Walker and C.~A.~ Weibel
   \emph{Toric varieties, monoid schemes and $cdh$ descent}, ArXiv math.KT/11061389,  
   (2011)
 
\bibitem{CLS}
David~A. Cox, John~B. Little, and Henry~K. Schenck, \emph{Toric varieties},
  Graduate Studies in Mathematics, vol. 124, American Mathematical Society,
  Providence, RI, 2011.
  
  
  \bibitem{Da}
V.~I. Danilov, \emph{The geometry of toric varieties}, Uspekhi Mat. Nauk
  \textbf{33} (1978), no.~2(200), 85--134, 247.
  
  \bibitem{Du}
D.~ Duarte, \emph{Nash modification on toric surfaces}, 
Revista de la Real Academia de Ciencias Exactas, F\' \i sicas y Naturales. Serie A. Matem\'aticas, 
first published on line DOI: 10.1007/s13398-012-0104-4. 

\bibitem{E-S}
D.~Eisenbud and B.~Sturmfels, \emph{Binomial ideals}, Duke Math. J. \textbf{84}
  (1996), no.~1, 1--45.

\bibitem{Ewald}
G.~Ewald, \emph{Combinatorial convexity and algebraic geometry}, Graduate Texts
  in Mathematics, vol. 168, Springer-Verlag, New York, 1996.

\bibitem{Ewald-Ishida}
G.~Ewald and M-N Ishida, \emph{Completion of real fans and Zariski-Riemann spaces},
 Toh\^oku Math. J., (2), 58, (2006), 189-218.

\bibitem{Fulton}
W.~Fulton, \emph{Introduction to toric varieties}, Annals of Mathematics
  Studies, vol. 131, Princeton University Press, Princeton, NJ, 1993, The
  William H. Roever Lectures in Geometry.

  \bibitem{GKZ}
I.~M. Gel{\cprime}fand, M.~M. Kapranov, and A.~V. Zelevinsky,
  \emph{Discriminants, resultants, and multidimensional determinants},
  Mathematics: Theory \& Applications, Birkh\"auser Boston Inc., Boston, MA,
  1994.
{\Large }
  \bibitem{GP-T}
P.~D. Gonz\'alez~P\'erez and B.~Teissier, \emph{Embedded resolutions of not
  necessarily normal affine toric varieties}, C.R. Math. Acad. Sci. Paris
  \textbf{334} (2002), no.~5, 379--382.
  
 \bibitem{GS-ENS}
G.~Gonzalez-Sprinberg, \emph{\'Eventails en dimension 2 et transform\'e de
  Nash}, Publications du Centre de Math\'ematiques de l'E.N.S., 45, rue d'Ulm,
  Paris 5e (1977), 1--68.

\bibitem{GS-CRAS}
\bysame, \emph{Transform\'e de {N}ash et \'eventail de dimension {$2$}}, C. R.
  Acad. Sci. Paris S\'er. A-B \textbf{284} (1977), no.~1, A69--A71.

  \bibitem{GS}
\bysame, \emph{R\'esolution de {N}ash des points doubles
  rationnels}, Ann. Inst. Fourier (Grenoble) \textbf{32} (1982), no.~2, x,
  111--178.
  
 \bibitem{G-M}
D.~Grigoriev and P.~ Milman, \emph{Nash desingularization for binomial varieties
as Euclidean division, a priori termination bound, polynomial complexity in dim $2$},  
Adv. Math. 231 (2012), no.~6, 3389--3428.

\bibitem{Ha}
R.~Hartshorne, \emph{Algebraic geometry}, Springer-Verlag, New York, 1977,
  Graduate Texts in Mathematics, No. 52.

\bibitem{H}
H.~Hironaka, \emph{On {N}ash blowing-up}, Arithmetic and geometry, {V}ol. {II},
  Progr. Math., vol.~36, Birkh\"auser Boston, Mass., 1983, 103--111.

  \bibitem{K-K}
K.~Kaveh and A.~G. Khovanskii, \emph{
Newton-Okounkov bodies, semigroups of integral points, graded algebras and intersection theory},
Annals of Math. \textbf{176} (2012), no.~2, 925--978.

\bibitem{TE}
G.~Kempf, F.~F. Knudsen, D.~Mumford, and B.~Saint-Donat, \emph{Toroidal
  embeddings. {I}}, Lecture Notes in Mathematics, Vol. 339, Springer-Verlag,
  Berlin, 1973.
  
\bibitem{Ku}
O.~Kuroda, \emph{The infiniteness of the SAGBI bases for certain invariant rings}, Osaka J. Math.,\textbf{39}, (2002), 665-680.

\bibitem{LJ-R}
M.~Lejeune-Jalabert and A.~J. Reguera, \emph{The {D}enef-{L}oeser series for
  toric surface singularities}, Proceedings of the {I}nternational {C}onference
  on {A}lgebraic {G}eometry and {S}ingularities ({S}panish) ({S}evilla,
  2001), Rev. Mat. Iberoamericana
  vol.~19, 2003, 581--612.

\bibitem{No}
A.~Nobile, \emph{Some properties of the {N}ash blowing-up}, Pacific J. Math.
  \textbf{60} (1975), no.~1, 297--305.

  \bibitem{Oda-M}
T.~Oda, \emph{Torus embeddings and applications}, Tata Institute of Fundamental
  Research Lectures on Mathematics and Physics, vol.~57, Tata Institute of
  Fundamental Research, Bombay, 1978, Based on joint work with Katsuya Miyake.

\bibitem{Oda}
\bysame, \emph{Convex bodies and algebraic geometry}, Ergebnisse der Mathematik
  und ihrer Grenzgebiete (3), vol.~15, Springer-Verlag, Berlin, 1988.

\bibitem{S}
J.~G. Semple, \emph{Some investigations in the geometry of curve and surface
  elements}, Proc. London Math. Soc. (3) \textbf{4} (1954), 24--49.
  
\bibitem{Si} A.~Sikora, \emph{Topology on the spaces of orderings of groups}, Bull. London Math. Soc. 36 (2004), no. 4, 519–526.

\bibitem{Spivakovsky}
M.~Spivakovsky, \emph{Sandwiched singularities and desingularization of
  surfaces by normalized {N}ash transformations}, Ann. of Math. (2)
  \textbf{131} (1990), no.~3, 411--491.
  
  \bibitem{St}
J.~ Steiner \emph{ \"Uber eine besondere Curve dritter Classe (und vierten Grades)}, Burchardt's Journal Band LIII,
231-237 (Gelesen in der Akademie der Wissenschaften zu Berlin am 7. Januar 1856).
Republished in:
Steiner, J. (1882). Gesammelte Werke, Band II (Herausgegeben von K. Weierstrass) (Berlin, G. Reimer, 1882),
639-647.

\bibitem{Sturmfels}
B.~Sturmfels, \emph{Gr\"obner bases and convex polytopes}, University Lecture
  Series, vol.~8, American Mathematical Society, Providence, RI, 1996.

\bibitem{Sumihiro-I}
H.~Sumihiro, \emph{Equivariant completion}, J. Math. Kyoto Univ. \textbf{14}
  (1974), 1--28.

\bibitem{T-hunting}
B.~Teissier, \emph{The hunting of invariants in the geometry of discriminants},
  Real and complex singularities, Oslo 1976, Sijthoff and Noordhoff
  International publishers, Alphen aan den Rijn, 1977, 565--677.

\bibitem{T-valuations}
\bysame, \emph{Valuations, deformations, and toric geometry}, Valuation theory
  and its applications, {V}ol. {II} ({S}askatoon, {SK}, 1999), Fields Inst.
  Commun., vol.~33, Amer. Math. Soc., Providence, RI, 2003, 361--459.
  
  \bibitem{Thom}
  R.~Thom, \emph{Sur la th\'eorie des enveloppes}, J. Math. Pures Appl. (9), \textbf{41}, 1962, 177-192.

   \bibitem{Thompson1} H. M  Thompson, \emph{Fan is to monoid as scheme is to ring:
   a generalization of the notion of a fan}, ArXiv math.AG/0306221.

   \bibitem{Thompson2}
   \bysame, \emph{Comments on toric varieties}, ArXiv  math.AG/0310336.

   \bibitem{Thompson3}
   \bysame, \emph{Toric singularities revisited}, J. Algebra \textbf{299} (2006), no. 2, 503--534.

   \bibitem{Wh}
  H.~Whitney, \emph{Tangents to an analytic variety},  Ann. of Math. (2)  \textbf{81} (1965), 496--549.

   \bibitem{Z}
  O.~Zariski, \emph{The compactness of the Riemann manifold of an abstract field of algebraic functions},  
  Ann. of Math., \textbf{41} (1940), 852--892.


\end{thebibliography}

\end{document}